\documentclass[11pt]{article}
\usepackage{latexsym}
\usepackage{amssymb}
\usepackage{graphicx}
\usepackage{graphicx,epstopdf}
\usepackage{amsmath}
\usepackage{epsfig}
\usepackage{multicol}
\usepackage{mathrsfs}
\usepackage{amsthm}
\usepackage{dsfont}
\usepackage{xcolor}
\usepackage{titlesec}
\usepackage{lipsum}
\usepackage{relsize}
 
\newtheorem{theorem}{Theorem}[section]
\newtheorem{lemma}[theorem]{\bf{Lemma}}
\newtheorem{prop}[theorem]{\bf{Proposition}}
\newtheorem{coro}[theorem]{\bf{Corollary}}
\newtheorem{defn}[theorem]{\bf{Definition}}
\newtheorem{remark}[theorem]{\bf{Remark}}

\newenvironment{pf}{{\noindent \bf Proof.\/}}{\hfill$\blacksquare$}

\newtheorem*{theorem*}{Theorem}

\begin{document}

\title{ The Dimension of Integral Self-Affine Sets via Fractal Perturbations: The Box  and the Hausdorff Dimensions,  Ergodic Measures
}

\author{Ibrahim Kirat \\
\footnotesize{\emph{Department of Mathematics, Istanbul Technical University, 34469 Maslak, Istanbul, Turkey}} \\
\emph{\footnotesize{E-mail: kirat@itu.edu.tr; ibkst@yahoo.com}} }

\date { }
\maketitle

\begin{abstract} Note by the author: 
Section 9.3 is added from the more general unpublished manuscript ``A Perturbation Method Leading to Full-Dimension Ergodic Measures on Integral Self-Affine Sets'', (2021) by I. Kırat.  

Original abstract:
An integral self-affine set $F=F(T,A)\subseteq \mathbb{R}^n$ is a self-affine set
which is generated by an $n\times n$  integer expanding matrix $T$ (not necessarily a similitude) and a finite set $A\subset \mathbb{Z}^n$ of
integer vectors so that $F=T^{-1}(F+A)$. The dimension problem of $F$ has not yet been settled fully.
For that, we introduce a fractal perturbation
 method with respect to $T,A$ and  get the dimension as the limit of the dimensions of a sequence of better-behaved perturbed fractals, for which a dimension formula already exists. An unexpected feature of this technique is that  the overlap structures of $F$ and its perturbations are eventually the same (i.e. the neighbor graphs are isomorphic), which is unlike some known  perturbations.  
 Our method has been developed especially for the problematic case of irreducible characteristic polynomial of $T$.
Also, we do not impose any separation condition on $F$ (like the open set condition) or any further restriction (such as size, etc.) on $T$ or $A$. 

As a by-product of the perturbation method, we prove the existence of the box dimension of $F$ too. 
Further, we
consider $F$ as a $T$-invariant subset of the n-torus (i.e, we consider $F \ \rm{mod  \ 1}$), and we rather use the perturbation method to show that there is an ergodic $T$-invariant Borel
probability measure on $F \ \rm{mod  \ 1}$ of full dimension. In contrast to some known results, this is not an almost-sure result.
\end{abstract}

\pagebreak
\titleformat{\chapter}[display]{\Large\bfseries\centering}%
    {\chaptername~\thechapter}{1ex}{}[\titlerule]
 
\begin{center}
\line(1,0){480}  
\end{center}
{
{\tableofcontents}}
\begin{center}
\line(1,0){480}
\end{center}

\pagebreak



\section{Introduction}\label{INTRO}

\ \ \ \ \ Let $M_n({\Bbb R})$ denote the set of $n\times n$ matrices
with real entries and $M_n({\Bbb Z})$ the set of $n\times n$ integer matrices. A matrix
$T\in M_n({\Bbb R})$ is called  {\it
expanding } (or {\it expansive }) if all its eigenvalues have moduli
$>1$. For an expanding matrix $T\in M_n({\Bbb R})$ and any set $ A:= \{a_1, ..., a_q\}
\subset {\Bbb R}^n$, called a \emph{digit set}, 
there exists a  unique nonempty compact set $F=F(T,A)$  satisfying
\begin{equation}\label{eqn00}
F = \bigcup _{j =1}^q T^{-1}(F + a_j).
\end{equation}

$F$ is called a {\it self-affine set (fractal) or a self-affine attractor}, and
can be viewed as the \emph{invariant set} or the \emph{attractor} of the (affine)
{\it iterated function system} (IFS)
\begin{equation}\label{IFS_maps}
 \{ \ f_j(x)={T}^{-1}(x+a_j) \ \}_{j=1}^q.
 \end{equation}

We sometimes write $F(T, A)$ for $F$ to
stress the dependence on $T$ and $A$. 
Further, if
$A\subset {\Bbb{Z}^n}$ and $T\in M_n({\Bbb{Z}})$,
then $F$ is called an \emph{integral self-affine set}.
Here, we shall study some dimension problems on such sets. 
If, additionally, 
$|\det(T)|=q$
and the integral self-affine set $F$ has positive Lebesgue measure, then $F$
is called an \emph{integral self-affine tile}.

Tiles are used to deal with the overlapping pieces of $F$ through graph directed sets (see \cite{K1}).

Let ${dim}_H F$ denote the Hausdorff dimension of $F$, ${\underline{dim}}_BF$  the lower box dimension of $F$ and ${\overline{dim}}_B F$  the upper box dimension of $F$ (see \cite{F1}, \cite{F2}
for a discussion of these dimensions).
If ${\underline{dim}}_B F={\overline{dim}}_B F$, this common value is the box dimension  (or the entropy dimension) of $F$ and denoted by ${dim}_B F$. 
 The IFS
$\{f_j\}_{j=1}^q$ is said to satisfy
the \emph{open set condition} (OSC) if there exists a bounded nonempty
open set $U$ such that $\bigcup_{j=1}^q f_j (U)\subset U$ with the union
disjoint \cite{Mo}. The $f_j$ are called \emph{similarity transformations} if
$$|f_j(x)-f_j(y)| = c_j|x-y| \ \ \ (1\leq j\leq q),$$ where $0<c_j<1$ and $|\cdot |$ stands for the Euclidean norm. In such a case, $F$ is called a {\it self-similar set}.
When the $f_j$ are similarity transformations and $\{f_j\}_{j=1}^q$  satisfies the OSC, it is known \cite{F2} that ${dim}_H F={dim}_B F=s$, where  $s$ is the unique number satisfying
$$\sum_1^q c_j^s=1.$$

As for the non-self-similar case, an almost-sure formula for the Hausdorff dimension has been given by  Falconer \cite{F3}. However, more precise
 dimension formulas have been given only in special cases.
 There are some partial results for $F$ in the plane \cite{HLR}, \cite{M}.  In the non-self-similar case, the study of exact fractal dimensions of $F$ in higher
dimensional Euclidean spaces is still very limited. For instance,  B$\acute{\mathrm{a}}$r$\acute{\mathrm{a}}$ny, Hochman and Rapaport \cite{BHR}, and Hochman and Rapaport \cite{HR} have  recent work on the dimension of planar self-affine sets and measures under certain separation conditions
(strong open set condition and exponential separation) and strong irreducibility of the matrices; however it is not applicable in case of integral self-affine sets.  Another aspect of integral self-affine sets is that many exceptional self-affine sets 
(i.e., the Hausdorff dimension is less than the singular value dimension, see \cite{F3},  \cite{M},  \cite{KK2}, \cite{KP1}) fall into this category. Here we neither impose separation conditions nor restrict our attention to planar self-affine sets. But we rather concentrate on integral self-affine sets and consider the following problems.

\bigskip

\textit{Problem 1.} How can one give a formula for $dim_H F$ or estimate $dim_H F$ ?

\bigskip
\textit{Problem 2.} Does $dim_B F$ exist ?
\bigskip

\bigskip
The Hausdorff \textit{dimension of a probability measure} $\mu$ on a metric space $X$
is defined by
$$dim(\mu) = \inf \ \{ dim_H B \ : \ B \ is \ a \ Borel \ set \ with  \ \mu(B) = 1 \}.$$
If $dim(\mu) = dim_H X$, we say that $\mu$ is a \textit{measure on $X$ of full dimension}. Let $T\in M_n({\Bbb Z})$ be expanding and $A\subset {\Bbb{Z}^n}$.
Then $F(T,A) \  \rm    {mod  \ 1}$ is a $T$-invariant subset of the n-torus $\mathbb{T}^n$, i.e., $T(F(T,A) \  \rm    {mod  \ 1})\subseteq F(T,A) \  \rm    {mod  \ 1}$.

In this paper, we also consider the following existence problem, which has not been solved yet and is also the  homogeneous version of an open problem mentioned by Dekking \cite{De}, see also Question 5.1 in \cite{PS}. The linear and non-homogeneous  (i.e., some of the IFS functions defining $F$ have different linear parts) version of  Question 5.1 has recently been answered negatively by Das and
Simmons \cite{DS}.

\bigskip

\textit{Problem 3.} For an arbitrary expanding integer matrix $T\in M_n({\Bbb Z})$  and $A\subset {\Bbb{Z}^n}$, 
can we construct an ergodic $T$-invariant Borel probability measure on $F(T,A)\  \rm    {mod  \ 1}$ (or on a compact $T$-invariant set $K\subset \mathbb{T}^n$) of full dimension ?

\bigskip

Unlike the non-homogeneous version, we answer this question positively in certain cases. We expect our positive answer to be as useful as the negative answer in the non-homogeneous case. This problem was studied by McMullen \cite{M}, and Kenyon and Peres \cite{KP1} in the special case where $T$ is a direct sum of
certain integer matrices and for specific $A$. 
Here we don't have such restrictions on $T$ or $A$, but we make use of their results.

\bigskip

Our purpose here is to use a fractal perturbation process for the study of dimensions of $F\subset \mathbb{R}^n.$ Therefore, this paper may be considered as a sequel to
\cite{K1}, in which only Problems 1-2 were considered for planar integral self-affine sets. It is also a version of \cite{K21}. 
Our principal task  is to deal with the pathological case in which the characteristic polynomial of
$T$ does not have a factorization into linear factors over $\Bbb{Z}$. 
Before stating our results, we need to define the perturbed fractals (or perturbations of $F$).

 \medskip

Let $J$ be the real Jordan normal form of $T$,
where any diagonal block
$
{\tiny \begin{bmatrix}
   \lambda & 0 \\
   0 & \bar{\lambda}
\end{bmatrix}}
$ of $J$
corresponding to non-real complex eigenvalues  $\lambda=a\pm ib$ (if any)  is replaced by
$
{\tiny \begin{bmatrix}
   a & -b \\
   b & a
\end{bmatrix}} \ \ \textrm{or} \ \ {\tiny \begin{bmatrix}
   a & b \\
   -b & a
\end{bmatrix}}.
$

For convenience in the sequel, we use the notation $J_k:=J^k$, where $k\in \mathbb{N}$; but $T^k$ (a power of $T$) and $T_k$ (a sequence), which will be defined below, are  different matrices although they are related. Then  $P^{-1}T^kP=J_k$ for some real invertible matrix $P$. Here we choose  $P$  so that \underline{the distance between two distinct lattice points in $P^{-1}\mathbb{Z}^n$ is greater than $n$} (This is similar to the property of the standard integer lattice $\mathbb{Z}^n$ that the distance between two distinct lattice points in  $\mathbb{Z}^n$ is $\geq 1$). That property of $P^{-1}\mathbb{Z}^n$ can be achieved by multiplying $P$ by a constant if necessary.

Our basic idea is to reduce the study of $dim_H F(J,P^{-1}A)$ to integral self-affine sets. 
Let
\begin{equation}\label{IFS_digits}A_{k}=\sum_{i=1}^{k}T^{i-1}A=\{ a_{\textbf{j}}=\Sigma_{i=1}^{k}T^{(k-i)}a_{j_i} \  : a_{j_i}\in A \},
\end{equation} where
 $\textbf{j}=j_1j_2...j_k, \ \ \ 1\leq j_i \leq q$. Then $J_1=J, \ \ A_1=A$ \ and $$ \tilde{F}:=P^{-1}F=F(J,P^{-1}A)= F(J_k,P^{-1}A_{k})$$ for each $k$.
 For simplicity, we also write
$$\tilde{A}:=P^{-1}A,  \quad  \quad  \tilde{A}_{k}:=P^{-1}A_{k} \quad  \quad  \quad  \mathrm{ and}  \quad  \quad \quad  \tilde{a}_{\textbf{j}}:=P^{-1}a_{\mathbf{j}}$$ for $a_{\mathbf{j}}\in A_{k}.$ Thus $\tilde{A}_{k}\subset P^{-1}\mathbb{Z}^n$ and  the distance between two distinct  digits of $\tilde{A}_{k}$ is greater than $n$. We use 
this feature in the proof of Lemma \ref{neighboring_pieces2}. 
 Let $\lceil \cdot \rceil$ and $\lfloor \cdot \rfloor$ denote
the ceiling and the floor functions respectively. Corresponding to each $\tilde{a}_{\mathbf{j}}\in \tilde{A}_k$, define the following column vectors

\begin{equation}\label{floor_digit_lower_perturbation}
({d_{\mathbf{j}}})_{l1}:=\bigg \{
\begin{array}{c}
\ \ \ \lfloor \ (\tilde{a}_{\mathbf{j}})_{l1} \ \rfloor \quad \ \ \ \ \ \ \ \ \ \   \mbox{if $(\tilde{a}_{\mathbf{j}})_{l1}\geq 0$,} \\
-\lfloor \ |(\tilde{a}_{\mathbf{j}})_{l1}| \ \rfloor \quad \ \ \ \ \ \ \ \ \  \mbox{if $(\tilde{a}_{\mathbf{j}})_{l1}< 0$,}
\end{array}
\end{equation}

where
$l=1,2..,n$.
Then $$D_{k}=\{ d_{\mathbf{j}} : \ \mathbf{j}=(j_1,j_2,...,j_k), \ \ \ 1\leq j_i \leq q \}\subset \mathbb{Z}^n$$

while $\tilde{A}_{k}$ may not be a subset of  $\mathbb{Z}^n$.
By a translation of the digits in $\tilde{A}_{k}$, we may always assume that 
 $\tilde{A}_{k}$ or $D_{k}$ consists of non-negative vectors and contains  the zero vector \cite{KL1}. Next, we define a sequence
 of integer matrices $T_{k}$ by the $ij$-th entries:

\begin{equation}\label{pert_matrix} (T_{k})_{ij}:=\bigg \{ \begin{array}{c}
 \ \ \ \lceil (J_k)_{ij}  \rceil \quad \ \ \ \ \ \   \mbox{if $(J_k)_{ij}\geq 0$,}  \\
   -\lceil |(J_k)_{ij}|  \rceil \quad \ \ \ \ \     \mbox{if $(J_k)_{ij}< 0$,}
\end{array}
\end{equation}

Then, define the corresponding ``dynamically perturbed fractals'' or ``dynamical perturbations of $\tilde{F}$'' by
\begin{equation}\label{perturbed_fractals}
F_{k}=F(T_{k},D_{k})
\end{equation}
which are integral self-affine sets.
Finally, we write
\begin{equation}\label{perturbation_dimensions}
\delta_{k}=dim_H F_{k} \
\end{equation}
for the ``perturbation dimensions'', cf. \cite{K1}.

We write  $|\mathbf{j}|$ for the length of the multi-index $\mathbf{j}$. For those who want to trace the iterations, we also mention that the $k$-compositions $$\{ \ f_{\mathbf{j}}(x):=f_{j_1}\circ f_{j_2}\circ \cdots \circ f_{j_k}= J_k^{-1}(x+\tilde{a}_{\mathbf{j}}) \ \}_{|\mathbf{j}|=k}$$ form the IFS for $\tilde{F}= F(J_k,\tilde{A}_{k})$ and
$$\{ T_k^{-1}(x+d_{\mathbf{j}}) \ \}_{|\mathbf{j}|=k}$$ is the IFS for the perturbation $F_k:=F(T_k,D_k)$ 
 for each $k$. So the order of the indices in a multi-index $\mathbf{j}$ is the same as the order of the function indices in the above composition.

The dynamical perturbations $F_k$ defined for each of the data ($T^k,\ A_k$), $k\in \mathbb{N}$ generating the single integral self-affine set $F=F(T,A)$ should not be confused with   the classical perturbations on the fixed data ($T,A$), i.e. $k=1$, the first level in our study. In the classical perturbation case, it is known that the dimension is discontinuous in the digits of $A$ \cite{M}. The same results are expected to hold for our perturbations. But the situation is not the same. As we shall see, the justification of that requires considerable study. The dynamical perturbations that we propose have some special features:

\bigskip

\begin{itemize}
  \item[(P1)] Although we perturb both $T^k$ and $A_k$ at all levels, they give important (uniform) convergence results (see Proposition 2.8, Remark 2.9
  ) to be used in the proof of the main theorem below.

  \item[(P2)] Roughly speaking, each perturbation at level $k$ or higher levels rectifies the deflection of $F$ or $dim(F)$ resulting from previous levels. That is not the case for the classical perturbations since one has only rectifications for the first level.
  \item[(P3)] Perturbation dimensions can be computed by known methods (\cite {K21}, see earlier papers \cite{HLR}, \cite{KP} for $n=2$).

  \item[(P4)] For a subsequence of the dynamical perturbations $F_k$, or say for all large $k$,  the neighbor graph of $F_k$ is not completely independent of the neighbor graph of the original fractal $F$. Even the overlap structure is preserved in some sense (see Corollary 4.7). That is why, the overlaps  do not affect the proofs negatively.
   \item[(P5)] Our approach is not only based on the study of multiple symbolic codings of points of the fractals, it rather relies on the explorations of relations of the subsets with the same  coding or the same level.          
\end{itemize}

\noindent In this paper, we give three applications of the  dynamical perturbation method. Among other things, we prove the following main results.

\begin{theorem}\label{Hausdorff}
Let $F\subset \mathbb{R}^n$ be an integral self-affine set and $\delta_{k}=dim_H F_{k}$. Then

$$ dim_H F=\underset{k\rightarrow \infty}{\lim} \delta_{k}.$$

Further, there is a positive integer $k_0$ such that  $\delta_{k}  \leq dim_H F$  for every $k\geq k_0$. 
\end{theorem}

\medskip

Compared to the results in the literature, the theorem is general in two directions: any dimension $n$ and arbitrary integer matrix $T$ and digit set $A$. The applicability of the above theorem comes from the fact that $  \delta_{k}$ can be given by extensions of McMullen's formula \cite{M}. 
Planar  numerical examples for different cases can be found in \cite{K1}.
But accurate calculation of $  \delta_{k}$ may be highly nontrivial and we think that such computation is a topic for computer scientists.
Therefore, the practical performance of Theorem \ref{Hausdorff} will depend on the potential work of numerical analysts or computer scientists on the computation of
$\delta_{k}$. 
Another application of the method is the following.

\begin{theorem}\label{Main_Box}
Let $F\subset \mathbb{R}^n$ be an integral self-affine set. Then the box dimension of $F$ exists and

$$  dim_B F=\underset{k\rightarrow \infty}{\lim} dim_B F_k.
$$
\end{theorem}

\bigskip

By using Theorem \ref{Hausdorff}, we obtain the following result regarding \textit{Problem 3} above. 

\begin{theorem}\label{full_dimension}
Assume that $F=F(T,A)$ is an integral self-affine set. 
Then, viewed as an invariant subset of the n-torus under the toral endomorphism $T$, there is an ergodic $T$-invariant Borel probability measure on $F$  of full dimension.
\end{theorem}

As for the organization of the paper, in Section  \ref{pert_defl1}, 
we describe the ``fractal perturbation method (or the fractal deflection process)''. 
In Section \ref{Stability}, we prove some basic results on perturbed fractals $F_k$. 
The first two sections include frequently used concepts and results.
In Section \ref{Neighbor Structure}, the neighbor graphs of dynamical perturbations  are studied.  In contrast to the classical perturbations, our perturbed fractals eventually retains the neighbor structure of the original fractal.
In Section \ref{An application}, an important box dimension inequality is proved. In Section \ref{theorems}, we prove Theorem \ref{Hausdorff}. 
Section \ref{Lower_Bounds_Dimension} consists in the proofs of the lower bounds for box and Hausdorff dimensions. Then the proof of Theorem \ref{Main_Box} is completed.
In Section \ref{Full_dimension}, we study measures on $F_k$ and prove
Theorem \ref{full_dimension}. 
In the Appendix (Section \ref{appendix}), we give an extension of McMullen's box dimension formula for perturbed fractals $F_k$, which is needed in the paper. 

Concerning the notation, we mention that a black square marks the end of the proof of a lemma, proposition or theorem while a white square marks the end of 
a remark or a definition.

\section{Dynamically Perturbed (Deflected) Fractals}\label{pert_defl1}

\subsection{Reduction to Triangular Matrices  }\label{real}

To exclude trivial or known cases, we will impose the conditions that $A$ has at least two elements  and $T$ is not as in the following remark. 

\begin{remark}\label{trivial_cases}
\rm{ According to (\ref{floor_digit_lower_perturbation}),  (\ref{pert_matrix}),  (\ref{perturbed_fractals}),  we have $F=F_k$ for every $k$ (so that $dim_H F=\delta_{k} $) in the following trivial cases:

                 \begin{enumerate}
                   \item integral self-affine sets in the real line,
                   \item integral self-affine sets generated by  $2\times 2 $ integer matrices $T$ with integer eigenvalues (see (2.7) in \cite{K1}),
                   \item integral self-affine sets generated by  an integer matrix $T$ such that $T^{l}=mI$ ($m\in \mathbb{Z}$ and $I$ is the identity matrix) for some positive integer $l$. In this case, we  assume that $F=F(T^{l},A_l)$.  \hfill$\Box$
                 \end{enumerate}
}  

\end{remark}

\bigskip
Other cases makes much more trouble in the study of dimension. 
For that, we state a proposition from our recent work \cite[Proposition 4.1]{K1}. 

\begin{prop}\label{thmm}
Assume that the characteristic polynomial of $T\in M_n(\Bbb{Z})$, n=2,3, is irreducible over $\Bbb{Z}$, and $T, T^2$ and $T^3$ (if $T\in M_3(\Bbb{Z})$)
are not conjugate to a similarity matrix. Let $\lambda_i$ ($i=1,2 \ \textrm{or} \ 3$) be the eigenvalues of $T$. Then, for each eigenvalue $\lambda_i$,
$\{|\lambda_i|^n : n\in \Bbb{N}\}$ is an infinite sequence of irrational numbers.  
\end{prop}

If $T\in M_n(\Bbb{Z})$ is irreducible as above, that causes an infinite approximation of dimension in our study (unlike Remark \ref{trivial_cases}). Another case occurs when $T$ has repeated eigenvalues. In that 
case, we have at least one non-trivial Jordan block and this causes more complication as we will see below. 
We also mention that we will use the lower triangular Jordan normal  form of $T$. 

For the general case $T\in M_n({\Bbb{Z}})$, let $J$ be the real Jordan normal form of $T$ so that $P^{-1}T^{k}P=J^k$ for some invertible matrix $P$. Let $J_k=J^k$. 
Related to $J_k$, we  look for  lower triangular matrices $T_k$ with the following properties:

\medskip

  (i) \ \ $T_k$ must be an integer expanding matrix,

  (ii) \ the diagonal entries $(T_k)_{ii}$ must be nonnegative and increasing in $i$ or, more generally,  $|(T_k)_{ii}|$ is increasing in $i$. For real eigenvalues, we note that the condition  $(T_k)_{ii}\geq 0$
  can always be achieved by considering $\tilde{F}=F(J_k^2,\tilde{A}_k+J_k\tilde{A}_k)$ in place of $\tilde{F}=F(J_k,\tilde{A}_k)$. 

 \medskip

Property (ii) is required for the use of McMullen type formulas (see Appendix). Notice that no power of
$T\in M_3({\Bbb{Z}})$ may be conjugate to a similarity or an integer matrix when the hypotheses in Proposition \ref{thmm} are satisfied, see \cite[Corollary 4.2]{K1}.
We want to convert not only $J$ but also all its positive powers  to matrices satisfying property (i) and (ii). 

Let
$J_{\lambda_p}$ ($p=1,2,...,s$) be the real Jordan blocks of $J$, where $\lambda_p$ is the corresponding eigenvalue.  
It will suffice to restrict our study to the following type of Jordan blocks since  our argument  works for the other 
types. For the same reason, we may also assume that $|\lambda_1|< \cdots <|\lambda_s|$.
  $J$ is a real block diagonal matrix, and  each real Jordan block $J_{\lambda_p}$ is either of the form (if the corresponding eigenvalue $\lambda_p$ is real and has algebraic multiplicity $\geq 2$ )
\begin{equation}\label{Jordan1}
 J_{\lambda_p}=
\left[
\begin{array}{cccc}
  \lambda_p & \ \ \ \ \ \ \ 0 \ \ \ \cdots    & 0 & 0 \\
   1 &  \lambda_p  & 0 & 0 \\
   \vdots & \ \ \ \ \ \ \  1 \ \ \ \ddots  & \vdots & \vdots \\
   0 & \vdots & \ \ \ \  \lambda_p \ \ \ \ &  0 \\
  0 &   \ \ \ \ \ \ \ 0 \ \ \ \cdots   &   1 & \lambda_p \\
\end{array}
\right],
\end{equation}
where every entry on the sub-diagonal is $1$,
 or is a block matrix itself, consisting of $2\times2$ blocks (for non-real eigenvalue $\lambda_p   = a   + i b$  with given algebraic multiplicity $\geq 2$) of the form

\begin{equation}\label{Jordan2}
J_{\lambda_p}=
\left[
\begin{array}{cccc}
  C_p & \ \ \ \ \ \ \ 0 \ \ \ \cdots    & 0 & 0 \\
   I_2 &  C_p  & 0 & 0 \\
  \vdots & \ \ \ \ \ \ \  I_2 \ \ \ \ddots  & \vdots & \vdots \\
   0 & \vdots & \ \ \ \   C_p \ \ \ \ &  0 \\
  0 &   \ \ \ \ \ \ \ 0 \ \ \ \cdots  & I_2 & C_p \\
\end{array}
\right],
\end{equation}
where sub-diagonal blocks are $2\times2$ identity matrices and
$$ C_p=
{ \begin{bmatrix}
   a & -b \\
   b & a
\end{bmatrix} }={\rho_p}R_{\theta}$$
with $|\lambda_p|={\rho_p}>1$ and  $R_{\theta}$ is a rotation matrix.

As usual, we use the symbol $\oplus$ for the direct sum of matrices. Then 
$J_k=J^k=J_{\lambda_1}^k\oplus J_{\lambda_2}^k \oplus \cdots \oplus J_{\lambda_s}^k$ with $J_{1}=J$.  
For simplicity and consistency with $T_k$, 
we use the notation $J_{k,p}$ instead of $J_{\lambda_p}^k$. In this notation, $J_k$ takes the following form
$$J_k=J_{k,1}\oplus J_{k,2} \oplus \cdots \oplus J_{k,s}.$$

If $J_k$ is not an integer matrix, we make it so by considering
$T_k$ in (\ref{pert_matrix}). 
The point for using the ceiling function $\lceil x \rceil$ or the floor function $\lfloor x \rfloor$ in (\ref{pert_matrix}) is to convert $J_k$ to an expansive integer matrix so that property (i) is satisfied.

\subsection{Perturbed (Deflected) Fractals}\label{pert_defl}

\begin{defn}\label{upper_perturbation}

 \rm{(i) Let $F_k
$ be as in  (\ref{perturbed_fractals}) and $\delta_k
$ be as in  (\ref{perturbation_dimensions}).
$F_k$ may be termed as a \textit{(dynamical) lower perturbation of $F$ (or $P^{-1}F$}) and
$\delta_k$ {\rm the} $k$-th \textit{lower perturbation dimension} {\rm of} $F$.

(ii) If we interchange the floor function in (\ref{floor_digit_lower_perturbation}) and the ceiling function in (\ref{pert_matrix}), then the resulting perturbed fractal in  (\ref{perturbed_fractals}) may be called an \textit{upper perturbation of $P^{-1}F$}. In that case, we may write $\upsilon_{k}$ for its dimension. (In such a case, the proofs that follow can easily be adjusted accordingly.)} \hfill$\Box$

\end{defn}

\bigskip
 It is known that elements of a general self-affine set  can be represented by infinite series or radix expansions as follows:
\begin{equation*}   
F(T,A)=\left\{\sum_{i=1}^\infty T^{-i} a_{j_i} : \ a_{j_i}\in A \right\}.
\end{equation*} 
For convenience, write $J_k$ instead of $J^k$ so that we use the same type of labelling for both $F_k$ and $J^k$. Let $\tilde{F}:=F(J,\tilde{A})=P^{-1}F$.
For an arbitrary  element $x\in \tilde{F}$, we have $x\in F(J_k,\tilde{A}_k)=\tilde{F}$ for every $k$. This allows us to define a sequence. Let $x_k\in F_{k}=F(T_k,D_{k})$ denote \textit{the element whose series representation uses the same index (or multi-index) sequence as that in the representation of $x$}; 
 i.e.,
\begin{equation}\label{special_correspondence}
x=\sum_{i=1}^\infty J_{k}^{-i} \tilde{a}_{{\mathbf{j}}_i}, \ \ \ \ \ \ \ \ \ \ \ \ x_k=\sum_{i=1}^\infty T_{k}^{-i} d_{{\mathbf{j}_i}},
\end{equation}
where $a_{{\mathbf{j}}_i}\in \tilde{A}_k$ and $d_{{\mathbf{j}_i}}\in D_{k}$.
Note that $x$ may have more than one infinite series representation. 
We  shall prove  the special convergence $\underset{k\rightarrow\infty} {\lim} x_{k}=x $ (uniformly) 
so that $F_{k}\rightarrow \tilde{F}$. 
We need it for the proof of our main theorem too. First, we give a lemma to be used in the verification of the above convergence. In fact, the  following 
three lemmas will often be used for convergence results, diameter estimates, etc. First, we introduce some notation.

\medskip

\textit{Notation.} \begin{itemize}
                     \item  If $M$ is a block matrix with all blocks are $2\times 2$ matrices, denoted by $M_{ij}$, then by writing $C_p \odot  M$, we mean
 the block  matrix with entries $(C_p \odot  M)_{ij}=C_pM_{ij}$, where $$ C_p=
{ \begin{bmatrix}
   a & -b \\
   b & a
\end{bmatrix} }={\rho_p}R_{\theta}=|\lambda_p|R_{\theta}$$ as in (\ref{Jordan2}). 
So $C_p \odot  M$ and $M$ have the same size (this is similar to scalar multiplication in linear algebra). Furthermore, for any real number $\beta$, $\left\lceil \beta C_p^{k} \right\rceil  $  is the perturbed matrix with entries
\begin{equation*}\label{Exp33}
\left( \left\lceil \beta C_p^{k} \right\rceil  \right)_{ij}:=\bigg \{
\begin{array}{c}
\ \ \lceil \ (\beta C_p^{k})_{ij} \ \rceil  \quad \ \ \ \ \ \ \ \ \ \ \  \mbox{if $(\beta C_p^{k})_{ij}\geq 0$,} \\
-\lceil \ |(\beta C_p^{k})_{ij}| \ \rceil \quad \ \ \ \ \ \ \ \ \  \mbox{if $(\beta C_p^{k})_{ij}< 0$.}
\end{array}
\end{equation*}
      \item  Recall that $J_{1,p}=J_{\lambda_p}$ ($1\leq p \leq s$)  has size $2n_p\times 2n_p$ for non-real eigenvalues.
Denote by $(\tilde{a}_{\mathbf{j}})_p$ the projection of the vector $\tilde{a}_{\mathbf{j}}$ to the coordinates  $\sum_{t=1}^{p}2n_{t-1}+1,...,\sum_{t=1}^{{p}}2n_{t}$ ($n_{0}=0$). Thus $(\tilde{a}_{\mathbf{j}})_p$ is the p-th block of $\tilde{a}_{\mathbf{j}}$  corresponding to the Jordan block
$J_{\lambda_p}$. For example, if
$\lambda_1,\ \overline{\lambda_1}=\pm 2i$, \  $\lambda_2=1$, $J_{\lambda_1}=\tiny{\left[
\begin{array}{cc}
  0 &  -2 \\
  2 &  0 \\
\end{array}
\right]},$  $J_{\lambda_2}=\tiny{\left[
\begin{array}{c}
 1
\end{array}
\right]},$ $\tilde{a}_{\mathbf{j}}=\tiny{\left[
\begin{array}{ccc}
3 & 4 & 5
\end{array}
\right]^t}$ (``$t$'' for transpose),  then $(\tilde{a}_{\mathbf{j}})_1=\tiny{\left[
\begin{array}{ccc}
3 & 4
\end{array}
\right]^t}$,  $(\tilde{a}_{\mathbf{j}})_2=\tiny{ \left[
\begin{array}{ccc}
5 
\end{array}
\right]}.$

                     \item  It is also appropriate to use big O notation in the following lemma.
\end{itemize}

\medskip

\begin{lemma}\label{deflection0} Let
 $\tilde{F}=P^{-1}F=F(J_k,\tilde{A}_k)\subset \mathbb{R}^n$. 
Then there is a positive constant $c$ such that for all large $k$, we have $$\left|(J_{k}^{-1}-T_{k}^{-1})\tilde{a}_{\mathbf{j}} \right|\leq \frac{c}{\sqrt{m_1}}, $$
where $\tilde{a}_{\mathbf{j}}\in \tilde{A}_k$ and
$T_{k}$
has eigenvalues with moduli
$m_1\leq m_2\leq \cdots \leq  m_s.$ 
\end{lemma}

\begin{pf}  The proof is bulky and involves heavy calculation.
Let $F_{k}=F(T_k,D_k)$ be the perturbations of $\tilde{F}$. It is enough to work with real Jordan blocks.
Let $J=J_1=J_{1,1}\oplus J_{1,2} \oplus \cdots \oplus J_{1,s}$ be the real Jordan form of $T.$ 
To prove the inequality in the lemma, first we consider a 
block $J_{1,p}$  with $1\leq p \leq s$ and  compute  the block  $J^{-1}_{k,p}$ of $J^{-1}_{k}$. Do the same for $T^{-1}_{k,p}$.

\medskip

Case 1. Assume that $\lambda_p   = a  \pm i b$ are non-real conjugate eigenvalues of $J$.

\medskip

\noindent For non-real eigenvalues $\lambda_p   = a  \pm i b$  with  algebraic multiplicity $n_p\geq 2$, we notice that (\ref{Jordan2}) can be written as 
$$J_{1,p}=J_{\lambda_p}=C_p \odot I+N,$$ where  $I$ is the  $2n_p\times 2n_p$ identity matrix, and $N$ is the $2n_p\times 2n_p$  nilpotent matrix
$$N=\left[
\begin{array}{cccc}
  0 &  \ \ \ \ \ \ 0 \ \ \ \cdots    & 0 & 0 \\
  I_2 &  \ \ \ \ \ \ 0 \ \ \ \cdots   & 0 & 0 \\
   0 & \ \ \ \ \ \  I_2 \ \ \  \ddots  & \vdots & \vdots \\
\vdots  & \vdots \ &  \ \ \ \ 0 \ \ \ \ &  0 \\
  0 &      \ \ \ \ \ \ 0 \ \ \ \cdots   & I_2   &   0 \\
\end{array}
\right],$$
where sub-diagonal blocks are the $2\times 2$ identity matrices $I_2$. Let $\alpha_0:=k$ if $k<n_p$ and $\alpha_0:=n_p-1$ when $k\geq n_p
$.
Denote the $i$-th binomial coefficient by $b_i$; i.e.
$$b_i:=\begin{pmatrix}
k  \\
i
\end{pmatrix}.$$
Let $$
N_1=C_p^{-k}\sum_{i=1}^{\alpha_0}
b_i
C_p^{k-i} \odot  N^{i}=\sum_{i=1}^{\alpha_0}
b_i
C_p^{-i} \odot  N^{i},   \ \ \ \ \ \ \
N_2=\sum_{i=1}^{\alpha_0}
\left\lceil b_i
C_p^{k-i}\right\rceil  \left\lceil
C_p^{k}\right\rceil ^{-1}  \odot  N^{i}.$$ Consider large $k$, e.g. $k\geq n_p$, so that $k>\alpha_0 $ and hence the power $k-i$ is always positive in the following binomial expansion.
Then  $J_{1,p}:=J_{\lambda_p}=C_p \odot I+N$ implies that the  $p$-th  diagonal  block  of  $J_k$ is
\begin{equation*}\label{ }
 J_{k,p}:=(J_{1,p})^k=
C_p^{k} \odot  I+\sum_{i=1}^{\alpha_0}
b_i
C_p^{k-i}  \odot   N^i
=C_p^k \odot  \left( I+\sum_{i=1}^{\alpha_0}
b_i
C_p^{-i} \odot  N^i\right)=C_p^k \odot  \left( I-(-N_1)\right)
\end{equation*}
and
$p$-th diagonal block of $T_k$ is
\begin{equation}\label{Power_Jordan_Block1}T_{k,p}=\lceil C_p^k \rceil   \odot  \left( I-(-N_2)\right)\end{equation} since  $N^{n_p}=0$.
Also $N_1,N_2$ commute because $ C_p^{-i}$, $\left\lceil b_iC_p^{k-i}\right\rceil $, $\lceil C_p^k \rceil ^{-1}$  are rotation matrices. Similar to the calculus series expansion $\frac{1}{1-x}=\sum_{j=0}^{\infty} x^j \ (|x|<1)$,  the equality  $N_1^{n_p}=N_2^{n_p}=0$ then yields
\begin{equation}\label{Power_Jordan_Block2}  J_{k,p}^{-1}= C_p^{-k} \odot  \left( I+\sum_{j=1}^{\alpha_0}(-1)^jN_1^j\right), \ \ \ \ \ \ \ \ \ \ T_{k,p}^{-1}=\lceil C_p^k    \rceil ^{-1} \odot \left( I+\sum_{j=1}^{\alpha_0}(-1)^jN_2^j\right),
\end{equation}
for  $k\geq n_p$, and since $N_1,N_2$ commute, we have
{\footnotesize \begin{eqnarray}\label{maininequality2}
J_{k,p}^{-1}-T_{k,p}^{-1} & = &  (C_p^{-k}- \lceil C_p^k    \rceil ^{-1}) \odot \left( I+\sum_{j=1}^{\alpha_0}(-1)^jN_1^j\right)+ \lceil C_p^k \rceil ^{-1} \odot \sum_{j=1}^{\alpha_0}(-1)^j (N_1^j-N_2^j) \\
& = & (C_p^{-k}- \lceil C_p^k    \rceil ^{-1}) \odot \left( I+\sum_{j=1}^{\alpha_0}(-1)^jN_1^j\right)+ \lceil C_p^k \rceil ^{-1}  \odot  \left(  \sum_{j=1}^{\alpha_0}(-1)^j( N_1-N_2)\sum_{l=0}^{j-1}N_1^lN_2^{j-1-l}\right). \nonumber
\end{eqnarray}}

\noindent By (\ref{maininequality2}), we then need to obtain upper bounds for
$$   (\mathrm{I}) \ \  \lceil C_p^k \rceil ^{-1} \odot  \sum_{j=1}^{\alpha_0}(-1)^j (N_1^j-N_2^j)(\tilde{a}_{\mathbf{j}})_p, \ \ \ \ \
 (\mathrm{II}) \ \ (C_p^{-k}- \lceil C_p^k \rceil ^{-1}) \odot \left( I+\sum_{j=1}^{\alpha_0}(-1)^jN_1^j\right)(\tilde{a}_{\mathbf{j}})_p.$$

For that, we note the following:

\medskip

\noindent Let $||\cdot ||$ denote the spectral norm of a matrix, and $c$ stand for a constant. Recall that $ \lceil C_p^{k} \rceil$ is a rotation matrix.
Then from Case 2, (4.24) in the proof of Proposition 4.4 in \cite[p.16]{K1}, there is a  $2\times 2$  matrix $G$ such that
 \begin{equation}\label{estimate00}
   \lceil C_p^{k} \rceil ^{-1} = C_p^{-k}+G \ \ \textrm{with} \ \  ||G||<\frac{c}{m_p}
 \end{equation}
     for all large $k$.
Also, we can write
 \begin{equation}\label{estimate01}
\lceil C_p^{k} \rceil  = C_p^{k}+G_0,  \ \ \ \     \left\lceil b_iC_p^{k-i} \right\rceil =b_iC_p^{k-i}+G_i
\end{equation}
for some appropriate matrices $G_0, G_i  $. Clearly,  $||G_0||, ||G_i || \leq 2$.

\bigskip

\textit{Upper bound for \ \ \rm{(I)} \  $\lceil C_p^k \rceil ^{-1} \odot \sum_{j=1}^{\alpha_0}(-1)^j (N_1^j-N_2^j)(\tilde{a}_{\mathbf{j}})_p$} :

\medskip

\noindent Let
$$N_0=\left( \sum_{i=1}^{\alpha_0}  b_i C_p^{k-i} G+G_i( C_p^{-k}+G)  \right)\odot N^{i}.$$ 
Then, using $N_1=\sum_{i=1}^{\alpha_0} b_iC_p^{-i} \odot  N^{i},$ (\ref{estimate00}) and (\ref{estimate01}), we obtain
$$
N_2=\sum_{i=1}^{\alpha_0}
\left\lceil
b_i
C_p^{k-i}\right\rceil  \left\lceil
C_p^{k}\right\rceil ^{-1}  \odot  N^{i} = \sum_{i=1}^{\alpha_0}  \left(b_i C_p^{k-i}+G_i \right)( C_p^{-k}+G) \odot N^{i}=N_1+N_0,
$$
and $\sum_{l=0}^{j-1}N_1^lN_2^{j-1-l}$ in (\ref{maininequality2}) will be of the form
\begin{equation}\label{estimate02}
\sum_{l=0}^{j-1}N_1^lN_2^{j-1-l}=\sum_{l=0}^{j-1}N_1^l(N_1+N_0)^{j-1-l}=jN_1^{j-1}+H_1
\end{equation}
for some appropriate matrix $H_1$ depending on $N_0,N_1$. 
Furthermore, by (\ref{estimate01}),
\begin{eqnarray}\label{estimate03}
N_1-N_2 & = & C_p^{-k}  \left\lceil
C_p^{k}\right\rceil ^{-1}  \odot \left(\sum_{i=1}^{\alpha_0}
b_i
\left\lceil
C_p^{k}\right\rceil  C_p^{k-i} \odot  N^{i}  -  \sum_{i=1}^{\alpha_0}  \left\lceil b_i
C_p^{k-i}\right\rceil
C_p^{k} \odot  N^{i}   \right)  \nonumber \\
& = & C_p^{-k}\left\lceil
C_p^{k}\right\rceil ^{-1}  \odot  \left( \sum_{i=1}^{\alpha_0}
b_i
( C_p^{k}+G_0) C_p^{k-i} \odot  N^{i} - \sum_{i=1}^{\alpha_0}  \left(b_iC_p^{k-i}+G_i \right)
C_p^{k} \odot N^{i} \right) \nonumber \\
& = & C_p^{-k}\left\lceil
C_p^{k}\right\rceil ^{-1} \odot  \left( \sum_{i=1}^{\alpha_0} b_i
 G_0 C_p^{k-i} \odot N^{i} - \sum_{i=1}^{\alpha_0}
 G_i C_p^{k} \odot
 N^{i}
 \right) \nonumber   \\
& = & \left\lceil
C_p^{k}\right\rceil ^{-1}  \odot \left(  G_0  \odot N_1 -\left(\sum_{i=1}^{\alpha_0}  G_i \odot   N^{i}\right)
 \right)=\left\lceil
C_p^{k}\right\rceil ^{-1} \odot \left(  G_0 \odot N_1 + H_2\right)
\end{eqnarray}
for all large $k$. Here $H_2:=-\left(\sum_{i=1}^{\alpha_0}  G_i \odot   N^{i}\right)$  by definition.

Now, regarding the second term  $\lceil C_p^k \rceil ^{-1} \odot \sum_{j=1}^{\alpha_0}(-1)^j (N_1^j-N_2^j)$ in (\ref{maininequality2}), we then get
{\footnotesize\begin{eqnarray}\label{estimate0} 
\lceil C_p^k \rceil ^{-1}  \odot  \left( \sum_{j=1}^{\alpha_0}(-1)^j( N_1-N_2)\sum_{l=0}^{j-1}N_1^lN_2^{j-1-l}\right) & = & \lceil C_p^k \rceil ^{-2}  \odot \left(   \sum_{j=1}^{\alpha_0}(-1)^j\left(  G_0  \odot N_1 +H_2\right)(jN_1^{j-1}+H_1)\right) \nonumber \\
& = &  \lceil C_p^k \rceil ^{-2}  \odot  \sum_{j=1}^{\alpha_0}(-1)^j \left[G_0  \odot jN_1^j+G_0  \odot N_1H_1\right] \\
& + & \lceil C_p^k \rceil ^{-2}  \odot \sum_{j=1}^{\alpha_0}(-1)^j \left[jN_1^{j-1}H_2+H_1H_2\right] \label{estimate000}  \\
\nonumber \end{eqnarray}}
by (\ref{estimate02}) and (\ref{estimate03}). 

We next analyze the matrices in the last two lines. $G_0, G_i$ are constant matrices with $||G_0||, ||G_i || \leq 2$. Hence $||G_0||, ||G_i ||, ||H_2||$ do not depend on $k$. But $N_0$ and $N_1$ depend on $k$. So does $H_1$ as a result. Notice that the sums in the definitions of $N_0,N_1$ 
 are finite. Below $c_1,c_2,c_3,c_4,c_5$ are constants. For the terms depending on $k$ and involving $N_1$ but not $H_1$ \big(that is, {\footnotesize $\lceil C_p^k \rceil ^{-2}\odot  \left[ \sum_{j=1}^{\alpha_0}(-1)^j G_0  \odot jN_1^j 
 +  \sum_{j=1}^{\alpha_0}(-1)^j jN_1^{j-1}H_2 \right]$}\big), we find upper bounds for the terms depending on $k$  and for the terms of $N_1$.
By (\ref{estimate00}), we have
$$
 \left|\left|  \left\lceil
C_p^{k}\right\rceil ^{-2}b_i
C_p^{-i} \right|\right| \leq   (\rho_p^{-k}+||G||)^{2} k^{\alpha_0}   \rho_p^{-i} \leq m_p^{-\frac{3}{2}}c_1$$
because $$\lim_{k\rightarrow \infty}\frac{k^{\alpha_0}(\rho_p^{-k}+||G||)^{\frac{1}{2}}}{\rho_p^{i}}=0   \ \ \ \ (i=1,2,...,\alpha_0).$$
This gives $|| \lceil C_p^k \rceil ^{-2} N_1||= O(m_p^{-\frac{3}{2}})$, and we have
\begin{equation}\label{estimate1}
\left|\left|\lceil C_p^k \rceil ^{-2}  \odot \sum_{j=1}^{\alpha_0}(-1)^j G_0  \odot jN_1^j \\
 +  \lceil C_p^k \rceil ^{-2}  \odot \sum_{j=1}^{\alpha_0}(-1)^j jN_1^{j-1}H_2\right|\right|= O(m_p^{-\frac{3}{2}}).
 \end{equation}
As for the terms depending on $k$ and containing $H_1$ \big(that is, {\footnotesize $\lceil C_p^k \rceil ^{-2}\odot  \left[\sum_{j=1}^{\alpha_0}(-1)^j G_0  \odot N_1H_1
 +  
  \sum_{j=1}^{\alpha_0}(-1)^j H_1H_2 \right]$}, hence containing $N_0, N_1$\big),   we find upper bounds for the terms depending on $k$  and for the terms of $N_0$.
We have $$ \left|\left|  \left\lceil
C_p^{k}\right\rceil ^{-2}G_i \right|\right|\leq (\rho_p^{-k}+||G||)^{2}||G_i||\leq m_p^{-\frac{3}{2}}c_2$$  for all large $k$,
and by (\ref{estimate00}),
\begin{eqnarray*} \left|\left| \left\lceil C_p^{k}\right\rceil ^{-2} b_i
C_p^{k-i} G \right|\right|  & \leq &   (\rho_p^{-k}+||G||)^{2}k^{\alpha_0}\rho_p^{k-i}||G||  \\
& \leq & m_p^{-\frac{3}{2}}(\rho_p^{-k}+||G||)^{\frac{1}{2}}k^{\alpha_0}\rho_p^{k-i}\frac{c}{m_p}    \quad \quad \mathrm{since} \ \frac{ m_p}{2}\leq \rho_p^{k}.
\end{eqnarray*}

But $$\lim_{k\rightarrow \infty}\frac{\rho_p^{k-i}}{m_p}= \rho_p^{-i}\lim_{k\rightarrow \infty}\frac{\rho_p^{k}}{m_p}\leq 
\frac{1}{\rho_p^{i}} \ \ \ \ (i=1,2,...,\alpha_0)
$$ because $ \rho_p^{k}\leq m_p$,
and $\lim_{k\rightarrow \infty}k^{\alpha_0}(\rho_p^{-k}+||G||)^{\frac{1}{2}}=\lim_{k\rightarrow \infty}k^{\alpha_0}(\rho_p^{-k}+\frac{c}{m_p})^{\frac{1}{2}}=0$.
Consequently, we obtain 
$$ \left|\left| \left\lceil C_p^{k}\right\rceil ^{-2} b_i C_p^{k-i} G  \right|\right|\leq m_p^{-\frac{3}{2}}c_3.$$
It follows that $||\lceil C_p^k \rceil ^{-2}H_1||= O(m_p^{-\frac{3}{2}})$, and hence we get
\begin{equation}\label{estimate2}
\left|\left|\lceil   C_p^k \rceil ^{-2}  \odot \sum_{j=1}^{\alpha_0}(-1)^j G_0  \odot N_1H_1
 +  \lceil C_p^k \rceil ^{-2}   \odot \sum_{j=1}^{\alpha_0}(-1)^j H_1H_2\right|\right|= O(m_p^{-\frac{3}{2}}).
 \end{equation}
Therefore, (\ref{estimate0}), (\ref{estimate000}), (\ref{estimate1}), (\ref{estimate2}) lead to $$\left|\left| \lceil C_p^k \rceil ^{-1}  \odot \sum_{j=1}^{\alpha_0}(-1)^j (N_1^j-N_2^j) \right|\right|=\left|\left| \lceil C_p^k \rceil ^{-1}  \odot \left(  \sum_{j=1}^{\alpha_0}(-1)^j( N_1-N_2)\sum_{l=0}^{j-1}N_1^lN_2^{j-1-l}\right) \right|\right| = O(m_p^{-\frac{3}{2}}).$$
Additionally, we get
$$\left| m_p^{-1}(\tilde{a}_{\mathbf{j}})_p \right| \leq c_4\rho_p^{-k}\sum_{i=0}^{k-1}||J_{\lambda_p}^i|| \max\{|a| : a\in A\} = c_4\max\{|a| : a\in A\} \sum_{i=0}^{k-1} ||\rho_p^{-k} J_{\lambda_p}^i|| \leq c_5$$ 
because 
$ \rho_p^k\leq m_p$ and $\tilde{a}_{\mathbf{j}}\in \tilde{A}_k=\sum_{i=0}^{k-1}P^{-1} T^i A=\sum_{i=0}^{k-1}J^iP^{-1}A$ \ by  (\ref{IFS_digits}). 
Hence $\left| (\tilde{a}_{\mathbf{j}})_p \right|\leq c_5m_p \Longrightarrow$

\begin{eqnarray}\label{eqn1}
\left| \lceil C_p^k \rceil ^{-1}  \odot \left(  \sum_{j=1}^{\alpha_0}(-1)^j( N_1-N_2)\sum_{l=0}^{j-1}N_1^lN_2^{j-1-l}\right) (\tilde{a}_{\mathbf{j}})_p \right|  = O(m_p^{-\frac{1}{2}})
\end{eqnarray}

\bigskip

\textit{Upper bound for \ \ \rm{(II)} \ $(C_p^{-k}- \lceil C_p^k \rceil ^{-1}) \odot \left( I+\sum_{j=1}^{\alpha_0}(-1)^jN_1^j\right)(\tilde{a}_{\mathbf{j}})_p$} :
\medskip

\noindent Now, regarding the first term in (\ref{maininequality2}), we have
\begin{equation}\label{second_term}
(C_p^{-k}- \lceil C_p^k \rceil ^{-1}) \odot \left( I+\sum_{j=1}^{\alpha_0}(-1)^jN_1^j\right)=C_p^{-k}\lceil C_p^k \rceil ^{-1}(\lceil C_p^k \rceil -C_p^k) \odot \left( I+\sum_{j=1}^{\alpha_0}(-1)^jN_1^j\right).
\end{equation}
But $\lceil C_p^k \rceil -C_p^k=G_0$, and by (\ref{Power_Jordan_Block2}), we have  $J_{k,p}^{-1}= C_p^{-k} \odot  \left( I+\sum_{j=1}^{\alpha_0}(-1)^jN_1^j\right)$. Then
(\ref{second_term}) implies  $$ \left|(C_p^{-k}- \lceil C_p^k \rceil ^{-1}) \odot \left( I+\sum_{j=1}^{\alpha_0}(-1)^jN_1^j\right)(\tilde{a}_{\mathbf{j}})_p\right|= \left|\lceil C_p^k \rceil ^{-1}G_0 \odot J_{k,p}^{-1}(\tilde{a}_{\mathbf{j}})_p\right|\leq || \lceil C_p^k \rceil ^{-1}G_0||K,$$ where $K$ is the diameter of  $\tilde{F}$. Using  (\ref{estimate00}) and   the inequality $\frac{ m_p}{2}\leq \rho_p^{k}$, we get
 $$ \left|\left|  \left\lceil
C_p^{k}\right\rceil ^{-1}G_0 \right|\right|\leq (\rho_p^{-k}+||G||)||G_0||\leq m_p^{-\frac{1}{2}}c_2$$  for all large $k$. 
Therefore,
\begin{eqnarray}\label{eqn2}
\left|(C_p^{-k}- \lceil C_p^k \rceil ^{-1})  \odot \left( I+\sum_{j=1}^{\alpha_0}(-1)^jN_1^j\right)(\tilde{a}_{\mathbf{j}})_p\right| = O(m_p^{-\frac{1}{2}}).
\end{eqnarray} 
By (\ref{maininequality2}), (\ref{eqn1}), (\ref{eqn2}), we finally 
get  $$\left|(J_{k,p}^{-1}-T_{k,p}^{-1})(\tilde{a}_{\mathbf{j}})_p \right| \leq \frac{c_p}{\sqrt{m_p}}\leq \frac{c_p}{\sqrt{m_1}}$$ for some constant $c_p$.

\medskip

The same argument applies if the algebraic multiplicity of $\lambda_p$ is 1.

\medskip

Case 2. For a real eigenvalue $\lambda_p  $,
we have a similar consideration.  For the proof, one needs to replace $ C_p$ in the non-real-eigenvalue case by
$\lambda_p$,  the perturbed matrix $\left\lceil  b_i C_p^{k} \right\rceil  $   by  $\left\lceil  b_i \lambda_p^{k}\right\rceil  $ and the identity matrix $I_2$ in the nilpotent matrix $N$ by $1$. Then one can modify the above argument appropriately.
\end{pf}

\bigskip

From the proof of the above important lemma, we also get the following norm estimates, which we shall often use.

\begin{lemma}\label{lemma_norm_estimate1}
Let $J_k$ and $T_k$ be the expanding matrices generating the self-affine set $\tilde{F}$ and its perturbations $F_k$ respectively. Let  $r_s=|\lambda_s|^k$, \ 
 $m_s=\lceil r_s \rceil$ 
be their spectral radii  respectively and let $r_1=|\lambda_1|^k<r_2=|\lambda_2|^k<\cdots <r_s$, \  $m_1\leq m_2\leq \cdots \leq  m_s$ as before.
 Then there exist integer constants $c', c'',\alpha_0 >0$ such that
\begin{eqnarray}\label{norm_estimates1}
||J_k^l ||\leq c' k^{ c''} l^{\alpha_0}  r_s^l, \  \ \ \ \  ||T_k^l ||\leq c' k^{ c''} l^{\alpha_0} m_s^l, \  \ \ \ \ 
||J_k^{-l} ||\leq c' k^{ c''} l^{\alpha_0}  r_1^{-l}, \ \ \ \ \   ||T_k^{-l} ||\leq c' k^{ c''} l^{\alpha_0} m_1^{-l} \ \ \ \ 
\end{eqnarray}
for any positive integer $l$ and all large $k.$
In particular,  $\alpha_0=0$  when all the eigenvalues of $J$ have algebraic multiplicity $1$.
\end{lemma}

\begin{pf}
We consider the harder case of non-real eigenvalues $\lambda_p$ ($p=1,2,...,s$) with algebraic multiplicity $n_p\geq 2$.
We have $J_k=J_{k,1}\oplus J_{k,2} \oplus \cdots \oplus J_{k,s}$. 
 Again, it suffices to study the Jordan blocks $J_{k,p}$, $T_{k,p}$. To estimate the norms, recall $T_{k,p}$  and the powers of $J_{\lambda_p}$in (\ref{Power_Jordan_Block1}) and (\ref{Power_Jordan_Block2}) : \quad
$$J_{k,p}= C_p^k    \odot  \left( I+N_1 \right),   \quad \quad T_{k,p}=\lceil C_p^k \rceil   \odot  \left( I+N_2 \right),$$    $$J_{k,p}^{-1}= C_p^{-k} \odot  \left( I+\sum_{j=1}^{\alpha_0}(-1)^jN_1^j\right),\quad \quad T_{k,p}^{-1}=\lceil C_p^k    \rceil ^{-1} \odot \left( I+\sum_{j=1}^{\alpha_0}(-1)^jN_2^j\right),$$

\noindent where $$
N_1=C_p^{-k}\sum_{i=1}^{\alpha_0}
b_i(k)
C_p^{k-i} \odot  N^{i}
=\sum_{i=1}^{\alpha_0}
b_i(k)
C_p^{-i} \odot  N^{i},   \ \ \ \ \ \ \
N_2=\sum_{i=1}^{\alpha_0}
\left\lceil b_i(k)
C_p^{k-i}\right\rceil  \left\lceil
C_p^{k}\right\rceil ^{-1}  \odot  N^{i},$$
$N^{n_p}=N_1^{n_p}=N_2^{n_p}=0$, i.e. they are nilpotent, $I,\ N, \ N_1, \ N_2 $ have size $2n_p\times 2n_p$, $\alpha_0=\alpha_0(n_p)$ is a constant depending on  $n_p$, the $b_i(k)=\binom{k}{i}$ are binomial coefficients  and $ C_p^k $,  $\lceil C_p^k    \rceil$ are rotation matrices with spectral radii $r_p=|\lambda_p|^k=\rho_p^{k}$, \
$m_p$ (for $n_p=1$, the consideration will be similar and easier). Here
$ C_p^k$ corresponds to non-real eigenvalues. For real eigenvalues, 
 $ C_p^k$ or $\lceil C_p^k    \rceil$ is replaced by $r_p$ or $m_p$ and the sizes of $I,\ N_1, \ N_2 $ become $n_p\times n_p$.  It is enough to work with $T_{k,p}$ for $l>0$, the other case is similar. Similar to the notation of a scalar matrix,
we write $diag \left( \lceil  C_p^k \rceil \right) :=\lceil  C_p^k \rceil  I$ ($I$ is the identity matrix). By the definition of $ \odot$ in the last proof, we then notice that,
in fact,
 $$||N_2||=\left| \left| \sum_{i=1}^{\alpha_0} diag \left( \left\lceil b_i(k)
C_p^{k-i}\right\rceil  \left\lceil
C_p^{k}\right\rceil ^{-1}  \right)   N^{i} \right| \right| \leq   \sum_{i=1}^{\alpha_0} \left| \left| diag \left( \left\lceil b_i(k)
C_p^{k-i}\right\rceil  \left\lceil
C_p^{k}\right\rceil ^{-1}  \right) \right| \right|   \left| \left|  N^{i} \right| \right|.$$
Clearly,   $\left\lceil b_i(k)C_p^{k-i} \right\rceil =b_i(k) C_p^{k-i}+G_i$
for some appropriate
matrices $ G_i  $ with $||G_i || \leq 2$.
But
\begin{eqnarray}
  \left| \left| diag \left( \left\lceil b_i(k)
C_p^{k-i}\right\rceil  \left\lceil
C_p^{k}\right\rceil ^{-1}  \right) \right| \right| &\leq & \left| \left|  \left\lceil b_i(k)C_p^{k-i}\right\rceil \right| \right|
\left| \left| \left\lceil
C_p^{k}\right\rceil ^{-1}\right| \right|  \nonumber \\
   & \leq &  \left(\left| \left|  b_i(k) C_p^{k-i}\right| \right|+ ||G_i|| \right)
\left| \left| \left\lceil
C_p^{k}\right\rceil ^{-1}\right| \right| \nonumber \\
   & \leq & \left(b_i(k)|\lambda_p|^{k-i}+2 \right) \left(\rho_p^{-k}+||G|| \right) \nonumber \\ 
   & \leq & \left(b_i(k)|\lambda_p|^{k-i}+2 \right) \left(\rho_p^{-k}+\frac{c}{m_p} \right) \leq b_i(k)+2
\end{eqnarray}
for all large $k$ by (\ref{estimate00}).
Therefore $$i\leq \alpha_0 \Rightarrow |b_i(k)|=\left|\binom{k}{i}\right|\leq k^{\alpha_0}  \Rightarrow \left|\sum_{i=1}^{\alpha_0}b_i(k) \right|\leq \alpha_0 k^{\alpha_0} \Rightarrow ||N_2||\leq  \sum_{i=1}^{\alpha_0} (b_i(k)+2) \leq \alpha_0 k^{\alpha_0}+2\alpha_0\leq c'k^{\alpha_0}$$ for some constant $c'.$ We will use this estimate for the
following. 
$$\left| \left| T_{k,p}^{l} \right| \right| = \left| \left| diag \left( \lceil  C_p^k \rceil \right)^{l}  \left( I+N_2 \right)^{l} \right| \right| \leq  \left| \left| diag \left( \lceil  C_p^k \rceil \right)^{l}\right| \right|   \left| \left|  \left( I+N_2\right)^{l}\right| \right|=\left| \left|  \lceil  C_p^k \rceil ^{l}\right| \right| \cdot  \left| \left|  \left( I+N_2\right)^{l}\right| \right|.  $$
Since $\lceil  C_p^k \rceil $ is a rotation matrix, we have $\left| \left|  \lceil  C_p^k \rceil ^{l}\right| \right| =m_p^l.$  By the binomial expansion,
$\left( I+N_2\right)^{l}= I+\sum_{i=1}^{\alpha_0}b_i(l)N_2^i$ because $N_2$ is nilpotent with $N_2^i=0$ for $i\geq n_p$.
 \begin{equation}\label{T_{k_p}_estimate}
 |b_i(l)|\leq l^{\alpha_0}\Rightarrow \left|\sum_{i=1}^{\alpha_0}b_i(l) \right|\leq \alpha_0 l^{\alpha_0}\Longrightarrow ||\left( I+N_2\right)^{l} ||\leq c' k^{c''} l^{\alpha_0}\Longrightarrow \left| \left| T_{k,p}^{l} \right| \right| \leq c' k^{c''} l^{\alpha_0} m_p^l
 \end{equation}
 for large $c'$. As a result, $ \left| \left| T_{k}^{l} \right| \right|\leq c' k^{c''} l^{\alpha_0} m_p^l\leq  c'  k^{c''} l^{\alpha_0} m_s^l$  if we choose  $c', c''>0$ sufficiently large.
 
 For the last assertion, we observe that  the term $l^{\alpha_0}$ in (\ref{norm_estimates1}) comes from the estimation of the binomial coefficients in $N_1$ and $N_2.$
 If $J$ has only simple eigenvalues, then the nilpotent matrix $N$ will be the zero matrix. Therefore, $N_1$ and $N_2$ will be zero matrices by their definitions. As a result, 
 there will be no contribution from the binomial coefficients to our estimation. Thus $\alpha_0=0$ when $J$ has simple eigenvalues.
\end{pf}

\begin{remark}\label{constant_remark} \rm{ 
When $k$ is fixed or has no contribution to computation, we will regard it as a constant and omit the term $k^{c''}$ above. Also, $c', c''$ depend on the sign of $l$, $-l$ and the matrix $J$ (or $T_k$). Further, ignoring the term $l^{\alpha_0}$ in these estimates will not cause serious problems in our context.} $\hfill \Box$
\end{remark}

Sometimes, it is more appropriate to use cruder and simpler estimates for negative powers like the following.

\begin{lemma}\label{lemma_norm_estimate2}
Let $J_k$ and $T_k$ be  as in Lemma \ref{lemma_norm_estimate1}. Then there exist  constants $c'>0$ and $ \eta>1$ such that
\begin{equation*}\label{norm_estimates2}
|| J_k^{-l}||\leq \frac{c'}{(\eta^k)^l}, \qquad \qquad  ||T_k^{-l} ||\leq \frac{c'}{(\eta^k)^l}  \qquad \qquad l\in \mathbb{N}.
\end{equation*}
\end{lemma}

\begin{pf}
We shall also  use the inequality in \cite[Remark 3.2, p.487]{KK1} on the norm and spectral radius $\lambda(M)$ of a matrix $M$. But we modify it for our purposes. For
any number $\frac{1}{\eta}>\lambda(M)$, there exists a constant $c'$ such that
\begin{equation}\label{spectral}
 ||M^l||\leq c'\frac{1}{\eta^l}
\end{equation} for every $l\in \mathbb{N}.$
In  (\ref{spectral}),  we take $M=J^{-1}$. Choose a number $\eta$ so that $1<\eta  < \sqrt{|\lambda_1|}<|\lambda_1|$,  where $\frac{1}{|\lambda_1|}$ spectral radius of $J^{-1}$.
Since $ \lambda(J^{-1})=\frac{1}{|\lambda_1|}<\frac{1}{\eta}$,  we get
$ || J^{-kl}||\leq \frac{c'}{(\eta^k)^l}$ for every $k,l$ by  (\ref{spectral}). Further, $J_k^{-l}=J^{-kl}$ by definition so that $$ || J_k^{-l}||\leq \frac{c'}{(\eta^k)^l}.$$

Recall that $m_1=\lceil |\lambda_1|^k \rceil$. On the other hand, it follows from the proof of Lemma \ref{lemma_norm_estimate1} that
 $$||T_k^{-l}||\leq  c'  k^{c''} l^{\alpha_0} m_1^{-l}= m_1^{-l/2} (c'  k^{c''} l^{\alpha_0} m_1^{-l/2})\leq m_1^{-l/2} (c'  k^{c''} l^{\alpha_0} |\lambda_1|^{k(-l/2)})$$ for some constants
$c', c'',\alpha_0 $.  Note that $\underset{l\rightarrow \infty}{\lim} (c'  k^{c''} l^{\alpha_0} |\lambda_1|^{k(-l/2)})=0.$  Then $||T_k^{-l}||\leq  c'\frac{1}{(\sqrt{m_1})^{l}}$ for some  $c'.$   In view of  $ 1<\eta^k  < \sqrt{|\lambda_1|^k}\leq \sqrt{m_1},$ we have
  \begin{equation*}\label{Ineq 1}  ||T_k^{-l}||\leq  c'\frac{1}{(\sqrt{m_1})^{l}}\leq \frac{c'}{(\eta^k)^l}.
  \end{equation*}
\end{pf}

\bigskip

\begin{prop}\label{deflection} Let $x\in \tilde{F}=P^{-1}F=F(J_k,\tilde{A}_k)$ and $ x_{k}\in F_{k}=F(T_k,D_k)$ be as in (\ref{special_correspondence}). Then
$\underset{k\rightarrow\infty} {\lim} x_{k}=x$ uniformly (meaning that the convergence does not depend on x) and hence  $F_{k} \rightarrow \tilde{F} $ in the Hausdorff metric.
\end{prop}

\begin{pf} 
 Assume that $F_{k}=F(T_k,D_k)$ be the perturbations of $\tilde{F}$ so
that $T_{k}$ has eigenvalues with absolute values
$$1<m_1\leq m_2\leq \cdots \leq  m_s.$$ 
Consider the
series expansions $x=\sum_{i=1}^\infty J_{k}^{-i} \tilde{a}_{{\mathbf{j}_i}}\in \tilde{F}=F(J_k,\tilde{A}_k)$, \ $x_{k}=\sum_{i=1}^\infty T_{k}^{-i} d_{{\mathbf{j}_i}}\in F_{k}$.
We may assume that $  \tilde{a}_{{\mathbf{j}_i}}, \ d_{{\mathbf{j}_i}} \geq 0$  with  $d_{{\mathbf{j}_i}}=\lfloor \tilde{a}_{{\mathbf{j}_i}} \rfloor$. Then
\begin{equation}\label{special_convergence}
x-x_{k}=\sum_{i=1}^\infty (J_{k}^{-i}-T_{k}^{-i}) \tilde{a}_{{\mathbf{j}_i}}+\sum_{i=1}^\infty T_{k}^{-i} (\tilde{a}_{{\mathbf{j}_i}}-d_{{\mathbf{j}_i}}).
\end{equation}
 We now find an upper bound for the norm of the second sum in (\ref{special_convergence}). 
By Lemma \ref{lemma_norm_estimate2}, we have
  \begin{equation}\label{Ineq 2}  ||J_k^{-i}||\leq \frac{c'}{(\eta^k)^i}, \quad \quad \quad \quad \quad  ||T_k^{-i}||\leq \frac{c'}{(\eta^k)^i}.\end{equation} 
These inequalities imply
\begin{equation}\label{Ineq 3} \left|  \sum_{i=1}^\infty T_{k}^{-i} (\tilde{a}_{{\mathbf{j}_i}}-d_{{\mathbf{j}_i}}) \right|\leq   \sum_{i=1}^\infty ||T_{k}^{-i}||\cdot |\tilde{a}_{{\mathbf{j}_i}}-d_{{\mathbf{j}_i}}| \leq \sum_{i=1}^\infty c'\frac{n}{(\eta^k)^i} =c'
\frac{n}{\eta^k-1}.
\end{equation}

\noindent  As for the first sum $\sum_{i=1}^\infty (J_{k}^{-i}-T_{k}^{-i}) \tilde{a}_{{\mathbf{j}_i}}$ in (\ref{special_convergence}), we need an upper bound for its norm.
In Lemma \ref{deflection0}, we saw that
$$|T_k^{-1}\tilde{a}_{{\mathbf{j}}}|-|J_k^{-1}\tilde{a}_{{\mathbf{j}}}|\leq |(J_k^{-1} - T_k^{-1})\tilde{a}_{{\mathbf{j}}}|\leq \frac{c}{\sqrt{m_1}}\Longrightarrow $$
\begin{equation}\label{Ineq 4}|T_k^{-1}\tilde{a}_{{\mathbf{j}}}|\leq  \frac{c}{\sqrt{m_1}}+|J_k^{-1}\tilde{a}_{{\mathbf{j}}}|,  \quad \quad \tilde{a}_{{\mathbf{j}}}\in \tilde{A}_k. 
\end{equation}
Now we rewrite 
$\sum_{i=1}^\infty (J_{k}^{-i}-T_{k}^{-i}) \tilde{a}_{{\mathbf{j}_i}}$ as 
$$\sum_{i=1}^\infty (J_{k}^{-i}-T_{k}^{-i}) \tilde{a}_{{\mathbf{j}_i}}=(J_{k}^{-1}-T_{k}^{-1})\tilde{a}_{{\mathbf{j}_1}} +\sum_{i=2}^\infty (J_{k}^{-i}-T_{k}^{-i}) \tilde{a}_{{\mathbf{j}_i}}.$$ 
Note that $J_{k}^{-1}\tilde{a}_{{\mathbf{j}_i}}\in \tilde{F}\Longrightarrow |J_{k}^{-1}\tilde{a}_{{\mathbf{j}_i}}|\leq diam(\tilde{F})$. Let $\textsf{K}$ be a number such that $diam(\tilde{F})\leq \textsf{K}.$ Then
\begin{eqnarray}\label{estimate3}
\left| \sum_{i=1}^\infty (J_{k}^{-i}-T_{k}^{-i}) \tilde{a}_{{\mathbf{j}_i}}\right| & \leq &  \left| (J_{k}^{-1}-T_{k}^{-1})\tilde{a}_{{\mathbf{j}_1}} \right|+\sum_{i=2}^\infty \left( |J_{k}^{-i}\tilde{a}_{{\mathbf{j}_i}}|+|T_{k}^{-i} \tilde{a}_{{\mathbf{j}_i}}| \right) \nonumber \\
& \leq & \frac{c}{\sqrt{m_1}}+\sum_{i=2}^\infty \left( ||J_{k}^{-i+1}||\cdot |J_{k}^{-1}\tilde{a}_{{\mathbf{j}_i}}|+||T_{k}^{-i+1}||\cdot |T_{k}^{-1}\tilde{a}_{{\mathbf{j}_i}}| \right) \nonumber \\
& \leq & \frac{c}{\sqrt{m_1}}+c'\sum_{i=2}^\infty  (\eta^k)^{-i+1} \left(\frac{c}{\sqrt{m_1}}+2\textsf{K}  \right) \quad \quad \quad  \mathrm{by} \ \ (\ref{Ineq 2}), \ (\ref{Ineq 4})  \nonumber \\
&   \leq & \frac{c}{\sqrt{m_1}}+c'\frac{\frac{c}{\sqrt{m_1}}+2\textsf{K}}{\eta^k-1}.
\end{eqnarray}
Consequently, 
\begin{eqnarray}\label{estimatex_k}
|x-x_{k}|\leq\left| \sum_{i=1}^\infty (J_{k}^{-i}-T_{k}^{-i}) \tilde{a}_{{\mathbf{j}_i}} \right|+\left| \sum_{i=1}^\infty T_{k}^{-i} (\tilde{a}_{{\mathbf{j}_i}}-d_{{\mathbf{j}_i}})\right| \leq \frac{c}{\sqrt{m_1}}+c'\frac{\frac{c}{\sqrt{m_1}}+2\textsf{K}+n}{\eta^k-1}
\end{eqnarray}
by  (\ref{special_convergence}), (\ref{Ineq 3}), (\ref{estimate3}).
This gives $\underset{k\rightarrow\infty} {\lim} x_{k}=x$
because $\underset{k\rightarrow\infty}{\lim} \eta^k = 
\infty $ in view of the inequalities 
$1<\eta<\sqrt{|\lambda_1|}<|\lambda_1|$ and $\eta^k<\sqrt{|\lambda_1|^k}\leq  \sqrt{m_1}$ (see the proof of Lemma \ref{lemma_norm_estimate2}).

From (\ref{estimatex_k}), we deduce that
$|x-x_{k}|\leq \frac{c''}{\eta^k-1},$ where 
$c''$ is a constant independent of $k$ and $x\in \tilde{F}$.
Let  $d_H$ be the Hausdorff metric on non-empty compact sets.  One can now see that
$d_H(\tilde{F},F_{k})\leq \sup_{x\in \tilde{F}}|x-x_{k}|\leq \frac{c''}{\eta^k-1}$ and $\underset{k\rightarrow\infty} {\lim} d_H(\tilde{F},F_{k})=0.$
\end{pf}

\bigskip

\bigskip

\begin{remark}\label{rem2}  {\rm \begin{itemize}
                                   \item[(a)]  Another version of the above proposition is given below.  Unlike Proposition \ref{deflection}, we assume that the 
multi-indices ${\mathbf{j}}_i$ and ${\mathbf{j}'_i}$
  are not necessarily equal except for the first $\textsf{r} \leq k$ indices. That is, $${\mathbf{j}}_1=j_1j_2...j_{\textsf{r}}j_{\textsf{r}+1}...j_k, \quad \quad \quad \quad {\mathbf{j}}'_1=j'_1j'_2...j'_{\textsf{r}}j'_{\textsf{r} + 1}...j'_k=j_1j_2...j_{\textsf{r}}j'_{\textsf{r} + 1}...j'_k$$ when $\textsf{r} < k$ and ${\mathbf{j}}_1={\mathbf{j}}'_1$ when $\textsf{r} = k$.  
For each  positive  integer $\textsf{r}$, we consider any integer 
   $k=k_{\textsf{r}}\geq \textsf{r}$ and points in the form of 
\begin{equation}\label{special_correspondence-1}
x_{\textsf{r}}=\sum_{i=1}^\infty J_{k}^{-i} \tilde{a}_{{\mathbf{j}}_i}, \ \ \ \ \ \ \ \ \ \ \ \ y_{_{k_{\textsf{r}}}}=\sum_{i=1}^\infty T_{k}^{-i} d_{{\mathbf{j}'_i}},
\end{equation}
where $\tilde{a}_{{\mathbf{j}}_i}\in \tilde{A}_k$, \  $d_{\mathbf{j}'_i}\in D_{k}$ and ${\mathbf{j}}_1, {\mathbf{j}}'_1$ are as above.  Thus $j_1j_2...j_\textsf{r}=j'_1j'_2...j'_\textsf{r}$, \ \ $\tilde{a}_{{\mathbf{j}}_1}=\sum_{i=1}^k J^{k-i} \tilde{a}_{j_i}$ and 
$d_{{\mathbf{j}}'_1}=\lfloor  \tilde{a}_{{\mathbf{j}}'_1} \rfloor=\lfloor  \sum_{i=1}^k J^{k-i} \tilde{a}_{j'_i}\rfloor$ by the definition of $d_{{\mathbf{j}}'_i}$. But  $d_{{\mathbf{j}}'_i}=\lfloor  \tilde{a}_{\mathbf{j}'_i} \rfloor$ may not be equal to 
$d_{{\mathbf{j}}_i}=\lfloor  \tilde{a}_{\mathbf{j}_i} \rfloor
$ for $i\geq 2$.   
By necessary adjustments in the above proof, we can prove the following.

\bigskip

{\it Let $x_{\textsf{r}}\in \tilde{F}=F(J_k,\tilde{A}_k)$ and
$  y_{_{k_{\textsf{r}}}}\in F_{k}=F(T_k,D_k)$ be as in (\ref{special_correspondence-1}).
Then}
$\underset{r\rightarrow\infty} {\lim} |x_{\textsf{r}}-y_{_{k_{\textsf{r}}}}|=0$ uniformly (meaning that the convergence is independent of the points $x_{\textsf{r}}, y_{_{k_{\textsf{r}}}}$). 
                                   \item[(b)]   We have also the following variants of the above estimates.
                                   By using Lemma \ref{lemma_norm_estimate1}, one may  obtain the inequalities
 $$|x-x_k|\leq c' k^{ c''}  |\lambda_1|^{-k}, \quad  \quad \quad |x_{\textsf{r}}-y_{_{k_{\textsf{r}}}}|\leq c'  {\textsf{r}}^{c''}  |\lambda_1|^{-\textsf{r}}$$ 
  for some constants $c',c''$. $\hfill \Box$
                                 \end{itemize} } 
\end{remark}

\section{Basic Results }\label{Stability}

We have introduced the sequence $\left\{ F_k=F(T_k,D_k) \right\}$ of perturbed fractals in Section \ref{pert_defl} to deal with the case where the characteristic polynomial of $T$, $f(x)\in \Bbb{Z}[x]$, has eigenvalues with irrational absolute values.

Recall that $F(T,A)=\left\{\sum_{i=1}^\infty T^{-i} a_{j_i} : \ a_{j_i}\in A \right\}$. Alternatively, we may use the sequence notation $x=(a_{j_i})_{i=1}^\infty$ or $x=(a_{j_i})$ for $x=\sum_{i=1}^\infty T^{-i} a_{j_i}.$
In most of the proofs concerning dimension, we shall use covers by $n$-cubes. More explicitly, we can take $m$ to be  $m=\lfloor \lambda \rfloor$, where $\lambda$ is the minimum of the absolute values of the eigenvalues of $T$ and cover $\tilde{F}$ by n-cubes 
of side length $c_0m^{-r}$ ($c_0$ is a constant, but $r\in \mathbb{N}$ may vary). Moreover, we may always assume that $2\leq m$ because  $F=F(T^k,A_k)$ for all $k\in \mathbb{N}$ and the absolute values of the eigenvalues of $T^k$ are $\geq 2$ for all sufficiently large $k$.

In what follows, we write $diam(S)$ for the diameter of a set $S$. 

\begin{defn}\label{piece} {\rm Assume that $F=F(T,A)\subset \mathbb{R}^n$ is an integral self-affine set and  $l$  is a positive integer.

(i) If $T$ is a similarity or more generally, all of its eigenvalues have the same modulus, 
then by a \textit{ piece of $F$ of level $l$ (or level-l piece of $F$)}, we mean the set of points $x=(a_{j_i})\in F$ such that the first $l$ digits, $a_{j_i},...,a_{j_l}$, are fixed. That is, it is a set
of the form $P_l(F)=\sum_{i=1}^l T^{-i} a_{j_i}+T^{-l} F$.

(ii) If $J_k$ is not a similarity, we now consider $\tilde{F}=P^{-1}F=F(J_k,\tilde{A}_k)$ ($J_k$ is the Jordan form of $T^k$) or the perturbed fractal $F_k=F(T_k,D_k)$ for each $k$. Let
$J_{\lambda_p}$ ($p=1,2,...,s$) be the Jordan blocks of $J$ of size $n_p\times n_p$ for real eigenvalues $\lambda_p$ or $2n_p\times 2n_p$ for non-real eigenvalues. We shortly write  $J_{k,p}$ for $J_{\lambda_p}^k$. Assume that $J_k=diag(J_{k,1},J_{k,2},... ,J_{k,s})$ and $J_{k,p}$ has  eigenvalues of moduli $r_p=|\lambda_p^k|$ with $1<r_1< r_2 <\cdots <  r_s$.
Let the corresponding eigenvalues of $T_k$ have moduli 
$m_p$ with
$1 < m_1 \leq m_2 \leq \cdots  \leq m_s$. So $r_p$ and $m_p$ depend on $k$. For a given positive integer $l=l_1$, inductively define the integers $$l_{p+1}=\lfloor l_p\log_{m_{p+1}} m_p \rfloor$$ for $F_k$, or if $l=\tilde{l}_1$, define $$\tilde{l}_{p+1}=\lfloor \tilde{l}_p\log_{r_{p+1}} r_p \rfloor=\left\lfloor \tilde{l}_p \frac{\log{|\lambda_p}|}{\log{|\lambda_{p+1}}|}  \right\rfloor$$ for $\tilde{F}$, where $p=1,2,..,s-1$. Then $l_s \leq \cdots \leq l_2   \leq  l_1$ and $\tilde{l}_s \leq \cdots \leq \tilde{l}_2   \leq  \tilde{l}_1$.

We next consider the coordinates corresponding to Jordan blocks. Denote by $d_{{\mathbf{j}_i}}(\beta)$ ($\beta=1,2,...,n$) the $\beta$-th coordinate of a digit $d_{{\mathbf{j}_i}}\in D_k.$ Then by a \textit{ piece of $F_k$ of level $l=l_1$}, we mean a set 
$P_l(F_k)$ of points $x_k=(d_{{\mathbf{j}_i}})\in F_k$ with $d_{\mathbf{j}_1}(\beta),...,d_{\mathbf{j}_{l_p}}(\beta)$ are fixed for $1+\sum_{t=1}^{p}n_{t-1}\leq \beta \leq \sum_{t=1}^{{p}}n_{t}$  \ \ ($n_{0}=0$, $p=1,2,..,s$). Note that $\sum_{t=1}^{{s}}n_{t}=n$, the space dimension.

Let $(d_{{\mathbf{j}_i}})_{p} $ denote projection of  $d_{{\mathbf{j}_i}}$ to the coordinates $1+\sum_{t=1}^{p}n_{t-1},..., \sum_{t=1}^{{p}}n_{t}$. Thus 
$(d_{{\mathbf{j}_i}})_{p} $ corresponds to $p$-th block  $T_{k,p}$ of $T_k$. Then for a piece and its $p$-th block of $P_l(F_k)$, we have the following inclusions in terms of cylinder sets. 
$$\sum_{i=1}^{l_{1}} T_{k}^{-i} d_{{\mathbf{j}_i}}  + T_{k}^{-l_{1}} F_k \subseteq P_l(F_k) \subseteq \sum_{i=1}^{l_{s}} T_{k}^{-i} d_{{\mathbf{j}_i}}  + T_{k}^{-l_{s}} F_k$$
$$\sum_{i=1}^{l_{1}} T_{k,p}^{-i} (d_{{\mathbf{j}_i}})_{p}  + T_{k,p}^{-l_{p}} \left(T_{k,p}^{-(l_{1}-l_{p})}(F_k)_{p}\right)  \subseteq (P_l(F_k))_{p} \subseteq \sum_{i=1}^{l_{p}} T_{k,p}^{-i} (d_{{\mathbf{j}_i}})_{p}  + T_{k,p}^{-l_{p}} (F_k)_{p}$$ and $P_l(F_k)$ is a finite union of the so-called cylinder sets $\sum_{i=1}^{l_{1}} T_{k}^{-i} d_{{\mathbf{j}_i}}  + T_{k}^{-l_{1}} F_k \subseteq P_l(F_k)$ for certain $d_{{\mathbf{j}_i}}\in D_k$ ($1\leq i \leq l_1$) because we have some restrictions on $d_{{\mathbf{j}_i}}$ (i.e., $d_{\mathbf{j}_1}(\beta),...,d_{\mathbf{j}_{l_p}}(\beta)$ are fixed for the coordinates $\beta$ corresponding the $p$-th Jordan block) in this definition. Note that 
$T_{k,p}^{-(l_{1}-l_{p})}(F_k)_{p}\subseteq (F_k)_{p}.$
Then these inequalities will allow 
us to \underline{abuse the notation} and say that the $p$-th block of $P_l(F_k)$ is in the form of 
\begin{equation}\label{sloppy_notation}
 (P_l(F_k))_{p}=\sum_{i=1}^{l_{p}} T_{k,p}^{-i} (d_{{\mathbf{j}_i}})_{p}  + T_{k,p}^{-l_{p}} (F_k)_{p}
\end{equation} as if we were in case (i).
Because that does not affect our proofs negatively, however it is not immediate to see that (cf. Remark \ref{Actual_Case1}-(ii)). 
The definition for $P_l(\tilde{F})$ is similar.

(iii) 
We say that $P_l(F_k)$ or $P_l(\tilde{F})$ is \textit{commensurable} with an $n$-cube of side length $c_0m^{-r}$ ($c_0, m$ are natural numbers) if
$$ c_0m^{-r} < diam (P_l(F_k)) \leq c_0m^{-r+1}  \ \ {\rm or} \ \ c_0m^{-r} < diam (P_l(\tilde{F})) \leq c_0m^{-r+1}.$$ Setting $c_1=c_0$, \ $c_2=mc_0$, this can be written as
$$ c_1 m^{-r} < diam (P_l(F_k)) \leq  c_2 m^{-r} \ \ {\rm or} \ \  c_1 m^{-r} < diam (P_{l}(\tilde{F})) \leq  c_2 m^{-r}. \quad \hfill \Box $$ }
\end{defn}

\begin{remark}\label {piece_notation} \rm{ (a) When $T$ is a diagonal matrix, $l=r$ and $A$ has some special form, a piece $P_l(F_k)$ is also known as an ``approximate $n$-cube'' of diameter $\approx m_1^{-l}$
(i.e. $c' m_1^{-l}\leq diam( P_l(F_k))\leq c'' m_1^{-l}$ for some constants $c',c''>0$) \cite[ p.3]{M}, \cite{KP1}.
The non-diagonal case will be studied in Lemma \ref{inequality_pieces} below.

(b) The large values of $l$ are relevant in our study. Therefore, we may only consider the case $1\leq l_s\leq \cdots \leq l_2   \leq  l_1.$

(c) It is not hard to see that every piece is a finite union of smaller pieces. 

(d) A level-$l$ piece of $\tilde{F}=F(J_{k_1}, \tilde{A}_{k_1})$ and  a level-$l$ piece of $\tilde{F}=F(J_{k_2}, \tilde{A}_{k_2})$ are not necessarily the same for $k_1\neq k_2$ although 
$\tilde{F}=F(J_{k_1}, \tilde{A}_{k_1})=F(J_{k_2}, \tilde{A}_{k_2})$. That is because the definition of a piece  heavily depends on generating matrix $J_{k}$ and the digit set $ \tilde{A}_{k}$, and 
 matrices  $J_{k_1},\ J_{k_2} $ or digit sets $ \tilde{A}_{k_1},\  \tilde{A}_{k_2}$ are different in general. 
To be more specific, one might use the notation $P_{l,k}(\tilde{F})$ for the pieces of $\tilde{F}$ with respect to the generating data $J_{k},\tilde{A}_{k}$.}  $\hfill\Box$
\end{remark}

\noindent In our dimension consideration, 
\begin{enumerate}
  \item pieces will be our main objects,
  \item the same-level pieces of $\tilde{F}$ and $F_k$ will be closely related for large $k$,
  \item the number of overlapping pieces commensurable with an $n$-cube will not be large enough to affect our proofs. Thus overlaps may be ignored. 
\end{enumerate}

\bigskip
Clearly,  $l_p$ and $\tilde{l}_p$ depend on $k$, $l_1$, $\tilde{l}_1$  by definition. We first see how $l_p$ and $\tilde{l}_p$ are related to each other when $l_1(k)=\tilde{l}_1(k)$, i.e 
when $P_{\tilde{l}_1}(\tilde{F})$ and $P_{l_1}(F_k)$ are of the same level. To put it differently, we are interested in the asymptotic behavior of $l_p$ and $\tilde{l}_p$ 
when $l_1(k)=\tilde{l}_1(k)$ for all $k.$ The following is used for diameter estimates and  piece counting process with respect to the defining digits  of pieces or the indices of those digits. 

\begin{lemma}\label{number_of_pieces} Let $l_p=l_p(k)$ and $\tilde{l}_p=\tilde{l}_p(k)$ ($p=1,2,...,s$) be defined as in  Definition \ref{piece}.  If $l=l_1(k)=\tilde{l}_1(k)$  for each $k$, then there is a positive integer $k_0$ (independent of $l$) 
such that  $l_p=\tilde{l}_p-\iota_p\leq \tilde{l}_p$  with 
$\iota_p\leq p-1\leq s-1$
for $k\geq k_0$. 
\end{lemma}

\begin{pf} 
 Note that, given $l=l_1$, we have inductively defined $$l_{p+1}=\lfloor l_p\log_{m_{p+1}} m_p \rfloor\ \ \ \ \ (p=1,2,...,s-1)  $$
 for each $F_k$ in Definition \ref{piece} (see also   (\ref{pert_matrix})). Hence $l_p$ depends on $l_1$ and $k.$ The proof will be inductive too. 
Remember that $l_1(k)=l$ for every $k$ by hypothesis. If $l_p$ 
has the same value for each $k$, then $l_{p+1}$ will only depend on $m_p,\ m_{p+1}$ and hence on $k$.   Writing $l_p=l_p(k)$ in that case, we have $\lim_{k\rightarrow \infty} l_1(k)=\lim_{k\rightarrow \infty} \tilde{l}_1(k)=l$.  
 In fact,  for all large $k$, we will assume that 

\medskip
(1) $\tilde{l}_p(k)$ has constant values because $\tilde{l}_1(k)$ is constant by hypothesis,

(2) $ l_p=\tilde{l}_p-\iota_p $ for a nonnegative integer $\iota_p=\iota_p(k)\leq p-1$

\medskip
\noindent and show that 
$l_{p+1}=\tilde{l}_{p+1}-\iota_{p+1} $ for a nonnegative integer $\iota_{p+1}\leq p$ again. 
Here $\iota_p(k)$ depends on $k$, but we assume that $\iota_p(k)\leq p-1$ by the induction hypothesis. Therefore, without loss of generality, we may assume that the sequence 
$\{\iota_p(k) \}$ is eventually constant because we can partition it into a finite number of eventually constant sequences. In other words, $\{\iota_p(k) \}$ is a finite disjoint union of  eventually constant sequences and we may work with such sequences. Since $\tilde{l}_p(k)$ is already constant by (1),  $ l_p=\tilde{l}_p-\iota_p $ will be eventually  constant for those sequences.

Remember that $r_p=|\lambda_p|^k, \ r_{p+1}=|\lambda_{p+1}|^k$ are the moduli of the eigenvalues of the diagonal blocks of $J_k$ with $r_1 < r_2 < \cdots < r_s$. Also 
in the case of real eigenvalues,  
$m_p=\lceil r_p\rceil$ by definition. 
In the case of non-real eigenvalues, recall 
from  (\ref{estimate01})
 that $\lceil C_p^{k} \rceil  = C_p^{k}+G_0$
for some appropriate
matrix $G_0 $ with $||G_0|| \leq 2$. Thus $m_p=||\lceil C_p^{k} \rceil ||\leq || C_p^{k} ||+2=r_p+2$. In any case, we have $m_p \leq 2r_p$ for large $k$. This allows us to give the proof only for real eigenvalues (the proof in the other case is identical). Since $1<r_p< r_{p+1}$, we have
$$\lim_{k\rightarrow \infty}   \frac{\log{ (|\lambda_p|^k)}}{\log{ (2|\lambda_{p+1}|^{k})}}=\lim_{k\rightarrow \infty}   \frac{\log{ (r_p)}}{\log{ (2r_{p+1})}}\leq \lim_{k\rightarrow \infty}   \frac{\log{\lceil r_p\rceil}}{\log{\lceil r_{p+1}\rceil}}\leq \lim_{k\rightarrow \infty}  \frac{\log{ (2r_p)}}{\log{ (r_{p+1})}}=\lim_{k\rightarrow \infty}  \frac{\log{ (2|\lambda_p|^{k})}}{\log{ (|\lambda_{p+1}|^k)}}.$$
But the first and the last limits exist and are equal. That is,  

$$\lim_{k\rightarrow \infty}  \frac{\log{ (|\lambda_p|^k)}}{\log{ (2|\lambda_{p+1}|^{k})}}=\lim_{k\rightarrow \infty}   \frac{\log{ (2|\lambda_p|^{k})}}{\log{ (|\lambda_{p+1}|^k)}}=\lim_{k\rightarrow \infty} \frac{\log |\lambda_p|^k}{\log|\lambda_{p+1}|^k}=\frac{\log |\lambda_p|}{\log|\lambda_{p+1}|}
\Longrightarrow
$$
$$ \lim_{k\rightarrow \infty} l_p  \frac{\log{ m_p}}{\log{m_{p+1}}}=\lim_{k\rightarrow \infty} l_p  \frac{\log{\lceil r_p\rceil}}{\log{\lceil r_{p+1}\rceil}}=\lim_{k\rightarrow \infty} l_p\frac{\log{ r_p}}{\log{ r_{p+1}}}=\lim_{k\rightarrow \infty} l_p\frac{\log |\lambda_p|}{\log|\lambda_{p+1}|}.$$
Now assuming that $\iota_p(k)$  is eventually constant,  for large $k$, 
we  may write $ l_p=\tilde{l}_p-\iota_p $ for a nonnegative integer $\iota_p\leq p-1$ and $\iota_p$, $\tilde{l}_p$ have  (eventually) constant values 
by the induction hypothesis. So is $l_p$.
\noindent Hence $\lim_{k\rightarrow \infty} l_p=\tilde{l}_p-\iota_p$, i.e. the limit exists, and 

\begin{equation}\label{Limit1}
 \lim_{k\rightarrow \infty} l_p  \frac{\log{\lceil r_p\rceil}}{\log{\lceil r_{p+1}\rceil}}=\lim_{k\rightarrow \infty} l_p \frac{\log{ r_p}}{\log{ r_{p+1}}}=\frac{\log |\lambda_p|}{\log|\lambda_{p+1}|}(\tilde{l}_p-\iota_p). 
\end{equation}  
Since $1<|\lambda_p|\leq |\lambda_{p+1}|$
, we have $0\leq\iota_p\frac{\log{|\lambda_p|}}{\log{|\lambda_{p+1}|}}\leq \iota_p $
 .
Then $$\tilde{l}_{p+1}(k)-\iota_p-1 < \left\lfloor\tilde{l}_p \frac{\log |\lambda_p|}{\log|\lambda_{p+1}|}-\iota_p\right\rfloor \leq \left\lfloor(\tilde{l}_p-\iota_p)  \frac{\log |\lambda_p|}{\log|\lambda_{p+1}|} \right\rfloor \leq \left\lfloor\tilde{l}_p \frac{\log |\lambda_p|}{\log|\lambda_{p+1}|} \right\rfloor=\tilde{l}_{p+1}(k).$$
Choose an $\epsilon>0$ less than $1$ so that $\tilde{l}_{p+1}(k)-\iota_p-1 < (\tilde{l}_p-\iota_p)  \frac{\log |\lambda_p|}{\log|\lambda_{p+1}|} -\epsilon$. By 
(\ref{Limit1}),  there is an integer 
 $k_0(p)$ such that
$$  (\tilde{l}_p-\iota_p)  \frac{\log |\lambda_p|}{\log|\lambda_{p+1}|} -\epsilon<l_p\frac{\log{\lceil r_p\rceil}}{\log{\lceil r_{p+1}\rceil}}<  (\tilde{l}_p-\iota_p)  \frac{\log |\lambda_p|}{\log|\lambda_{p+1}|} +\epsilon$$
for $k\geq k_0(p)$. This yields 
$$\tilde{l}_{p+1}(k)-(\iota_p+1) < (\tilde{l}_p-\iota_p)  \frac{\log |\lambda_p|}{\log|\lambda_{p+1}|} -\epsilon <l_p\frac{\log{\lceil r_p\rceil}}{\log{\lceil r_{p+1}\rceil}}< (\tilde{l}_p-\iota_p)  \frac{\log |\lambda_p|}{\log|\lambda_{p+1}|} +\epsilon< \tilde{l}_p  \frac{\log |\lambda_p|}{\log|\lambda_{p+1}|} +1.$$
Then, for a nonnegative integer $\iota_{p+1}\leq \iota_p+1\leq p$, we obtain  $$l_{p+1}(k)= \left\lfloor l_p\frac{\log{m_p}}{\log{m_{p+1}}} \right\rfloor= \left\lfloor l_p \frac{\log{\lceil r_p\rceil}}{\log{\lceil r_{p+1}\rceil}} \right\rfloor =\tilde{l}_{p+1}(k)-\iota_{p+1}$$ for $k\geq k_0(p)$. 
Since $l_p$ is eventually constant,  the integer $k_0(p)$ obtained from the limit (\ref{Limit1}) is independent of $l$.
Now set $k_0=\max \{k_0(p) : p=1,2,...,s-1 \}.$
\end{pf}

\bigskip
Next, we estimate the diameter of a piece $P_l(\tilde{F})$
to be used in some proofs. One may do the same for $diam(P_l(F_k))$ or related sets.

\begin{lemma}\label{inequality_pieces} Let $\tilde{F}=F(J_{k}, \tilde{A}_{k})$ and $P_l(\tilde{F})$ be a piece of level $l$.
Then
$$c' l^{-\alpha_0} (|\lambda_1|^{k})^{-l}\leq diam(P_l(\tilde{F})) \leq c'' l^{\alpha_0} (|\lambda_1|^{k})^{-l}$$ for suitable constants $c'(k), \ c''(k), \ \alpha_0>0.$
\end{lemma}

\begin{pf}
A piece $P_l(\tilde{F})$ has blocks $(P_l(\tilde{F}))_{p}$ 
corresponding to each Jordan  
block $J_{k,p}$, \ $p=1,2,...,s.$ So $(P_l(\tilde{F}))_{p}$ is the projection of $P_l(\tilde{F})$ to the coordinates corresponding to  $J_{k,p}$. We have mentioned the following 
inclusions in  Definition \ref{piece}-(ii). 
$$
\sum_{i=1}^{\tilde{l}_{1}} J_{k}^{-i} a_{{\mathbf{j}_i}}  + J_{k}^{-\tilde{l}_{1}} \tilde{F} \subseteq P_l(\tilde{F}), \quad \quad \quad \quad \quad \quad 
(P_l(\tilde{F}))_{p} \subseteq \sum_{i=1}^{\tilde{l}_{p}} J_{k,p}^{-i} (a_{{\mathbf{j}_i}})_{p}  + J_{k,p}^{-\tilde{l}_{p}} (\tilde{F})_{p}.$$ 
By  \cite[Theorem 2, p. 370]{L},  from the first inequality above, we get 
$$\frac{diam\left((\tilde{F})_{1}\right)}{||J_{k,1}^{\tilde{l}_{1}}  ||} \leq diam( (P_l(\tilde{F}))_{1}) \leq diam( P_l(\tilde{F})).$$
From the second inequality above, we have
\begin{equation}\label{diam_estimate} 
 diam( (P_l(\tilde{F}))_{p})\leq  ||J_{k,p}^{-\tilde{l}_{p}} || \cdot diam((\tilde{F})_{p}) 
\leq ||J_{k,p}^{-\tilde{l}_{p}} || \cdot \underset{1\leq p\leq s}{\max} diam((\tilde{F})_{p}).
\end{equation}
Of course, we  assume that $\underset{1\leq p\leq s}{\min} diam((\tilde{F})_{p})\neq 0$.  Using 
Remark \ref{constant_remark} and 
replacing $T_{k,p}$ by $J_{k,p}$, $m_p^l$ by $r_p^{\tilde{l}_{p}}$ 
in
(\ref{T_{k_p}_estimate}) (the situation for $J^{-\tilde{l}_{p}}_{k,p}$ is similar), one can see that there exist positive integer  constants $c', \ c''$ (depending on $k$)  such that
\begin{equation}\label{diam_estimate2}
||J_{k,p}^{\tilde{l}_{p}} || \leq c' \tilde{l}_{p}^{\alpha_0} (|\lambda_p|^{k})^{\tilde{l}_{p}},  \quad \quad \quad         ||J_{k,p}^{-\tilde{l}_{p}} ||\leq c'' \tilde{l}_{p}^{\alpha_0} (|\lambda_p|^{k})^{-\tilde{l}_{p}}.
\end{equation}
 From the definition of ${\tilde{l}}_{p}$ in Definition \ref{piece}, we have $\tilde{l}_p\leq \cdots\leq \tilde{l}_2\leq \tilde{l}_1$ and  $$ |\lambda_p|^{\tilde{l}_p}\leq \cdots \leq |\lambda_2|^{\tilde{l}_2}\leq |\lambda_1|^{\tilde{l}_1}< |\lambda_2|^{\tilde{l}_2} |\lambda_2|< |\lambda_3|^{\tilde{l}_3}|\lambda_3| |\lambda_2|< \cdots <|\lambda_p|^{\tilde{l}_p}\left(|\lambda_p|\cdots |\lambda_3| |\lambda_2|\right) \leq |\lambda_p|^{\tilde{l}_p}\det T.$$ 
 Then it is clear that
$|\lambda_p|^{\tilde{l}_p}\approx |\lambda_1|^{\tilde{l}_1}$, i.e. 
\begin{equation}\label{diam_estimate3} \left(\det T \right)^{-1}|\lambda_1|^{\tilde{l}_1}\leq
|\lambda_p|^{\tilde{l}_p}\leq |\lambda_1|^{\tilde{l}_1}.
\end{equation}
Writing $l=\tilde{l}_1$,  
 we then obtain from (\ref{diam_estimate}), (\ref{diam_estimate2}) (for $p=1$)
and 
 (\ref{diam_estimate3})
  the following estimate
\footnotesize{ \begin{equation}\label{diam_estimate4} c' l^{-\alpha_0} (|\lambda_1|^{k})^{-l}\leq diam((P_l(\tilde{F}))_{1})\leq diam(P_l(\tilde{F})) \leq \overset{s}{\underset{p=1}{\Sigma}} diam((P_l(\tilde{F}))_{p})  \leq c'' \overset{s}{\underset{p=1}{\Sigma}} \tilde{l}_{p}^{\alpha_0} (|\lambda_p|^{k})^{-\tilde{l}_{p}} \leq c'' l^{\alpha_0} (|\lambda_1|^{k})^{-l}
\end{equation}}
\normalsize for suitably modified $c', \ c''.$
\end{pf}

\bigskip
 The following lemma is crucial for our proofs. It enables us to handle some proofs with pieces. If we use the sloppy notation of (\ref{sloppy_notation}) for simplicity, the following holds.

\begin{lemma}\label{bound} Assume that  $F, \tilde{F}, F_k$ are as in Definition \ref{piece}. Let $\textsf{C}$ be an n-cube of side length $c_0m^{-r}$.
Then there exist natural numbers $\alpha_1, \alpha_2$ (independent of $r$ and $l$)  such that the number of distinct pieces $P_l$ which are \underline{commensurable with} $\textsf{C}$ and \underline{intersect} $\textsf{C}$ are bounded above by the integer $U_r = \alpha_1 l^{\alpha_2}$.

In particular, we can choose $\alpha_2=0$ (so that $U_r$ is independent of  $r$ and $l$) when all the eigenvalues of $J$ have algebraic multiplicity $1$.
\end{lemma}

\begin{pf} For brevity, we sometimes use the sequence notation $(a_{\mathbf{j}_1},...,a_{\mathbf{j}_{l}})$ for $\sum_{i=1}^{l} T^{-ki} a_{\mathbf{j}_i}\in F=F(T^k,A_k)$ or
$(\tilde{a}_{\mathbf{j}_1},...,\tilde{a}_{\mathbf{j}_{l}})$ for 
$\sum_{i=1}^{l} J_k^{-i} \tilde{a}_{\mathbf{j}_i}\in \tilde{F}=F(J_k,\tilde{A}_k)$, and $(d_{\mathbf{j}_1},...,d_{\mathbf{j}_{l}})$ for $\sum_{i=1}^{l} T_{k}^{-i} d_{{\mathbf{j}_i}}\in F_k=F(T_k,D_k)$ ($1\leq l\leq \infty$). $l= \infty$ is a formal expression corresponding to infinite sequences. 
Since $F_k \rightarrow \tilde{F}$ by Proposition \ref{deflection}, there is an integer $\textsf{K}$ such that $diam (F), diam (\tilde{F}),diam (F_k) \leq \textsf{K}.$

The idea of the proof is to bound the norms of certain integer vectors corresponding to pieces $P_l$ commensurable with $\textsf{C}$ and intersect $\textsf{C}$.
Since the proofs for $F, F_k$ are  identical, we only give the proof for $F_k=F(T_k,D_k)$.  Let  $$T_k=diag(T_{k,1},T_{k,2},... ,T_{k,s}), \ \ \ \ k=1,2,3,....$$
 Let  the moduli of the eigenvalues satisfy $$1<m_1\leq m_2\leq \cdots \leq m_s,$$ and $l,l_p$ be as in Definition \ref{piece}.

We study  pieces $P_l(F_k)$ of level $l$.   We use the notation from the previous section. Namely,
denote by $v_p$ the projection of a vector $v\in \mathbb{R}^n$ to the coordinates  $1+\sum_{t=1}^{p}n_{t-1},...,\sum_{t=1}^{{p}}n_{t}$ ($n_{0}=0$).
Thus $v_p$ is the $p$-th block of $v$  corresponding to the $p$-th diagonal block $T_{k,p}$ of $T_k$.
According to  Definition \ref{piece}-(ii),
 every piece may be expressed through arbitrarily fixed $l$ digits $v_1,v_2,...,v_l\in D_k$, which may not be unique.
 More explicitly, a piece $P_l(F_k)$ will be of the form
 $$P_l(F_k)=\{ (d_{{\mathbf{j}_i}})\in F_k \ : \ (d_{{\mathbf{j}_1}})_p=(v_1)_p, \ (d_{{\mathbf{j}_2}})_p=(v_2)_p, ..., \ (d_{{\mathbf{j}_{l_p}}})_p=(v_{l_p})_p \ \textrm{for} \ p=1,2,...,s \}.$$
 Now choose the digits $v_1,v_2,...,v_l\in D_k$, \ $v'_1,v'_2,...,v'_l\in D_k$
and consider the corresponding pieces $P^1_l(F_k):=P_l(F_k)$ and
$$P^2_l(F_k)=\{ (d'_{{\mathbf{j}_i}})\in F_k \ : \ (d'_{{\mathbf{j}_1}})_p=(v'_1)_p, \ (d'_{{\mathbf{j}_2}})_p=(v'_2)_p, ..., \ (d'_{{\mathbf{j}_{l_p}}})_p=(v'_{l_p})_p \ \textrm{for} \ p=1,2,...,s \}.$$
If we think of $P_l(F_k)$ as a set in $\mathbb{R}^{n} $ rather than the set of sequences and use the sloppy notation of (\ref{sloppy_notation}),  it will be in the form of 
 $(P_l(F_k))_{p}=\sum_{i=1}^{l_{p}} T_{k,p}^{-i} (d_{{\mathbf{j}_i}})_{p}  + T_{k,p}^{-l_{p}} (F_k)_{p}$ 
the sum vectors 
$$((d_{{\mathbf{j}_1}})_{p},(d_{{\mathbf{j}_2}})_{p},...,(d_{{\mathbf{j}_{l_{p}}}})_{p})=\sum_{i=1}^{l_{p}} T_{k,p}^{-i} (d_{{\mathbf{j}_i}})_{p} $$  uniquely determine a piece $P_l(F_k)$. The individual digits $d_{{\mathbf{j}_{i}}}$ in such a sum may not  identify the pieces. Because different sets of $l_{p}$ digits may give the same sum. Therefore, to find the required upper bound stated in the lemma we will work with such sums.

Consider a block $T_{k,p}$ with $p=1,2,...,s$.
Recall that $T_{k,p}$  has size $n_p\times n_p$ for  real eigenvalues and $2n_p\times 2n_p$ for non-real eigenvalues. Next consider the points of $P^1_l(F_k), \ P^2_l(F_k)$ restricted to the corresponding block  $T_{k,p}$: Choose an arbitrary  $p^*\in\{1,2,...,s\}$. Let $(F_k)_{p^*}\subset \mathbb{R}^{n_p} $ or $(F_k)_{p^*}\subset \mathbb{R}^{2n_p} $
denote the set of points of $F_k$ restricted to the corresponding block  $T_{k,p^*}.$ Then
\begin{equation}\label{piece_form}
 (P^1_l(F_k))_{p^*}=\sum_{i=1}^{l_{p^*}} T_{k,p^*}^{-i} (d_{{\mathbf{j}_i}})_{p^*}  + T_{k,p^*}^{-l_{p^*}} (F_k)_{p^*}, \quad  (P^2_l(F_k))_{p^*}=\sum_{i=1}^{l_{p^*}} T_{k,p^*}^{-i} ( d'_{{\mathbf{j}_i}})_{p^*} + T_{k,p^*}^{-l_{p^*}} (F_k)_{p^*}.
\end{equation}
 Assume that $P^1_l(F_k), \ P^2_l(F_k)$ are commensurable with an n-cube $\textsf{C}$ of side length $c_0m^{-r}$ and they intersect $\textsf{C}$.  Fix $P^1_l(F_k)$, but let $P^2_l(F_k)$ vary in $\textsf{C}$.  For such pieces, we would like to show that  the number  of distinct sum vectors $((d'_{{\mathbf{j}_1}})_{p},(d'_{{\mathbf{j}_2}})_{p},...,(d'_{{\mathbf{j}_{l_{p}}}})_{p})$ is bounded by a number $U_r=\alpha_1 l^{\alpha_2}$.
Because $P^1_l(F_k), \ P^2_l(F_k)$ intersect $\textsf{C}$,
there exist points $$p_1=((d_{{\mathbf{j}_1}})_{p^*},(d_{{\mathbf{j}_2}})_{p^*},...,(d_{{\mathbf{j}_{l_{p^*}}}})_{p^*})+q_1\in (P^1_l(F_k))_{p^*},
\quad \quad  q_1\in T_{k,p^*}^{-l_{p^*}} (F_k)_{p^*},$$
$$p_2=((d'_{{\mathbf{j}_1}})_{p^*},(d'_{{\mathbf{j}_2}})_{p^*},...,(d'_{{\mathbf{j}_{l_{p^*}}}})_{p^*})+q_2\in (P^2_l(F_k))_{p^*},
\quad \quad  q_2 \in T_{k,p^*}^{-l_{p^*}} (F_k)_{p^*}$$
such that (by reverse triangle inequality)
$$ \left| \left((d_{{\mathbf{j}_1}}- d'_{{\mathbf{j}_1}})_{p^*},(d_{{\mathbf{j}_2}}- d'_{{\mathbf{j}_2}})_{p^*},...,(d_{{\mathbf{j}_{l_{p^*}}}}- d'_{{\mathbf{j}_{l_{p^*}}}})_{p^*} \right) \right|-|q_1-q_2 | \leq | p_1-p_2|\leq cm^{-r}=diam(\textsf{C}),$$ where $c=c_0\sqrt{n}$. Since $P^1_l(F_k), P^2_l(F_k)$ are commensurable with $\textsf{C}$, we have
\begin{eqnarray}\label{bound1} \left| \left((d_{{\mathbf{j}_1}}- d'_{{\mathbf{j}_1}})_{p^*},(d_{{\mathbf{j}_2}}- d'_{{\mathbf{j}_2}})_{p^*},...,(d_{{\mathbf{j}_{l_{p^*}}}}- d'_{{\mathbf{j}_{l_{p^*}}}})_{p^*} \right) \right| & \leq & |q_1-q_2 |+cm^{-r} \\ \nonumber
& \leq & diam (T_{k,p^*}^{-l_{p^*}} (F_k)_{p^*})+  cm^{-r} \\      \nonumber
& \leq & diam (P^1_l(F_k)) + cm^{-r}  \\      \nonumber
& \leq & c_2m^{-r}+cm^{-r}.
\end{eqnarray}
Absorbing $c_2$ into $c$, we write only one constant $c$ in this inequality:

 $$\left| \sum_{i=1}^{l_{p^*}} T_{k,p^*}^{-i} (d_{{\mathbf{j}_i}}- d'_{{\mathbf{j}_i}})_{p^*}\right| =\left| \left((d_{{\mathbf{j}_1}}- d'_{{\mathbf{j}_1}})_{p^*},(d_{{\mathbf{j}_2}}- d'_{{\mathbf{j}_2}})_{p^*},...,(d_{{\mathbf{j}_{l_{p^*}}}}- d'_{{\mathbf{j}_{l_{p^*}}}})_{p^*} \right) \right| \leq cm^{-r}
 $$  for $p^*\in\{1,2,...,s\}$.
 By using  \cite[Theorem 2, p. 370]{L} on norm inequalities again,
 that gives
 \begin{eqnarray*} \frac{\left| \sum_{i=1}^{l_{p^*}} T_{k,p^*}^{l_{p^*}-i} (d_{{\mathbf{j}_i}}- d'_{{\mathbf{j}_i}})_{p^*} \right|}{\left| \left|   T_{k,p^*}^{l_{p^*}} \right| \right|} \leq \left| T_{k,p^*}^{-l_{p^*}}\sum_{i=1}^{l_{p^*}} T_{k,p^*}^{l_{p^*}-i} (d_{{\mathbf{j}_i}}- d'_{{\mathbf{j}_i}})_{p^*} \right| & \leq & cm^{-r}.
\end{eqnarray*}
From the first and the last expressions, we get
\begin{eqnarray}\label{bound33} \nonumber \left| \sum_{i=1}^{l_{p^*}} T_{k,p^*}^{l_{p^*}-i} (d_{{\mathbf{j}_i}}- d'_{{\mathbf{j}_i}})_{p^*} \right|  \leq   \left| \left|   T_{k,p^*}^{l_{p^*}} \right| \right| 
 cm^{-r}.
\end{eqnarray}
We estimate  these norms. 
Using Lemma \ref{lemma_norm_estimate1}, Remark \ref{constant_remark}, 
we see that  there exist constants $c',c'', \alpha_0    
\in \mathbb{N}
$ such that
\begin{eqnarray}\label{bound4}
\left| \sum_{i=1}^{l_{p^*}} T_{k,p^*}^{l_{p^*}-i} (d_{{\mathbf{j}_i}}- d'_{{\mathbf{j}_i}})_{p^*} \right| & \leq &  \left| \left|   T_{k,p^*}^{l_{p^*}} \right| \right|
 cm^{-r} \leq
c c' l_{p^*}^{\alpha_0}
m_{p^*}^{l_{p^*}} m^{-r}
\end{eqnarray}
and   similarly, 
 \begin{equation}\label{diam_piece}
  diam (P^1_{l}(F_k))\leq \overset{s}{\underset{p=1}{\Sigma}} diam \left(T_{k,p}^{-l_p}F_k \right) \leq  \overset{s}{\underset{p=1}{\Sigma}} \left| \left| T_{k,p}^{-l_p}  \right| \right|  diam (F_k)\leq  c'' l_{p^*}^{\alpha_0} \overset{s}{\underset{p=1}{\Sigma}} m_p^{-l_p} \textsf{K}.
 \end{equation}
 
Recall that $l_{p+1}=\lfloor l_p\log_{m_{p+1}} m_p \rfloor$. This gives $$m_{p+1}^{l_{p+1}} \leq  m_p^{l_p}  < m_{p+1}^{l_{p+1}+1}  \Rightarrow  m_p^{-l_p} \leq m_{p+1}^{-l_{p+1}}, \ \
  m_{p+1}^{-l_{p+1}}< m_{p+1} m_p^{-l_p}
.$$ Then
 \begin{eqnarray*}
\Sigma_{p=1}^{s}m_p^{-l_p} \textsf{K} & = & \Sigma_{p\leq p^*} m_p^{-l_p} \textsf{K} + \Sigma_{p>p^*} m_p^{-l_p} \textsf{K} \\
& \leq & \Sigma_{p\leq p^*}m_{p^*}^{-l_{p^*}}\textsf{K} +\Sigma_{p>p^*}m_{1}\cdots m_{s}m_{p^*}^{-l_{p^*}}\textsf{K} \\
& \leq & s m_{p^*}^{-l_{p^*}}\textsf{K}+sm_{1}m_{2}\cdots m_{s}m_{p^*}^{-l_{p^*}}\textsf{K} \\
& \leq & 2sm_{1}m_{2}\cdots m_{s}m_{p^*}^{-l_{p^*}}\textsf{K}
\end{eqnarray*}
for each $p^*\in\{1,2,...,s\}$. 
By virtue  of the commensurability of $P^1_{l}(F_k)$ in Definition \ref{piece}-(iii), \ inequality (\ref{diam_piece}) then 
yields
$$c_1m^{-r} \leq diam (P^1_{l}(F_k))
\leq c'' l_{p^*}^{\alpha_0} \ m_{p^*}^{-l_{p^*}}2sm_{1}m_{2}\cdots m_{s}\textsf{K}.$$ 
Therefore, we obtain $  c_1m^{-r}m_{p^*}^{l_{p^*}} \leq c'' l_{p^*}^{\alpha_0} 2sm_{1}m_{2}\cdots m_{s}\textsf{K}$. 
If we take into account of multiplicities of eigenvalues, we have $m_{1}m_{2}\cdots m_{s}\leq |\det(T_k)|.$
We then use inequality  (\ref{bound4}) to reach 
\footnotesize \begin{equation}\label{radius}
 \left| \sum_{i=1}^{l_{p^*}} T_{k,p^*}^{l_{p^*}-i} (d_{{\mathbf{j}_i}}- d'_{{\mathbf{j}_i}})_{p^*} \right|  \leq c c' l_{p^*}^{\alpha_0} m_{p^*}^{l_{p^*}}m^{-r} \leq c  c' l_{p^*}^{\alpha_0}  \frac{c'' l_{p^*}^{\alpha_0} 2sm_{1}m_{2}\cdots m_{s}\textsf{K}}{c_1}\leq c  c' l_{p^*}^{\alpha_0}  \frac{c'' l_{p^*}^{\alpha_0} 2s|\det(T_k)|\textsf{K}}{c_1}.
\end{equation} \normalsize 
Recall that $l_{p^*}\leq l=l_1$ or $\bar{l}_{p^*}\leq l=\tilde{l}_1$. Then we may take $U_r=\left(2\overset{s}{\underset{p^*=1}{\sum}}\left[1+\lceil \frac{c  c'c''2s|\det(T_k)|\textsf{K}}{c_1} \rceil \cdot  l^{2\alpha_0} \right]\right)^n$ for suitable constants.

\medskip

In summary, our starting point was to consider the level-$l$ pieces, identified by the vectors $\sum_{i=1}^{l_{p}} T_{k,p}^{-i} (d_{{\mathbf{j}_i}})_{p}$,  
 $\sum_{i=1}^{l_{p}} T_{k,p}^{-i} (d'_{{\mathbf{j}_i}})_{p} $, ($p=1,2...,s$), which intersect an n-cube $\textsf{C}$ of side length $c_0m^{-r}$ 
and commensurable with $\textsf{C}$.  We have fixed $\sum_{i=1}^{l_{p}} T_{k,p}^{-i} (d_{{\mathbf{j}_i}})_{p}$ 
 at the beginning of the proof, and   then obtained  upper bounds 
for the norms of  the vectors
$\sum_{i=1}^{l_p} (T_{k,p})^{l_p-i} (d_{{\mathbf{j}_i}}- d'_{{\mathbf{j}_i}})_p$.
 This means that the vectors $\sum_{i=1}^{l_p} (T_{k,p})^{l_p-i} (d'_{{\mathbf{j}_i}})_p$ 
fall into a sphere with center $\sum_{i=1}^{l_p} (T_{k,p})^{l_p-i} (d_{{\mathbf{j}_i}})_p$ 
and fixed radius given by (\ref{radius}), or into an $n$-cube with side length the diameter of the sphere. But these vectors are integer vectors. Therefore,
the number of the integer vectors $\sum_{i=1}^{l_p} (T_{k,p})^{l_p-i} (d'_{{\mathbf{j}_i}})_p=T_{k,p}^{l_p}\sum_{i=1}^{l_p} (T_{k,p})^{-i} (d'_{{\mathbf{j}_i}})_p$ 
corresponding to $\textsf{C}$ must be bounded.  It follows that the number of the  vectors $\sum_{i=1}^{l_p} (T_{k,p})^{-i} (d'_{{\mathbf{j}_i}})_p$
intersecting $\textsf{C}$ is also bounded
by a natural number $U_r$ as stated in the lemma.

As for $\tilde{F}=F(J_k,\tilde{A}_k)$, consider the level-$l$ pieces, identified by the vectors $\sum_{i=1}^{\tilde{l}_{p^*}}J_{k,{p^*}}^{-i} (\tilde{a}_{\mathbf{j}_i})_{p^*}$, $\sum_{i=1}^{\tilde{l}_{p^*}}J_{k,{p^*}}^{-i} (\tilde{a}'_{\mathbf{j}_i})_{p^*}$.  Assume that they intersect an n-cube $\textsf{C}$ of side length $c_0m^{-r}$ 
and commensurable with $\textsf{C}$. 
 Similar to (\ref{radius}), we get
\begin{eqnarray}\label{bound5}
\left|\sum_{i=1}^{\tilde{l}_{p^*}} ((J_k)^{\tilde{l}_{p^*}-i} (\tilde{a}_{\mathbf{j}_i}-\tilde{a}'_{\mathbf{j}_i}))_{p^*} \right|  & = &\left| \sum_{i=1}^{\tilde{l}_{p^*}} 
J_{k,{p^*}}^{\tilde{l}_{p^*}-i} (\tilde{a}_{\mathbf{j}_i}-\tilde{a}'_{\mathbf{j}_i})_{p^*} \right| \nonumber \\
& \leq & c  c''' \tilde{l}_{p^*}^{\alpha_0}  \frac{c'''' \tilde{l}_{p^*}^{\alpha_0} 2sr_{1}r_{2}\cdots r_{s}\textsf{K}}{c_1} \nonumber \\
& \leq &c  c''' \tilde{l}_{p^*}^{\alpha_0}  \frac{c'''' \tilde{l}_{p^*}^{\alpha_0} 2s|\det(J_k)|\textsf{K}}{c_1},
\end{eqnarray}
where
$\tilde{a}_{\mathbf{j}_i}, \tilde{a}'_{\mathbf{j}_i}\in \tilde{A}_k$  and  $c''',c'''', \alpha_0 \in \mathbb{N}$   
are constants. However, this time,  $\sum_{i=1}^{\tilde{l}_{p^*}}J_{k,{p^*}}^{\tilde{l}_{p^*}-i} (\tilde{a}_{\mathbf{j}_i}-\tilde{a}'_{\mathbf{j}_i})_{p^*}$ may not be integer vectors.  To overcome this difficulty, we pass to integer vector. For that, we use the integral self-affine set $F=F(T^k,A_k)$ and the identity
$P^{-1}T^k=J_kP^{-1}$ so that $P^{-1}F=\tilde{F}=F(J_k,\tilde{A}_k)$.
This yields
 \begin{eqnarray*}
 \frac{1}{||P||} \left|\sum_{i=1}^{\tilde{l}_{p^*}} ((T^k)^{{l_{p^*}}-i} (a_{\mathbf{j}_i}-a'_{\mathbf{j}_i}))_{p^*} \right|
 & \leq & \left| \sum_{i=1}^{\tilde{l}_{p^*}} (P^{-1}(T^k)^{{\tilde{l}_{p^*}}-i} (a_{\mathbf{j}_i}-a'_{\mathbf{j}_i}))_{p^*} \right| \\
 & = & \left|\sum_{i=1}^{\tilde{l}_{p^*}} ((J_k)^{\tilde{l}_{p^*}-i} (\tilde{a}_{\mathbf{j}_i}-\tilde{a}'_{\mathbf{j}_i}))_{p^*} \right| \\
 & \leq & c  c''' \tilde{l}_{p^*}^{\alpha_0}  \frac{c'''' \tilde{l}_{p^*}^{\alpha_0} 2s|\det(J_k)|\textsf{K}}{c_1}
\end{eqnarray*}
 for each $p^*\in\{1,2,...,s\}$ by (\ref{bound5}).
Thus $\sum_{i=1}^{\tilde{l}_{p^*}} ((T^k)^{\tilde{l}_{p^*}-i} (a_{\mathbf{j}_i}-a'_{\mathbf{j}_i}))_{p^*}$ are integer vectors. Then we fix $\sum_{i=1}^{\tilde{l}_{p^*}} ((T^k)^{\tilde{l}_{p^*}-i} a_{\mathbf{j}_i})_{p^*}$, let $\sum_{i=1}^{\tilde{l}_{p^*}} ((T^k)^{\tilde{l}_{p^*}-i} a'_{\mathbf{j}_i}))_{p^*}$ vary and reach the same conclusion as above.

The last claim of the lemma follows from the last assertion of Lemma \ref{lemma_norm_estimate1}.
\end{pf}
\medskip

\begin{remark}\label{Actual_Case1} \rm{ 

(i) We will see that the dependence on $l$ in the above lemma has no affect to our proofs so that one may also assume that $U_r$ is a constant without loss of generality. 

(ii)  In the actual case, i.e., if we don't assume the sloppy notation for pieces, and if we consider all pieces $P_l(F_k)$ which are commensurable with a single cube and intersect that cube, then the upper bound in the lemma takes the form $\alpha_32^{\alpha_1 l^{\alpha_2}}$ ($\alpha_3$ is a positive constant). Because in Definition \ref{piece}-(ii), 
we have observed that every piece is between two (cylinder) sets in $F_k$. That is, if we set $\mathbf{I}_1=\mathbf{j}_1...\mathbf{j}_l$ (with $l=l_1$) and $\mathbf{I}_s=\mathbf{j}_1...\mathbf{j}_{l_s}$, then
 $$(F_k)_{\mathbf{I}_1}:=\sum_{i=1}^{l_{1}} T_{k}^{-i} d_{{\mathbf{j}_i}}  + T_{k}^{-l_{1}} F_k \subseteq P_l(F_k) \subseteq \sum_{i=1}^{l_{s}} T_{k}^{-i} d_{{\mathbf{j}_i}}  + T_{k}^{-l_{s}} F_k=(F_k)_{\mathbf{I}_s},$$
and $P_l(F_k)$  is a finite union of the cylinder sets $(F_k)_{\mathbf{I}_1}$
for certain $d_{{\mathbf{j}_i}}\in D_k$ ($1\leq i \leq l_1=l$), but $d_{{\mathbf{j}_1}},...,d_{{\mathbf{j}_{l_{s}}}}$ are fixed. Let $m_1=\lceil |\lambda_1|^{k}\rceil$ as in Definition \ref{piece}. Note that for the cylinder sets in the above inclusions, we also have 
\begin{equation}\label{piece_cylinder}
 c' l^{-\alpha_0} (m_1)^{-l}\leq diam((F_k)_{\mathbf{I}_1}) \leq \ diam(P_l(F_k))
   \leq c'' l^{\alpha_0} (m_1)^{-l}
\end{equation} 
for suitable 
 constants $c'(k), \ c''(k), \ \alpha_0>0$ by the proof of Lemma \ref{inequality_pieces} (with $J_k$ replaced by $T_k$ and $|\lambda_1|^{k}$ replaced by $m_1$). That yields 
 $$
diam((F_k)_{\mathbf{I}_1}) \leq diam(P_l(F_k)) \leq c'' l^{\alpha_0} (m_1)^{-l}\leq  \frac{c''}{c'} l^{2\alpha_0} diam((F_k)_{\mathbf{I}_1}).$$ Also for  cylinder sets,  the above proof certainly holds.
If a cylinder set $(F_k)_{\mathbf{I}_1}$ is commensurable with a cube $\textsf{C}$ as in Lemma \ref{bound}, then the number of all possible  unions of such cylinder sets is bounded by a number in the form of  $\alpha_3 2^{((\frac{c''}{c'})^n\alpha_1 l^{\alpha_2+2n\alpha_0})}$ by the same lemma. Absorbing $2n\alpha_0$ into $\alpha_2$ and 
$(\frac{c''}{c'})^n$ into $\alpha_1$, we get $\alpha_32^{\alpha_1 l^{\alpha_2}}$. With this more general upper bound, our proofs will hold too. \hfill $\Box$
}

\end{remark}

\section{Neighbor Structure of Dynamical Perturbations }\label{Neighbor Structure}

In this section, we will see that overlap structures of  an integral self-affine set $F=F(T,A)$ and its perturbations $F_k$ are closely related. To compare the overlap structures $F$ and $F_k$, the concept of neighboring pieces is needed (see \cite{BM}). 
As in the introduction section, let  $f_{\mathbf{i}},f_{\mathbf{j}}$ denote the same-length compositions of the IFS maps in (\ref{IFS_maps}) defining $F$. Let $A_k$ be as in (\ref{IFS_digits}).
For the multi-indices $\mathbf{i}$, $\mathbf{j}$ of the same length $|\mathbf{i}|=|\mathbf{j}|=\textsf{r}$, let
$$F_{\mathbf{i}}:=f_{\mathbf{i}}(F)=T^{-k}(F+a_{\mathbf{i}}), \ \ \quad  F_{\mathbf{j}}:=f_{\mathbf{j}}(F)=T^{-k}(F+a_{\mathbf{j}}), \quad a_{\mathbf{i}}, \ a_{\mathbf{j}}\in A_\textsf{r}$$
be the corresponding subsets of $F$, sometimes called the cylinder sets. In the definition of $F_{\mathbf{i}}$, \ $F$ may be replaced by other sets.

\begin{defn} \rm{   
 $F_{\mathbf{i}}$ and $F_{\mathbf{j}}$ are  called ``\textit{neighboring pieces of $F$}" or \textit{neighbors} if $F_{\mathbf{i}}\cap F_{\mathbf{j}}\neq \emptyset$ and $\mathbf{i}\neq \mathbf{j}$. $\hfill \Box$}
\end{defn}

When $\textsf{r}=l$, the sets $F_{\mathbf{i}}$ are not necessarily level-$l$ pieces $P_l(F)$ in Definition \ref{piece}. However, a level-$l$ piece is a finite union of the sets $F_{\mathbf{i}}$ with $|\mathbf{i}|=l$.  Further, every set $F_{\mathbf{i}}$ coincides with some piece $P_l(F)$  when the expansive matrix $T$ is a similarity (or all of its  eigenvalues have the same modulus) and $\textsf{r}=l$.
All neighboring pieces can be determined by a graph called the ``\textit{neighbor graph}" of $F$, where the vertices are certain translation functions denoted by $$h(x)=f_{\mathbf{i}}^{-1}f_{\mathbf{j}}(x)=x+a_{\mathbf{j}}-a_{\mathbf{i}}, \ \ \ \ a_{\mathbf{i}}, a_{\mathbf{j}}\in A_\textsf{r}, \quad \quad F_{\mathbf{i}}\cap F_{\mathbf{j}}\neq \emptyset $$ (with the root vertex is the identity function) and the multi-indices $\mathbf{i},\mathbf{j}$ are edge labels of finite paths in \cite{BM}.

 First, we extract the following lemma from \cite[p.134]{BM}.
Let $\mathbf{j}\mathbf{j}_1$ denote a multi-index obtained by concatenating the multi-indices
$\mathbf{j}$ and  $\mathbf{j}_1$ (not necessarily of the same length, and we allow $\mathbf{j}$ or $\mathbf{j}_1$ to have length $0$, i.e. $\mathbf{j}\mathbf{j}_1=\mathbf{j}_1$ or $\mathbf{j}\mathbf{j}_1=\mathbf{j}$).

\begin{lemma}\label{vertex_condition}
Let $F$ be an integral self-affine set.  Then
\begin{itemize}
 \item[(a)] If $F_{\mathbf{i}}\cap F_{\mathbf{j}}\neq \emptyset$, then $|h(0)=a_{\mathbf{j}}-a_{\mathbf{i}}|\leq diam(F).$
 \item[(b)] If $F_{\mathbf{i}}\cap F_{\mathbf{j}}= \emptyset$, then $\underset{l\rightarrow\infty}{\lim} |a_{\mathbf{j}\mathbf{j}_1}-a_{\mathbf{i}\mathbf{i}_1}|\rightarrow \infty$ \rm{ ($|\mathbf{j}\mathbf{j}_1|=|\mathbf{i}\mathbf{i}_1|=l$)}.
\end{itemize}
\end{lemma}
 
Here we would like to present some of the relationships between the  neighbor graphs of $\tilde{F}$ and its lower perturbation $F_k=F(T_k,D_k)$. 
Let $G_{\tilde{F}}=(V_{\tilde{F}},E_{\tilde{F}})$ and $G_k=( V_{F_k},E_k)$ denote the neighbor graphs of $\tilde{F}=F(J,P^{-1}A)$ and $F_k$ respectively.
 $V_{\tilde{F}}$, $E_{\tilde{F}}$ refer to vertex set and edge (label) set respectively. We will exclude the trivial case $E_{\tilde{F}}=\emptyset$ from our consideration. 
We do not follow the notation of \cite{BM}, where  an edge label is denoted by  $(i,j)$ or 
 $i,j$ and a path label by $\mathbf{i},\mathbf{j}$. Instead, we use square brackets. 
An edge label is in the form of 
 $e=[i,j]$ and a path label is in the form of $[\mathbf{i},\mathbf{j}]$ (this notation shouldn't be confused with concatenation $\mathbf{i}\mathbf{j}$ !). 
If  $\mathbf{i}=i_1...i_{\textsf{r}}, \  \mathbf{j}=j_1...j_{\textsf{r}}$, then our ${\textsf{r}}$-path is $[\mathbf{i},\mathbf{j}]=[i_1...i_{\textsf{r}},j_1...j_{\textsf{r}}]$ so that $[i_1,j_1]$ is the first edge and 
$[i_{\textsf{r}},j_{\textsf{r}}]$ is the last edge in this path.
Let $E^{\textsf{r}}_{\tilde{F}}$ denote the set of sequences of $\textsf{r}$  edges, 
 or shortly the set of $\textsf{r}$-paths in $G_{\tilde{F}}$. Here we should be careful. $\tilde{F}=F(J,P^{-1}A)=F(J,\tilde{A})=F(J_k,\tilde{A}_k)$ for each $k\in \mathbb{N}$, but the edges of $G_{F(J_k,\tilde{A}_k)}$ are $k$-paths in 
$G_{F(J,\tilde{A})}$.
 Similarly, we use the notation $E_k^{\textsf{r}}$ for the set of  $\textsf{r}$-paths in  $G_k$. 
From  (\ref{special_correspondence}),
we remember the special correspondence 
\begin{equation}\label{special_correspondence2}
x=\sum_{i=1}^\infty J_{k}^{-i} \tilde{a}_{\mathbf{j}_i}=(\tilde{a}_{\mathbf{j}_i}), \ \ \ \ \ \ \ \ \ \ \ \ x_k=\sum_{i=1}^\infty T_{k}^{-i} d_{\mathbf{j}_i}=(d_{\mathbf{j}_i}),
\end{equation}
where $\tilde{a}_{\mathbf{j}_i}\in \tilde{A}_k$ and $d_{\mathbf{j}_i}\in D_{k}$.

There may be various ways to prove some relations between $G_{\tilde{F}}$ and $G_k$. 
The following shows that there is close relationship between the edges and paths of $G_k$ and those of $G_{\tilde{F}}.$ 
First, we introduce some notation.   For  a given  $\textsf{r}$, let $E_k(\textsf{r})  $ denote the set of restrictions of the edge labels of $E_k$ to the first $\textsf{r}$ indices  . Similarly, $\overline{E}_k(\textsf{r}) $ denote the set of restrictions of the edge labels of $E_k$ to the last $\textsf{r}$ indices.
 For example, let $\mathbf{i}_1=\mathbf{i}_0\mathbf{i}'_0, \quad  \mathbf{j}_1=\mathbf{j}_0\mathbf{j}'_0$ and $[\mathbf{i}_1,\mathbf{j}_1]\in E_k$. If $|\mathbf{i}_0|=|\mathbf{j}_0|=\textsf{r}$, then $[\mathbf{i}_0,\mathbf{j}_0]\in E_k(\textsf{r})  $; if $|\mathbf{i}'_0|=|\mathbf{j}'_0|=\textsf{r}$, then $[\mathbf{i}'_0,\mathbf{j}'_0]\in \overline{E}_k(\textsf{r}) .$
 Lastly, let  $\mathbf{i}_2=\mathbf{i}\mathbf{i'},  \mathbf{j}_2=\mathbf{j}\mathbf{j'}$ with $[\mathbf{i}_2,\mathbf{j}_2]\in E_k$, and let $E_k^2(\textsf{r})$ denote the set of restrictions of $2$-paths $[\mathbf{i}_1\mathbf{i}_2,\mathbf{j}_1\mathbf{j}_2]\in E_k^2$  to
$[\mathbf{i}'_0\mathbf{i}, \mathbf{j}'_0\mathbf{j}]$ when $|\mathbf{i}'_0|=|\mathbf{j}'_0|=|\mathbf{i}|=|\mathbf{j}|=\textsf{r}.$ Thus $E_k^2(\textsf{r})$ consists of concatenations of certain elements
of $\overline{E}_k(\textsf{r}) $ and  $E_k(\textsf{r})  $ respectively.

We also denote the set of multi-indices of length $\textsf{r}$ by $W_{\textsf{r}} =\{\mathbf{i} : |\mathbf{i}|=\textsf{r}\}$ and the set of infinite sequences of indices by  $W_{\infty}=\{\mathbf{i}=i_1i_2...i_l ... \ : \ i_l=1,2,..,q\}.$
Let $W=\bigcup_{\textsf{r}=1}^{\infty}W_\textsf{r}$ be the set of all  multi-indices. Then define a  metric $d$ on $W\cup W_{\infty}$ by $d(\mathbf{i},\mathbf{j})=0 \Leftrightarrow  \mathbf{i}=\mathbf{j}$, \ $d(\mathbf{i},\mathbf{j})=d(\mathbf{j},\mathbf{i})$ and
$$d(\mathbf{i},\mathbf{j})= \left\{
          \begin{array}{ll}
          0 & \hbox{\  if $ \mathbf{i}=\mathbf{j}$;} \\
          1/2^{\textsf{r}+1} & \hbox{\  else if $\mathbf{i}\in W_{\textsf{r}} , \ \mathbf{j}\in W_\textsf{s}$ ($\textsf{r}<\textsf{s}\leq \infty$) \ and \  $i_1=j_1, \ i_2=j_2,...,i_{\textsf{r}}=j_\textsf{r}$;} \\
          1/2^{l} & \footnotesize{ \hbox{\  otherwise $\mathbf{i}\in W_{\textsf{r}} , \ \mathbf{j}\in W_\textsf{s}$ ($\textsf{r}\leq \textsf{s}\leq \infty$) and $l\ (\leq \textsf{r})$ is the minimum
          index
          with
          $i_l\neq j_l.$} }
          \end{array}
        \right.$$

Consider $E^{\textsf{r}}_{\tilde{F}}$, $E_k$, $\overline{E}_k(\textsf{r}) $ and  $E_k(\textsf{r})  $  as defined above. Then we have the following results.

\begin{lemma}\label{neighboring_pieces1} Let $\tilde{F}=F(J,\tilde{A})$  and $F_k$ be its perturbations. 
For any given positive integer $\textsf{r}$, there exists an integer $k_{\textsf{r}}\geq \textsf{r}$ such that  
 $E_k(\textsf{r})  =E^{\textsf{r}}_{\tilde{F}}$ for $k\geq k_{\textsf{r}}$.
Consequently, in the Hausdorff metric induced by $d$, $E_k(\textsf{r})  $ converges to the infinite paths of $G_{\tilde{F}}$ as $\textsf{r}$  goes to infinity 
(hence $k\geq \textsf{r}\Rightarrow k\rightarrow \infty$ too).
\end{lemma}

\begin{pf} We will use Lemma \ref{vertex_condition} to prove the properties of the edges of $F_k$ stated in the proposition. Similar to (\ref{IFS_digits}), we introduce  new notation  for multi-index concatenation: 
$$ \mathbf{J}=\mathbf{j}_1\cdots \mathbf{j}_{\textsf{s}}, \quad \quad \quad \quad \quad \quad   \mathbf{I}=\mathbf{i}_1\cdots \mathbf{i}_\textsf{s},$$
\begin{equation}\label{new_notation}
  d_{\mathbf{J}}:=\Sigma_{i=1}^{\textsf{s}}T_{k}^{(\textsf{s}-i)}d_{\mathbf{j}_i} , \quad \quad  \quad \quad 
d_{\mathbf{I}}:=\Sigma_{i=1}^{\textsf{s}}T_{k}^{(\textsf{s}-i)}d_{\mathbf{i}_i}  \quad \quad    d_{\mathbf{i}_i}, d_{\mathbf{j}_i}\in D_k.
\end{equation}
Thus $\mathbf{J}$ and  $\mathbf{I}$  are  obtained by concatenation of $l$ blocks (multi-indices),  each of which has  length $k$.

It is well known that the neighbor graph of an integral self-affine set is finite, i.e. it is finite-type \cite{BM}.
Since $F$ and ${F_k}$ are integral self affine sets, the graphs $G_{\tilde{F}}$ and
$G_k$ are finite. By Proposition  \ref{deflection}, we have $F_k \longrightarrow \tilde{F}$. Then the cardinalities of the vertex sets $V_{\tilde{F}}, V_k$ are bounded above by a positive integer $\textsf{r}_0$ independent of $k$. To include all cyclic paths, we may choose $\textsf{r}_0$ to be the minimum of such bounds plus $1$. Thus all edges in $E_{\tilde{F}}$ (or all neighbor types in the terminology of $\cite{BM}$) will appear in the paths of  $E_{\tilde{F}}^{\textsf{r}_0}$.
Besides, the periods of all cyclic paths will
be bounded by $\textsf{r}_0.$  This will be used below.

 \textit{  For large $k$, initial parts of edge labels of $G_k$ coincides with some finite paths of $G_{\tilde{F}}$} :

 To show that each $\textsf{r}$-label $[\mathbf{i},\mathbf{j}]\notin E^{\textsf{r}}_{\tilde{F}}$ cannot be the initial part of an edge in $E_k$ for large $k$,
consider $[\mathbf{i},\mathbf{j}]\notin E^{\textsf{r}}_{\tilde{F}}$, where $|\mathbf{i}|=|\mathbf{j}|=r$ and the longer multi-indices $\mathbf{i}\mathbf{i'},\mathbf{j}\mathbf{j'}$ with any length $|\mathbf{i}\mathbf{i'}|=|\mathbf{j}\mathbf{j'}|=k\geq \textsf{r}.$
Since  $\tilde{F}_{\mathbf{i}}\cap \tilde{F}_{\mathbf{j}}= \emptyset$, by Lemma  \ref{vertex_condition}-(b) we have $$|\tilde{a}_{\mathbf{j}\mathbf{j'}}-\tilde{a}_{\mathbf{i}\mathbf{i'}}|>  diam(\tilde{F})+1+n$$ ($n$ is the space dimension) for all sufficiently large
$k\geq \textsf{r}$. 
By the definition in (\ref{floor_digit_lower_perturbation}), we have      $d_{\mathbf{j}\mathbf{j'}}=\lfloor \tilde{a}_{\mathbf{j}\mathbf{j'}} \rfloor,
\ \  d_{\mathbf{i}\mathbf{i'}}=\lfloor \tilde{a}_{\mathbf{i}\mathbf{i'}}  \rfloor$ assuming that $\tilde{a}_{\mathbf{j}\mathbf{j'}}, \tilde{a}_{\mathbf{i}\mathbf{i'}}\geq 0.$  This in turn gives
$$n>|\tilde{a}_{\mathbf{j}\mathbf{j'}}-\tilde{a}_{\mathbf{i}\mathbf{i'}}|-|d_{\mathbf{j}\mathbf{j'}}-d_{\mathbf{i}\mathbf{i'}}| \Longrightarrow
|d_{\mathbf{j}\mathbf{j'}}-d_{\mathbf{i}\mathbf{i'}}|> |\tilde{a}_{\mathbf{j}\mathbf{j'}}-\tilde{a}_{\mathbf{i}\mathbf{i'}}|-n> diam(\tilde{F})+1$$ for all large $k$.
Since $F_k \longrightarrow \tilde{F}$,  it follows that $\underset{k\rightarrow\infty}{\lim} diam(F_k) = diam(\tilde{F})$ and hence, there exists a positive integer $k_\textsf{r}\geq \textsf{r}$ such that for $k\geq k_{\textsf{r}}$, we have $  |diam(F_k)-diam(\tilde{F})|<1$. We finally obtain
$$|d_{\mathbf{j}\mathbf{j'}}-d_{\mathbf{i}\mathbf{i'}}|> diam(F_k)$$ for $k\geq k_{\textsf{r}}$. Then, by Lemma \ref{vertex_condition}-(a), $(F_k)_{\mathbf{i}\mathbf{i'}}\cap (F_k)_{\mathbf{j}\mathbf{j'}}= \emptyset$ for $k\geq k_{\textsf{r}}$. That is, $[\mathbf{i},\mathbf{j}]\notin E_k(\textsf{r})  $ for $k\geq k_{\textsf{r}}$.
It follows that   $E_k(\textsf{r})  \subseteq E^{\textsf{r}}_{\tilde{F}}$ for $k\geq k_{\textsf{r}}\geq \textsf{r}.$

Conversely, assume that
$[\mathbf{i},\mathbf{j}]\in E^{\textsf{r}}_{\tilde{F}}.
$
We want to prove the inclusion $E^{\textsf{r}}_{\tilde{F}}\subseteq E_k(\textsf{r})  $ for large  $k$. This part requires more delicate study.
 Since $\tilde{F}$ is finite-type (i.e., $G_{\tilde{F}}$  is a finite graph), its graph always contains eventually cyclic paths. As mentioned above, the periods of such cycles are bounded above by an integer. Let $\textsf{p}$ denote the least common multiple of all such periods. Then from the neighbor graph of $\tilde{F}$, the {\textsf{r}}-path $[\mathbf{i},\mathbf{j}]$ can be extended to eventually periodic infinite paths in the form of  $$[\mathbf{I},\mathbf{J}]=[\mathbf{i}_1 \overline{\mathbf{i}_2},\mathbf{j}_1 \overline{\mathbf{j}_2}]$$ with $\mathbf{i}_1 =\mathbf{i}\mathbf{i'}$, $\mathbf{j}_1 =\mathbf{j}\mathbf{j'}$, $|\mathbf{i}_1 |=|\mathbf{j}_1 |=k$, \  $|\mathbf{i}_2|=|\mathbf{j}_2|=k\textsf{p}$ and repeating blocks $\mathbf{i}_2$ and $\mathbf{j}_2$. Here we chose $k$ such that $k-\textsf{r}\geq {\textsf{r}}_0$ so as to  make last index of $[\mathbf{i}_1, \mathbf{j}_1]$ is an edge label of a cyclic path (keeping in mind that the period of any such path divides $\textsf{p}$) and thus we can give the repeating blocks $\mathbf{i}_2$ and $\mathbf{j}_2$. Let $l$ be the number of occurrences of $\mathbf{i}_2$, $\mathbf{j}_2$ in $\overline{\mathbf{i}_2}, \ \overline{\mathbf{j}_2}$. In (\ref{new_notation}), we now consider the finite eventually periodic indices $[\mathbf{I},\mathbf{J}]\in E_{\tilde{F}}^{k(\textsf{p}l+1)}$:
$$  d_{\mathbf{J}}:=T_{k}^{\textsf{p}l}d_{\mathbf{j}_1}+ \Sigma_{i=1}^{l}T_{k}^{\textsf{p}(l-i)}d_{\mathbf{j}_2}=d_{\mathbf{j}_1}+(T_{k}^{\textsf{p}l}-I)[(T_{k}^{\textsf{p}}-I)^{-1}d_{\mathbf{j}_2}+d_{\mathbf{j}_1}],
\quad \quad d_{\mathbf{j}_1}\in D_k, \ d_{\mathbf{j}_2}\in D_{k\textsf{p}}.
$$
$$
d_{\mathbf{I}}:=T_{k}^{\textsf{p}l}d_{\mathbf{i}_1}+ \Sigma_{i=1}^{l}T_{k}^{\textsf{p}(l-i)}d_{\mathbf{i}_2}=d_{\mathbf{i}_1}+(T_{k}^{\textsf{p}l}-I)[(T_{k}^{\textsf{p}}-I)^{-1}d_{\mathbf{i}_2}+d_{\mathbf{i}_1}],  \quad \quad d_{\mathbf{i}_1}\in D_k, \
d_{\mathbf{i}_2}\in D_{k\textsf{p}}.
$$
$$  \tilde{a}_{\mathbf{J}}:=J_{k}^{\textsf{p}l}\tilde{a}_{\mathbf{j}_1}+ \Sigma_{i=1}^{l}J_{k}^{\textsf{p}(l-i)}\tilde{a}_{\mathbf{j}_2}=\tilde{a}_{\mathbf{j}_1}+(J_{k}^{\textsf{p}l}-I)[(J_{k}^{\textsf{p}}-I)^{-1}\tilde{a}_{\mathbf{j}_2}+\tilde{a}_{\mathbf{j}_1}],
\quad \quad \tilde{a}_{\mathbf{j}_1}\in \tilde{A}_k, \ \tilde{a}_{\mathbf{j}_2}\in \tilde{A}_{k\textsf{p}}.
$$
$$
\tilde{a}_{\mathbf{I}}:=J_{k}^{\textsf{p}l}\tilde{a}_{\mathbf{i}_1}+ \Sigma_{i=1}^{l}J_{k}^{\textsf{p}(l-i)}\tilde{a}_{\mathbf{i}_2}=\tilde{a}_{\mathbf{i}_1}+(J_{k}^{\textsf{p}l}-I)[(J_{k}^{\textsf{p}}-I)^{-1}\tilde{a}_{\mathbf{i}_2}+\tilde{a}_{\mathbf{i}_1}],  \quad \quad \tilde{a}_{\mathbf{i}_1}\in \tilde{A}_k, \
\tilde{a}_{\mathbf{i}_2}\in \tilde{A}_{k\textsf{p}}.
$$
Now it is clear that 
\begin{equation}\label{hk1}\underset{l\rightarrow\infty}{\lim}|d_{\mathbf{J}}-d_{\mathbf{I}}|=\infty \Longleftrightarrow \left|(T_{k}^{\textsf{p}}-I)^{-1}(d_{\mathbf{j}_2}-d_{\mathbf{i}_2})+d_{\mathbf{j}_1}-d_{\mathbf{i}_1}\right|\neq 0
\end{equation}
 since $J_k$ in (\ref{pert_matrix}) is expanding and the modulus of its  eigenvalues are large when $k$ is large.
Similarly, 
\begin{equation}\label{h1}\underset{l\rightarrow\infty}{\lim}|h(0)|=\underset{l\rightarrow\infty}{\lim}|\tilde{a}_{\mathbf{J}}-\tilde{a}_{\mathbf{I}}|=\infty \Longleftrightarrow \left|(J_{k}^{\textsf{p}}-I)^{-1}(\tilde{a}_{\mathbf{j}_2}-\tilde{a}_{\mathbf{i}_2})+(\tilde{a}_{\mathbf{j}_1}-\tilde{a}_{\mathbf{i}_1})\right|\neq 0.
\end{equation}
On the other hand, by Lemma \ref{vertex_condition}-(a),
\begin{equation}\label{h2}|h(0)|=|\tilde{a}_{\mathbf{J}}-\tilde{a}_{\mathbf{I}}|\leq  diam(\tilde{F})
\end{equation}  for all  $l$ since $[\mathbf{I},\mathbf{J}]\in E_{\tilde{F}}^{k(\textsf{p}l+1)}$. Then by (\ref{h1}), we must have
\begin{equation}\label{h3}
 \left|(J_{k}^{\textsf{p}}-I)^{-1}(\tilde{a}_{\mathbf{j}_2}-\tilde{a}_{\mathbf{i}_2})+(\tilde{a}_{\mathbf{j}_1}-\tilde{a}_{\mathbf{i}_1})\right|= 0 \Longleftrightarrow \left|(\tilde{a}_{\mathbf{j}_2}-\tilde{a}_{\mathbf{i}_2})+(J_{k}^{\textsf{p}}-I)(\tilde{a}_{\mathbf{j}_1}-\tilde{a}_{\mathbf{i}_1})\right|= 0.
\end{equation}
By the reverse triangle inequality, this leads to
\begin{equation}\label{h4}
\frac{|\tilde{a}_{\mathbf{j}_1}-\tilde{a}_{\mathbf{i}_1} |}{||(J_{k}^{\textsf{p}}-I)^{-1}||}\leq \left|(J_{k}^{\textsf{p}}-I)(\tilde{a}_{\mathbf{j}_1}-\tilde{a}_{\mathbf{i}_1}) \right| \leq \left|\tilde{a}_{\mathbf{j}_2}-\tilde{a}_{\mathbf{i}_2} \right|\leq diam(\tilde{F})
\end{equation} in view of $\tilde{a}_{\mathbf{i}_2},\tilde{a}_{\mathbf{j}_2}\in E_{\tilde{F}}^{k{\textsf{p}}}$. 
This is possible if $\underset{k\rightarrow\infty}{\lim}\left|\tilde{a}_{\mathbf{j}_1}-\tilde{a}_{\mathbf{i}_1} \right|=0$ uniformly in $\mathbf{i}_1,\mathbf{j}_1$ with 
$[\mathbf{i}_1,\mathbf{j}_1]\in E_{\tilde{F}}^k$. (Otherwise, it would be bounded away from zero for a subsequence $\{k_\textsf{n}\}$ and for some $[\mathbf{i}_1,\mathbf{j}_1]\in E_{\tilde{F}}^{k_\textsf{n}}$, that is,  there exists $c >0$ such that 
$\left|\tilde{a}_{\mathbf{j}_1}-\tilde{a}_{\mathbf{i}_1} \right|\geq c$ for all $k_\textsf{n}=|\mathbf{j}_1|=|\mathbf{i}_1|$. But  $||J_{k_\textsf{n}}^{\textsf{p}}-I||$ grows  indefinitely  in $\textsf{n}$ (or $\underset{{\textsf{n}\rightarrow \infty}}{\lim}||(J_{k_\textsf{n}}^{\textsf{p}}-I)^{-1}||=0$), because the eigenvalues of $J_k$ approaches infinity in absolute value. Then  (\ref{h4}) would be wrong.) 
But $$\underset{k\rightarrow\infty}{\lim}\left|P^{-1}(a_{\mathbf{j}_1}-a_{\mathbf{i}_1}) \right|=\underset{k\rightarrow\infty}{\lim}\left|\tilde{a}_{\mathbf{j}_1}-\tilde{a}_{\mathbf{i}_1} \right|=0 \Longleftrightarrow \underset{k\rightarrow\infty}{\lim}\left|a_{\mathbf{j}_1}-a_{\mathbf{i}_1} \right|=0.$$ Since $a_{\mathbf{j}_1}-a_{\mathbf{i}_1}, \ a_{\mathbf{j}_2}-a_{\mathbf{i}_2}$ are integer vectors, there exists a positive integer $k_{\textsf{r}}\geq \textsf{r}$ such that $a_{\mathbf{i}_1}=a_{\mathbf{j}_1}$ for $k\geq k_{\textsf{r}}$. So for each $[\mathbf{i},\mathbf{j}]\in E^{\textsf{r}}_{\tilde{F}}$, there correspond  multi-indices  $\mathbf{i}_1, \mathbf{j}_1$ (of length $k$) satisfying the equality 
$\tilde{a}_{\mathbf{i}_1}=\tilde{a}_{\mathbf{j}_1}$. 
From this equality, 
we actually get the result.
By concatenating $l$ such pairs of indices, we have $\tilde{a}_{\mathbf{I}}=\tilde{a}_{\mathbf{J}}$: 
 
Assuming $d_{\mathbf{j}_1}=\lfloor \tilde{a}_{\mathbf{j}_1} \rfloor,
\ \  d_{\mathbf{i}_1}=\lfloor \tilde{a}_{\mathbf{i}_1}  \rfloor$ as in the proof of the first inclusion $E_k(\textsf{r})  \subseteq E^{\textsf{r}}_{\tilde{F}}$, 
we get
 $0=d_{\mathbf{j}_1}-d_{\mathbf{i}_1}  
 $ for  $k\geq k_{\textsf{r}}$. By (\ref{h3}), we have $\underset{k\rightarrow\infty}{\lim}\left|\tilde{a}_{\mathbf{j}_2}-\tilde{a}_{\mathbf{i}_2} \right|=0$ too.  
Recall that $[\mathbf{i}_2,\mathbf{j}_2]$ is a $\textsf{p}$-path  in $G_{F(J_k,\tilde{A}_k)}$. Using the same type of reasoning and induction, one can show that
$ 0=d_{\mathbf{j}_2}-d_{\mathbf{i}_2}$ for large $k$. For that, one should take into account of the fact that each part of $[\mathbf{i}_2, \mathbf{j}_2]$ is a path
in $G_{F(J,\tilde{A})}$ so that Lemma \ref{vertex_condition}-(a) can be used again. 
Thus $$ \left|(T_{k}^{\textsf{p}}-I)^{-1}(d_{\mathbf{j}_2}-d_{\mathbf{i}_2})+(d_{\mathbf{j}_1}-d_{\mathbf{i}_1})\right|=  \left|(d_{\mathbf{j}_2}-d_{\mathbf{i}_2})+(T_{k}^{\textsf{p}}-I)(d_{\mathbf{j}_1}-d_{\mathbf{i}_1})\right|= 0$$  for  $k\geq k_{\textsf{r}}$. Then  $$\underset{l\rightarrow\infty}{\lim}|d_{\mathbf{J}}-d_{\mathbf{I}}|\neq \infty$$ by (\ref{hk1}). Finally, Lemma \ref{vertex_condition}-(b) leads to
$[\mathbf{i},\mathbf{j}]\in  E_k(\textsf{r})  $  (in fact, $[\mathbf{i}_1,\mathbf{j}_1]\in  E_k$ too). 
Therefore, 
$E^{\textsf{r}}_{\tilde{F}}\subseteq E_k(\textsf{r})  $ for $k\geq k_{\textsf{r}}\geq \textsf{r}.$
\end{pf}

\medskip

We will use eventually periodic indices as in the preceding proof several times in this section.

\begin{lemma}\label{neighboring_pieces2}  Let $\tilde{F}=F(J,\tilde{A})$  and $F_k$ be its perturbations.
For any given positive integer $\textsf{r}$, there exists an integer $k_{\textsf{r}}\geq \textsf{r}$ such that for $k\geq k_{\textsf{r}},$ we have

\begin{center}
  $\overline{E}_k(\textsf{r}) =E^{\textsf{r}}_{\tilde{F}}$.
\end{center}

 \end{lemma}

\begin{pf} Note that we can rephrase this lemma as
``\textit{For large $k$, the tails of edge labels of $G_k$ coincide with some finite paths of $G_{\tilde{F}}$}''. 
The inclusion  
 $E^{\textsf{r}}_{\tilde{F}}\subseteq \overline{E}_k(\textsf{r}) $ is obtained as follows: 

If $[\mathbf{i}'_0,\mathbf{j}'_0]\in E^{\textsf{r}}_{\tilde{F}}$, 
 then $\tilde{F}_{\mathbf{i}'_0}\cap \tilde{F}_{\mathbf{j}'_0}\neq\emptyset$, that is, $\tilde{F}_{\mathbf{i}'_0}, \tilde{F}_{\mathbf{j}'_0}$ are  neighboring pieces of $\tilde{F}$. Thus  there exist multi-indices 
$ \mathbf{i}_0, \mathbf{j}_0$ such that  $\tilde{F}_{\mathbf{i}_0\mathbf{i}'_0}\cap \tilde{F}_{\mathbf{j}_0\mathbf{j}'_0}\neq \emptyset$  and $|\mathbf{i}_0\mathbf{i}'_0|=|\mathbf{j}_0\mathbf{j}'_0|=k
$ is arbitrarily large. For example, one may choose any multi-index $ \mathbf{i}_0$ so that $ \mathbf{j}_0 =\mathbf{i}_0$.

 Set 
$$\mathbf{i}_1=\mathbf{i}_0\mathbf{i}'_0, \quad \quad \quad \mathbf{j}_1=\mathbf{j}_0\mathbf{j}'_0.$$
As before, extend $[\mathbf{i}_1,\mathbf{j}_1]\in E^{k}_{\tilde{F}}$ (on the right) to eventually periodic 
finite paths in $G_{\tilde{F}}$ in the form of  $[\mathbf{i}_1\mathbf{i}_2\overline{\mathbf{i}_3},\mathbf{j}_1\mathbf{j}_2\overline{\mathbf{j}_3}]$ so that $|\mathbf{i}_2|=|\mathbf{j}_2|=k  $ and $|\mathbf{i}_3|=|\mathbf{j}_3|=k\textsf{p}$ as in the proof of  Lemma \ref{neighboring_pieces1}. 
By Lemma \ref{vertex_condition}-(a) 
\begin{equation*}\label{}   \ |\tilde{a}_{\mathbf{j}_1\mathbf{j}_2}-\tilde{a}_{\mathbf{i}_1\mathbf{i}_2}|\leq  diam(\tilde{F})
\end{equation*}  for all large $k$. Recall from the proof of Lemma \ref{neighboring_pieces1} that $l$ is the number of appearances of the repeating block $\mathbf{i}_3$ in $\overline{\mathbf{i}_3}$. Then setting $\mathbf{I}=\mathbf{i}_1\mathbf{i}_2\overline{\mathbf{i}_3},\ \mathbf{J}=\mathbf{j}_1\mathbf{j}_2\overline{\mathbf{j}_3}$, 
one can again conclude that $\underset{l\rightarrow\infty}{\lim}|d_{\mathbf{J}}-d_{\mathbf{I}}|\neq \infty$ for all large $k$
since $|\tilde{a}_{\mathbf{J}}-\tilde{a}_{\mathbf{I}}|\leq  diam(\tilde{F})$.
That gives $[\mathbf{i}_1,\mathbf{j}_1]\in E_{k}$ by  Lemma \ref{vertex_condition}-(b) leading to $E^{\textsf{r}}_{\tilde{F}}\subseteq \overline{E}_k(\textsf{r}) $.
 This argument also works the other way around. 
 
 \textit{   For large k, there are no other types of edge labels of $E_k$} :
  Conversely,  consider any pair of multi-indices $\mathbf{i}'_0,\mathbf{j}'_0$ with $[\mathbf{i}'_0,\mathbf{j}'_0]\in \overline{E}_k(\textsf{r})$. Thus $|\mathbf{i}'_0|=|\mathbf{j}'_0|=r$ and there exist 
a pair $\mathbf{i}_0,\mathbf{j}_0$ of multi-indices such that $[\mathbf{i}_0\mathbf{i}'_0,\mathbf{j}_0\mathbf{j}'_0]\in E_k$ ($k\geq r$). 
As above, extend $[\mathbf{i}_0\mathbf{i}'_0,\mathbf{j}_0\mathbf{j}'_0]$ (on the right) to eventually periodic 
finite paths in $G_k$ in the form of  $[\mathbf{I}=\mathbf{i}_0\mathbf{i}'_0\mathbf{i}_2\overline{\mathbf{i}_3},\mathbf{J}=\mathbf{j}_0\mathbf{j}'_0\mathbf{j}_2\overline{\mathbf{j}_3}]$ so that $|\mathbf{i}_2|=|\mathbf{j}_2|=k  $ and $|\mathbf{i}_3|=|\mathbf{j}_3|=k\textsf{p}$. Since these paths are in $G_k$ for each $k$, we have $\underset{l\rightarrow\infty}{\lim}|d_{\mathbf{J}}-d_{\mathbf{I}}|\neq \infty$
by Lemma \ref{vertex_condition}-(a).   As in the previous proofs,  conclude that there is an integer $k_r$ such that $d_{\mathbf{i}_0\mathbf{i}'_0}=d_{\mathbf{j}_0\mathbf{j}'_0}$, $d_{\mathbf{i}_2}=d_{\mathbf{j}_2}$, 
$d_{\mathbf{i}_3}=d_{\mathbf{j}_3}$ for   $k\geq k_r.$ Then 
$|\tilde{a}_{\mathbf{i}_0\mathbf{i}'_0}-\tilde{a}_{\mathbf{j}_0\mathbf{j}'_0}|, |\tilde{a}_{\mathbf{i}_2}-\tilde{a}_{\mathbf{j}_2}|, |\tilde{a}_{\mathbf{i}_3}-\tilde{a}_{\mathbf{j}_3}|<n$ by (\ref{floor_digit_lower_perturbation}). Recall that we can choose $P$  so that the distance between two distinct lattice points in $P^{-1}\mathbb{Z}^n$ is greater than $n$. This forces $\tilde{a}_{\mathbf{i}_0\mathbf{i}'_0}=\tilde{a}_{\mathbf{j}_0\mathbf{j}'_0}$, $\tilde{a}_{\mathbf{i}_2}=\tilde{a}_{\mathbf{j}_2}$, $\tilde{a}_{\mathbf{i}_3}=\tilde{a}_{\mathbf{j}_3}$ because $\tilde{a}_{\mathbf{i}_0\mathbf{i}'_0}, \tilde{a}_{\mathbf{j}_0\mathbf{j}'_0}$, $\tilde{a}_{\mathbf{i}_2},\tilde{a}_{\mathbf{j}_2}$, $\tilde{a}_{\mathbf{i}_3},\tilde{a}_{\mathbf{j}_3}\in \tilde{A}_k=P^{-1}A_k\subset P^{-1}\mathbb{Z}^n$.
Hence
$\underset{l\rightarrow\infty}{\lim}|\tilde{a}_{\mathbf{J}}-\tilde{a}_{\mathbf{I}}|< \infty$.  By Lemma \ref{vertex_condition}-(b), $[\mathbf{i}_0\mathbf{i}'_0, \mathbf{j}_0\mathbf{j}'_0]\in E^{\textsf{k}}_{\tilde{F}}$  for all large $k.$ This gives $[\mathbf{i}'_0, \mathbf{j}'_0]\in E^{\textsf{r}}_{\tilde{F}}$.
 That is, $\overline{E}_k(\textsf{r}) \subseteq E^{\textsf{r}}_{\tilde{F}}$ for $k\geq k_{\textsf{r}}$. 
 \end{pf}

\begin{prop}\label{neighboring_pieces4}  Let  $\tilde{F}=F(J,\tilde{A})$ and $(E_k)^{\textsf{r}}$ denote the set of \textsf{r}-paths in $E_k$. There is an integer $k_0$ such that 

\begin{center}
$E^{\textsf{r}k}_{\tilde{F}}= (E_k)^{\textsf{r}}$ \ \ ($\textsf{r}\in \mathbb{N}$) 
 \end{center} 
 
for $k\geq k_0$.
\end{prop}

\begin{pf} We again use the capital-bold-letter notation in (\ref{new_notation}) for multi-index concatenation  like
$$ \mathbf{J}=\mathbf{j}_1\cdots \mathbf{j}_\textsf{s}, \quad \quad \quad   \mathbf{I}=\mathbf{i}_1\cdots \mathbf{i}_\textsf{s},  \quad  \quad |\mathbf{j}_1|=\cdots=|\mathbf{j}_\textsf{s}|=|\mathbf{i}_1|=\cdots=|\mathbf{i}_\textsf{s}|=k.$$
It is enough to demonstrate the proof for $\textsf{r}=1$ and $\textsf{r}=2$ (the adjacency relation of the edges of $E_k$) since the general case follows from these two cases.

(a) Let $\textsf{r}=1$.  For the inclusion $E_k\subseteq E^k_{\tilde{F}}$, let $ [\mathbf{i}_1(k),\mathbf{j}_1(k)]\in E_{k}$.   
By Lemma \ref{neighboring_pieces1} and Lemma \ref{neighboring_pieces2},  for large $ k$, all edges of $E_{k}$ are in the form of $[\mathbf{i}_1
  =\mathbf{i}_0\mathbf{i}\mathbf{i}'_0, \mathbf{j}_1=\mathbf{j}_0\mathbf{j}\mathbf{j}'_0]$, where $|\mathbf{i}_1|=|\mathbf{j}_1|=k\geq \textsf{r}$ and
  $$[\mathbf{i}_0, \mathbf{j}_0], [\mathbf{i}'_0, \mathbf{j}'_0] \in E^{\textsf{r}}_{\tilde{F}}.$$

  On the other hand, as in the proof of Lemma \ref{neighboring_pieces1}, we can extend $ [\mathbf{i}_1(k),\mathbf{j}_1(k)]$ to eventually periodic finite (or infinite) paths $ [\mathbf{I}_1(k)\overline{\mathbf{I}_2(k)},\mathbf{J}_1(k)\overline{\mathbf{J}_2(k)}]$ in $G_{k}$ and we get $$\underset{k\rightarrow\infty}{\lim} \left|d_{\mathbf{J}_1}-d_{\mathbf{I}_1}\right|=0.$$  This eventually leads to $\underset{k\rightarrow\infty}{\lim} \left|d_{\mathbf{j}_1}-d_{\mathbf{i}_1}\right|=0$ (see the arguments in the proof of Lemma \ref{neighboring_pieces1}). Since $d_{\mathbf{j}_1}-d_{\mathbf{i}_1}$ are integer vectors, there exists a positive integer $k_0$ such that $d_{\mathbf{j}_1}(k)=d_{\mathbf{i}_1}(k)$ for $k\geq k_0$. 
 In view of $d_{\mathbf{j}_1}=\lfloor \tilde{a}_{\mathbf{j}_1} \rfloor,
\ \  d_{\mathbf{i}_1}=\lfloor \tilde{a}_{\mathbf{i}_1}  \rfloor$, that is possible only if $\tilde{a}_{\mathbf{j}_1}-\tilde{a}_{\mathbf{i}_1} $, as a function of $k$, 
is bounded. Now we rewrite $\tilde{a}_{\mathbf{j}_1}-\tilde{a}_{\mathbf{i}_1} $ as $$\tilde{a}_{\mathbf{j}_1}-\tilde{a}_{\mathbf{i}_1}=J_{k}^{\textsf{r}}(\tilde{a}_{\mathbf{i}_0\mathbf{i}}-\tilde{a}_{\mathbf{j}_0\mathbf{j}})
+\tilde{a}_{\mathbf{i}'_0}-\tilde{a}_{\mathbf{j}'_0}.$$ 
For this to be bounded, we must have $$\underset{k\rightarrow\infty}{\lim} \left|P^{-1}(a_{\mathbf{i}_0\mathbf{i}}-a_{\mathbf{j}_0\mathbf{j}})\right|=\underset{k\rightarrow\infty}{\lim} \left|\tilde{a}_{\mathbf{i}_0\mathbf{i}}-\tilde{a}_{\mathbf{j}_0\mathbf{j}}\right|=0$$  
because $ \underset{k\rightarrow\infty}{\lim} \left|\tilde{a}_{\mathbf{i}'_0}-\tilde{a}_{\mathbf{j}'_0}\right|\leq diam(\tilde{F})$  by Lemma \ref{vertex_condition}-(a).
This gives $\tilde{a}_{\mathbf{i}_0\mathbf{i}}=\tilde{a}_{\mathbf{j}_0\mathbf{j}}$ for large $k$ because 
$a_{\mathbf{i}_0\mathbf{i}}, \ a_{\mathbf{j}_0\mathbf{j}}$ are integer vectors (or $\tilde{a}_{\mathbf{i}_0\mathbf{i}}, \ \tilde{a}_{\mathbf{j}_0\mathbf{j}}$ are lattice points).
We then conclude that 
$$\tilde{F}_{\mathbf{i}_1}\cap \tilde{F}_{\mathbf{j}_1}=(\tilde{F}_{\mathbf{i}'_0})_{\mathbf{i}_0\mathbf{i}}\cap (\tilde{F}_{\mathbf{j}'_0})_{\mathbf{j}_0\mathbf{j}}=(\tilde{F}_{\mathbf{i}'_0})_{\mathbf{i}_0\mathbf{i}}\cap (\tilde{F}_{\mathbf{j}'_0})_{\mathbf{i}_0\mathbf{i}}=(\tilde{F}_{\mathbf{i}'_0}\cap \tilde{F}_{\mathbf{j}'_0})_{\mathbf{i}_0\mathbf{i}}\neq \emptyset$$  for all large $k$
because $[\mathbf{i}'_0, \mathbf{j}'_0] \in E^{\textsf{r}}_{\tilde{F}}\Longrightarrow\tilde{F}_{\mathbf{i}'_0}\cap \tilde{F}_{\mathbf{j}'_0}\neq \emptyset$. 
That leads to  $[\mathbf{i}_1,\mathbf{j}_1]  \in E^{k}_{\tilde{F}}$. 
\medskip

The above argument works in the other direction either. That is, 
$E^{k}_{\tilde{F}}\subseteq E_{k}$ for large $k$. As a result, there is an integer $k_0$ such that $E^k_{\tilde{F}}= E_k$ for $k\geq k_0$.  

\bigskip

(b) Let $\textsf{r}=2$. Next we can come to the issue of the adjacency relation of the edges of $E_k$. If $[\mathbf{i}_1\mathbf{i}_2, \mathbf{j}_1\mathbf{j}_2] \in E_{\tilde{F}}^{2k},$ then we use the idea in (a)
to show that $[\mathbf{i}_1\mathbf{i}_2, \mathbf{j}_1\mathbf{j}_2]\in (E_k)^2$. The difference is we work with $2$-paths in $G_k$ rather than  edges.

By using the neighbor graph $G_{\tilde{F}}$  of the original fractal $\tilde{F}$ and  sufficiently large $k$, every $2k$-path $[\mathbf{i}_1\mathbf{i}_2, \mathbf{j}_1\mathbf{j}_2] \in E_{\tilde{F}}^{2k} $ can be extended (on the right) to eventually periodic 
finite paths in the form of  $[\mathbf{i}_1\mathbf{i}_2\mathbf{i}_3\overline{\mathbf{i}_4},\mathbf{j}_1\mathbf{j}_2\mathbf{j}_3\overline{\mathbf{j}_4}] $ so that $|\mathbf{i}_3|=|\mathbf{j}_3|=k$ and $|\mathbf{i}_4|=|\mathbf{j}_4|=k\textsf{p}$.

We now consider the finite eventually periodic indices $\mathbf{I}=\mathbf{i}_1\mathbf{i}_2\mathbf{i}_3\overline{\mathbf{i}_4}$, $\mathbf{J}=\mathbf{j}_1\mathbf{j}_2\mathbf{j}_3\overline{\mathbf{j}_4}$. Let $l$ be the number of appearances of the repeating multi-index $\mathbf{j}_4$ in the block  $\overline{\mathbf{j}_4}$. Then   
$$  d_{\mathbf{J}}:=T_{k}^{\textsf{p}l}d_{\mathbf{j}_1\mathbf{j}_2\mathbf{j}_3}+ \Sigma_{i=1}^{l}T_{k}^{\textsf{p}(l-i)}d_{\mathbf{j}_4}=T_{k}^2d_{\mathbf{j}_1}+T_{k}d_{\mathbf{j}_2}+d_{\mathbf{j}_3}+
(T_{k}^{\textsf{p}l}-I)[(T_{k}^{\textsf{p}}-I)^{-1}d_{\mathbf{j}_4}+T_{k}^2d_{\mathbf{j}_1}+T_{k}d_{\mathbf{j}_2}+d_{\mathbf{j}_3}],
$$
$$
d_{\mathbf{I}}:=T_{k}^{\textsf{p}l}d_{\mathbf{i}_1\mathbf{i}_2\mathbf{i}_3}+ \Sigma_{i=1}^{l}T_{k}^{\textsf{p}(l-i)}d_{\mathbf{i}_4}=T_{k}^2d_{\mathbf{i}_1}+T_{k}d_{\mathbf{i}_2}+d_{\mathbf{i}_3}+
(T_{k}^{\textsf{p}l}-I)[(T_{k}^{\textsf{p}}-I)^{-1}d_{\mathbf{i}_4}+T_{k}^2d_{\mathbf{i}_1}+T_{k}d_{\mathbf{i}_2}+d_{\mathbf{i}_3}],
$$
where $ d_{\mathbf{j}_1},d_{\mathbf{j}_2},d_{\mathbf{j}_3}\in D_k, \ d_{\mathbf{j}_4}\in D_{k\textsf{p}}, \quad  d_{\mathbf{i}_1},d_{\mathbf{i}_2},d_{\mathbf{i}_3}\in D_k, \ d_{\mathbf{i}_4}\in D_{k\textsf{p}}$. 
To define $\tilde{a}_{\mathbf{J}},\tilde{a}_{\mathbf{I}}$, we simply replace $d$ by $\tilde{a}$ and $T$ by $J$ above. Since the rest of the argument is as in (a), 
we give a sketch of it.
 Lemma \ref{vertex_condition}-(a) gives 
$$|h(0)|  =  |\tilde{a}_{\mathbf{J}}-\tilde{a}_{\mathbf{I}}| \leq   diam(\tilde{F})
$$
for all  $l$ since $[\mathbf{I},\mathbf{J}]\in E_{\tilde{F}}^{k(\textsf{p}l+3)}$.
Since $[\mathbf{i}_1\mathbf{i}_2\mathbf{i}_3, \mathbf{j}_1\mathbf{j}_2\mathbf{j}_3] \in E_{\tilde{F}}^{3k},$ we also have $$|\tilde{a}_{\mathbf{j}_1\mathbf{j}_2\mathbf{j}_3}-\tilde{a}_{\mathbf{i}_1\mathbf{i}_2\mathbf{i}_3}|=\left| J_{k}^2(\tilde{a}_{\mathbf{j}_1}-\tilde{a}_{\mathbf{i}_1}) + J_{k}(\tilde{a}_{\mathbf{j}_2}- \tilde{a}_{\mathbf{i}_2}) + (\tilde{a}_{\mathbf{j}_3}- \tilde{a}_{\mathbf{i}_3}) \right| \leq   diam(\tilde{F}).$$
By the reverse triangle inequality,
\begin{eqnarray*}\label{} \tiny{ \left|(J_{k}^{\textsf{p}l}-I)\left[(J_{k}^{\textsf{p}}-I)^{-1} (\tilde{a}_{\mathbf{j}_4}-\tilde{a}_{\mathbf{i}_4})+ J_{k}^2(\tilde{a}_{\mathbf{j}_1}-\tilde{a}_{\mathbf{i}_1}) + J_{k}(\tilde{a}_{\mathbf{j}_2}- \tilde{a}_{\mathbf{i}_2}) + (\tilde{a}_{\mathbf{j}_3}- \tilde{a}_{\mathbf{i}_3}) \right] \right|} \leq \\
\left|\tilde{a}_{\mathbf{J}}-\tilde{a}_{\mathbf{I}}\right| + \left| J_{k}^2(\tilde{a}_{\mathbf{j}_1}-\tilde{a}_{\mathbf{i}_1}) + J_{k}(\tilde{a}_{\mathbf{j}_2}- \tilde{a}_{\mathbf{i}_2}) + (\tilde{a}_{\mathbf{j}_3}- \tilde{a}_{\mathbf{i}_3}) \right|  \leq 2 diam(\tilde{F}).
\end{eqnarray*}
In order for this to be true for all $l$, we must have
\begin{equation*}\label{}
\left[(J_{k}^{\textsf{p}}-I)^{-1} (\tilde{a}_{\mathbf{j}_4}-\tilde{a}_{\mathbf{i}_4})+ J_{k}^2(\tilde{a}_{\mathbf{j}_1}-\tilde{a}_{\mathbf{i}_1}) + J_{k}(\tilde{a}_{\mathbf{j}_2}- \tilde{a}_{\mathbf{i}_2}) + (\tilde{a}_{\mathbf{j}_3}- \tilde{a}_{\mathbf{i}_3}) \right]= 0. \\
\end{equation*}
But 
$$
 \left[(J_{k}^{\textsf{p}}-I)^{-1} (\tilde{a}_{\mathbf{j}_4}-\tilde{a}_{\mathbf{i}_4})+ J_{k}^2(\tilde{a}_{\mathbf{j}_1}-\tilde{a}_{\mathbf{i}_1}) + J_{k}(\tilde{a}_{\mathbf{j}_2}- \tilde{a}_{\mathbf{i}_2}) + (\tilde{a}_{\mathbf{j}_3}- \tilde{a}_{\mathbf{i}_3}) \right]= 0  \Longleftrightarrow
 $$
\begin{eqnarray}\label{main3} 
 \left|(\tilde{a}_{\mathbf{j}_4}-\tilde{a}_{\mathbf{i}_4})+(J_{k}^{\textsf{p}}-I)[J_{k}^2(\tilde{a}_{\mathbf{j}_1}-\tilde{a}_{\mathbf{i}_1}) + J_{k}(\tilde{a}_{\mathbf{j}_2}- \tilde{a}_{\mathbf{i}_2}) + (\tilde{a}_{\mathbf{j}_3}- \tilde{a}_{\mathbf{i}_3})] \right|= 0.
\end{eqnarray}
By the reverse triangle inequality again, this leads to
\begin{equation*}\label{}
\left| (J_{k}^{\textsf{p}}-I)\left[J_{k}^2(\tilde{a}_{\mathbf{j}_1}-\tilde{a}_{\mathbf{i}_1}) + J_{k}(\tilde{a}_{\mathbf{j}_2}- \tilde{a}_{\mathbf{i}_2}) + (\tilde{a}_{\mathbf{j}_3}- \tilde{a}_{\mathbf{i}_3})\right] \right| \leq \left|\tilde{a}_{\mathbf{j}_4}-\tilde{a}_{\mathbf{i}_4}\right|\leq diam(\tilde{F})
\end{equation*}
This is possible if $\underset{k\rightarrow\infty}{\lim}\left| J_{k}^2(\tilde{a}_{\mathbf{j}_1}-\tilde{a}_{\mathbf{i}_1}) + J_{k}(\tilde{a}_{\mathbf{j}_2}- \tilde{a}_{\mathbf{i}_2}) + (\tilde{a}_{\mathbf{j}_3}- \tilde{a}_{\mathbf{i}_3})\right|=0$ as in part (a). Thus 
$$\underset{k\rightarrow\infty}{\lim} \left| \tilde{a}_{\mathbf{j}_4}-\tilde{a}_{\mathbf{i}_4} \right|=0$$ by (\ref{main3}), and similarly
$\underset{k\rightarrow\infty}{\lim} \left| J_{k}^2(\tilde{a}_{\mathbf{j}_1}-\tilde{a}_{\mathbf{i}_1}) + J_{k}(\tilde{a}_{\mathbf{j}_2}- \tilde{a}_{\mathbf{i}_2}) \right|\leq \underset{k\rightarrow\infty}{\lim} \left|\tilde{a}_{\mathbf{j}_3}-\tilde{a}_{\mathbf{i}_3} \right|\leq   diam(\tilde{F}).$
Again that is possible 
if $\underset{k\rightarrow\infty}{\lim} \left| J_{k}(\tilde{a}_{\mathbf{j}_1}-\tilde{a}_{\mathbf{i}_1}) + (\tilde{a}_{\mathbf{j}_2}- \tilde{a}_{\mathbf{i}_2}) \right|=0.$ Then 
$\underset{k\rightarrow\infty}{\lim} \left| \tilde{a}_{\mathbf{j}_3}-\tilde{a}_{\mathbf{i}_3} \right|=0$ by (\ref{main3}) again. 
Continuing in this way,  we get  $$\underset{k\rightarrow\infty}{\lim}\left|\tilde{a}_{\mathbf{j}_1}-\tilde{a}_{\mathbf{i}_1} \right|=\underset{k\rightarrow\infty}{\lim}\left|\tilde{a}_{\mathbf{j}_2}-\tilde{a}_{\mathbf{i}_2} \right|=\underset{k\rightarrow\infty}{\lim}\left|\tilde{a}_{\mathbf{j}_3}-\tilde{a}_{\mathbf{i}_3} \right|=\underset{k\rightarrow\infty}{\lim}\left|\tilde{a}_{\mathbf{j}_4}-\tilde{a}_{\mathbf{i}_4} \right|=0.$$ 
Then $\underset{k\rightarrow\infty}{\lim}\left|a_{\mathbf{j}_1}-a_{\mathbf{i}_1} \right|=\underset{k\rightarrow\infty}{\lim}\left|a_{\mathbf{j}_2}-a_{\mathbf{i}_2} \right|=\underset{k\rightarrow\infty}{\lim}\left|a_{\mathbf{j}_3}-a_{\mathbf{i}_3} \right|=\underset{k\rightarrow\infty}{\lim}\left|a_{\mathbf{j}_4}-a_{\mathbf{i}_4} \right|=0.$
Due to the fact that   $a_{\mathbf{j}_1}-a_{\mathbf{i}_1}, \ a_{\mathbf{j}_2}-a_{\mathbf{i}_2}, \ a_{\mathbf{j}_3}-a_{\mathbf{i}_3}, \ a_{\mathbf{j}_4}-a_{\mathbf{i}_4}$ are integer vectors, there exists a positive integer $k_0$ such that $a_{\mathbf{j}_1}=a_{\mathbf{i}_1},\ a_{\mathbf{j}_2}=a_{\mathbf{i}_2}, \ a_{\mathbf{j}_3}=a_{\mathbf{i}_3}, \ a_{\mathbf{j}_4}=a_{\mathbf{i}_4}$ for $k\geq k_0$. 
Thus in view of $d_{\mathbf{j}_1}=\lfloor \tilde{a}_{\mathbf{j}_1} \rfloor,
\ \  d_{\mathbf{i}_1}=\lfloor \tilde{a}_{\mathbf{i}_1}  \rfloor$, 
we have
 $d_{\mathbf{j}_1}-d_{\mathbf{i}_1} =0 $ for  $k\geq k_0$. Similarly, $ d_{\mathbf{j}_2}-d_{\mathbf{i}_2}=0, \ d_{\mathbf{j}_3}-d_{\mathbf{i}_3}=0, \ d_{\mathbf{j}_4}-d_{\mathbf{i}_4}=0$ for $k\geq k_0$, 
and hence $$\underset{l\rightarrow\infty}{\lim}|d_{\mathbf{J}}-d_{\mathbf{I}})|\neq \infty.$$ Finally, part (a) and Lemma \ref{vertex_condition}-(b) yield
$[\mathbf{i}_1\mathbf{i}_2,\mathbf{j}_1\mathbf{j}_2]\in  (E_k)^2$. 
Therefore, $E^{2k}_{\tilde{F}}\subseteq (E_k)^2$ for $k\geq k_0.$  

The other inclusion $ (E_k)^2\subseteq E^{2k}_{\tilde{F}}$ is handled analogously.
\end{pf}

\bigskip

\begin{remark} \rm{  Proposition \ref{neighboring_pieces4} shows that not only are  the dimensions of $\tilde{F}$ and its perturbations $F_k$ related, but also their topologies. This is because  the topological properties (like simple connectedness, the disk-like or the ball-like property \cite[Theorem 9]{BM}) of $\tilde{F}$ are usually given in terms of its graph. For instance,  
 $\tilde{F}$ is connected if and only if $F_k$ is connected for all large $k$. The proof directly follows from \cite[Theorem 4.3]{KL1} and  Proposition \ref{neighboring_pieces4}. \hfill $\Box$}
\end{remark}

  Further, the following is an immediate consequence of Proposition \ref{neighboring_pieces4} and its proof. It implies that the points or subsets  of $\tilde{F}=F(J_k,\tilde{A}_k)$ and  $F_k$ given by the correspondence  (\ref{special_correspondence})
eventually (i.e. for large $k$) have the same multiple codings (index sequences).  

A \textit{graph homomorphism} from a graph $G=(V_G,E_G)$ to a graph $H=(V_H,E_H)$ is a  function from the vertex set  $V_G$ to the vertex set $V_H$ which is 
edge-preserving \cite{HN}. A \textit{graph isomorphism} of graphs $G$ and $H$ is a  bijection between the vertex sets of $G$ and $H$ which is both edge-preserving and label-preserving
\cite[\ p.424]{HHH}.

\begin{coro}\label{neighboring_pieces5} There exists an integer $k_0$ such that  there is a label-preserving homomorphism from $G_{F(J_k,\tilde{A}_k)}$  onto $G_k$ for each $k\geq k_0$. 
In particular, the sets of infinite paths of $G_{F(J_k,\tilde{A}_k)}$  and $G_k$ are the same for all large $k$.
\end{coro}

\begin{pf} Now a graph homomorphism  is apparent:
 
 By Proposition \ref{neighboring_pieces4}, when $\tilde{F}=F(J,\tilde{A})$ we have  $E^{\textsf{r}k}_{\tilde{F}}= (E_k)^{\textsf{r}}$ for $k\geq k_0$. Moreover, from the proof of the same proposition  we know that  $[\mathbf{i}_1,\mathbf{j}_1]\in  E^{k}_{\tilde{F}} \Longleftrightarrow [\mathbf{i}_1,\mathbf{j}_1]\in E_k,  \quad \quad [\mathbf{i}_1\mathbf{i}_2,\mathbf{j}_1\mathbf{j}_2]\in  E^{2k}_{\tilde{F}} \Longleftrightarrow [\mathbf{i}_1\mathbf{i}_2,\mathbf{j}_1\mathbf{j}_2]\in(E_k)^2$ for $k\geq k_0$.
Then we can define an edge-preserving and label-preserving homomorphism   
by the vertex function  
$\varphi_k: V_{F(J_k,\tilde{A}_k)} \rightarrow V_k$ with $h(x)=x+ \tilde{a}_{\mathbf{j}_1}-\tilde{a}_{\mathbf{i}_1}\underset{\varphi_k}{\longmapsto} h_k(x)=x+d_{\mathbf{j_1}}-d_{\mathbf{i_1}} $
where  $[\mathbf{i_1},\mathbf{j_1}]\in E_k$.  Then the edges  for $\varphi_k\left(V_{F(J_k,\tilde{A}_k)}\right)$ will be as follows. 
When $[\mathbf{i}_1\mathbf{i}_2,\mathbf{j}_1\mathbf{j}_2]\in(E_k)^2$, we let 
 $x+d_{\mathbf{j}_1}-d_{\mathbf{i}_1} \underset{[\mathbf{i}_2, \mathbf{j}_2]}{\longrightarrow} x+d_{\mathbf{j}_1\mathbf{j}_2}-d_{\mathbf{i}_1\mathbf{i}_2}.$ 
In this case, we naturally define  $\varphi_k\left(x+\tilde{a}_{\mathbf{j}_1\mathbf{j}_2}-\tilde{a}_{\mathbf{i}_1\mathbf{i}_2}\right)=x+d_{\mathbf{j}_1\mathbf{j}_2}-d_{\mathbf{i}_1\mathbf{i}_2}$. 
In general, if $\mathbf{I}=\mathbf{i}_1\mathbf{i}_2... \mathbf{i}_{\textsf{r}}, \ \ \mathbf{J}=\mathbf{j}_1\mathbf{j}_2...\mathbf{j}_{\textsf{r}}$ and 
$[\mathbf{I},\mathbf{J}]\in E^{\textsf{r}k}_{\tilde{F}}$, then 
$$\varphi_k\left(x+\tilde{a}_{\mathbf{J}}-\tilde{a}_{\mathbf{I}}\right)=
x+d_{\mathbf{J}}-d_{\mathbf{I}}$$ and $$x+d_{\mathbf{j}_1...\mathbf{j}_{\textsf{r-1}}}-d_{\mathbf{i}_1...\mathbf{i}_{\textsf{r-1}}} \underset{[\mathbf{i}_r, \mathbf{j}_r]}{\longrightarrow} x+d_{\mathbf{J}}-d_{\mathbf{I}}.$$ 
 So the edge-label set for $\varphi_k\left(V_{F(J_k,\tilde{A}_k)}\right)$  is the same as $E^{k}_{\tilde{F}}$, i.e. $\varphi_k$ is label-preserving. But $\varphi_k\left(V_{F(J_k,\tilde{A}_k)}\right)=V_k$ by Proposition \ref{neighboring_pieces4} again. 
 Then $\varphi_k$ is a surjective  homomorphism for $k\geq k_0$.
\end{pf}

\bigskip
Corollary \ref{neighboring_pieces5} is sufficient for our purposes. In fact, $\varphi_k$ is a graph  isomorphism for $k\geq k_0$ due to our special choice of $P$ in the introduction section: 

If $x+d_{\mathbf{J}_1}-d_{\mathbf{I}_1}=\varphi_k\left(x+\tilde{a}_{_{\mathbf{J}_1}}-\tilde{a}_{_{\mathbf{I}_1}}\right)=\varphi_k\left(x+\tilde{a}_{_{\mathbf{J}_2}}-\tilde{a}_{_{\mathbf{I}_2}}\right)
=x+d_{\mathbf{J}_2}-d_{\mathbf{I}_2}$, then $d_{\mathbf{J}_1}-d_{\mathbf{I}_1}=d_{\mathbf{J}_2}-d_{\mathbf{I}_2}$. 
 Recall that we can choose $P$  so that the distance between two distinct lattice points in $P^{-1}\mathbb{Z}^n$ is greater than $n$ (or any number as we wish). 
This forces $\tilde{a}_{_{\mathbf{J}_1}}-\tilde{a}_{_{\mathbf{I}_1}}=\tilde{a}_{_{\mathbf{J}_2}}-\tilde{a}_{_{\mathbf{I}_2}}$ for  $k\geq k_0$, 
as we observed in the last paragraph of the proof of Lemma \ref{neighboring_pieces2}.  In a similar fashion, one can show that $\varphi_k$ is a well-defined for $k\geq k_0$, i.e. if $x+\tilde{a}_{_{\mathbf{J}_1}}-\tilde{a}_{_{\mathbf{I}_1}}=x+\tilde{a}_{_{\mathbf{J}_2}}-\tilde{a}_{_{\mathbf{I}_2}}$ for possibly different paths $[\mathbf{I}_1,\mathbf{J}_1], \ [\mathbf{I}_2,\mathbf{J}_2]$, then $\varphi_k\left(x+\tilde{a}_{_{\mathbf{J}_1}}-\tilde{a}_{_{\mathbf{I}_1}}\right)=x+d_{\mathbf{J}_1}-d_{\mathbf{I}_1}=
x+d_{\mathbf{J}_2}-d_{\mathbf{I}_2}=\varphi_k\left(x+\tilde{a}_{_{\mathbf{J}_2}}-\tilde{a}_{_{\mathbf{I}_2}}\right)$.


\section{An Application: 
The Box Dimension}\label{An application}

From now on, we  may tacitly work with sufficiently large indices $k$ so that the results in the previous sections apply.

\begin{lemma}\label{graph-tile2} (cf. Proposition 3.1 in \cite{K1})
 Let $F=F(T,A)\subset \mathbb{R}^n$ be a self-affine set.
Then there exists an
integral self-affine tile $ { \Gamma }$ 
such that  $F\subset  { \Gamma }$.
\end{lemma}

\begin{pf} Let $m\geq 2$ be an integer and $I\in M_n({\mathbb{Z}})$ denote the identity matrix. Simply take $M=mI$  and
$$D'= \{ {\tiny (i_1, i_2, \cdots  i_n ) \ : \   i_1, i_2, ...,  i_n\in \{0,1,...,m-1\} }\}$$ so that
$F(M,D')$ is the unit n-cube. Then it is easy to see that there exist $ a \in \mathbb{Z}^n$ and $c_0 \in \mathbb{N}$   such that $\Gamma:=F(M,D)$  has the required property 
for $D=a+c_0D'$.
\end{pf}

\begin{remark}\label{auxiliary}   {\rm  The condition $F\subset  { \Gamma }$ is also needed to represent an integral self-affine set as a graph-directed set for dimension computations. In that case,
we may use an auxiliary tile $ { \Gamma }$ (see \cite{HLR} or Appendix), which is not necessarily a cube.}  $\hfill \Box$
\end{remark}

Now we start with our box-dimension consideration. Let $\lambda=|\lambda_1|$ be the minimum of the absolute values of the eigenvalues of $J$ (the Jordan form)  and set $m=\lfloor \lambda \rfloor$. Note that $\lambda^{k}$ is $r_1$ in Definition \ref{piece}. Here we are assuming that $2\leq m$. Take $M=m I$ ($I$ is the identity matrix) as the
generating matrix of the tile $\Gamma$ containing $\tilde{F},F_k.$
Let $\Gamma=F(M,D)$ be the cube  in
the proof of Lemma \ref{graph-tile2} and $N$ be the number of elements in $D$. 
We next define the index set $ \Sigma_r=\{\textbf{i}=(i_1,i_2,...,i_r): 1\leq  i_j \leq N \}.$
For 
$\textbf{i}\in \Sigma_r$, we set
$d_{\textbf{i}}=\Sigma_{j=1}^{r}M^{(k-j)}d_{i_j}$, where $d_{i_j}\in D$.
Let
\begin{equation*}\label{psi}
\psi_{\textbf{i}}( \Gamma ) =M^{-r}(  \Gamma  +d_{\textbf{i}}).
\end{equation*} 
Then we can cover $\tilde{F},F_k$ by the smaller n-cubes 
\begin{equation}\label{smaller n-cubes}
 \textsf{C}_{\textbf{i}}:=\psi_{\textbf{i}}( \Gamma )\subseteq \Gamma
\end{equation}
 of side lengths $c_0m^{-r}$, where  $r$ is an arbitrarily fixed positive integer. 
 These cubes are also known as \textit{mesh cubes}.

\medskip

When $k\neq 1$, it seems more appropriate to use the longer, but more precise notation $P_{l,k}(\tilde{F})$  for the pieces of $\tilde{F}=(J_k,\tilde{A}_k)=P^{-1}F$,  see Remark \ref{piece_notation}-(b).  
First, we would like to express the box dimension formula in terms of pieces $P_{l,k}(\tilde{F})$ and related sets to be defined.

Perhaps we should also mention the following, which indicates that  no distinct-level pieces $P_{l,k}(\tilde{F})$, $P_{l',k}(\tilde{F})$ with $l'\geq 2l$ can be commensurable with the same size cubes for large $k.$

\bigskip

\begin{lemma}\label{large_pieces} 
 Let  $\tilde{F}=F(J_k,\tilde{A}_k)$. If  $P_{l,k}(\tilde{F})$ is commensurable with an $n$-cube of side length $c_0m^{-r}$, then   $P_{l',k}(\tilde{F})$  is not so 
for $l'\geq 2l$ and  large $ k, l.$ In fact, $$diam(P_{l',k}(\tilde{F}))<diam(P_{l,k}(\tilde{F}))$$ for  $l'\geq 2l$ and all large $ k, l$. The same is true for $P_l(F_k)$ or related sets $S_{l,k}(\tilde{F})$,  $S_{l}(F_k)$ to be defined.
\end{lemma}
%
\begin{pf} 
It  suffices to work with  the harder case of $m_p, l_p.$
 Let $\iota:=\iota_1 $  be a positive integer. Note that 
\begin{equation}\label{iota1}
 l'_1:=l' \geq l+\iota=l_1+\iota_1 \Longrightarrow  l'_p:=(l')_p \geq l_p+\iota_p 
\end{equation} by an induction argument: Assume that $ l'_p:=(l')_p \geq l_p+\iota_p $ \ ($1\leq p \leq s-1$). Then 
\footnotesize{\begin{eqnarray*} 
l'_{p+1}=\left\lfloor l'_p\frac{\log{ m_p}}{\log{ m_{p+1}}} \right\rfloor & \geq & 
\left\lfloor (l_p+\iota_p)\frac{\log{ m_p}}{\log{ m_{p+1}}} \right\rfloor  \  \quad  \quad  \quad \quad   \mathrm{by \ the \ definition \ of} \  l'_{p+1} \ \mathrm{and  \ the \ induction \ hypothesis} \\
& \geq & \left\lfloor l_p\frac{\log{ m_p}}{\log{ m_{p+1}}} \right\rfloor+\left\lfloor \iota_p\frac{\log{ m_p}}{\log{ m_{p+1}}} \right\rfloor \  \quad  \quad   \quad  \quad       \mathrm{by \ a \ property \ of \ the \ floor \ function} \\
& \geq & l_{p+1} + \iota_{p+1} \ \  \quad \quad   \quad   \quad   \quad   \quad   \quad   \quad   \quad   \quad   \quad  \quad  \quad  \ \ \mathrm{by \ the \ definitions \ of} \  l_{p+1}, \ \iota_{p+1} \\
\end{eqnarray*}}\normalsize
which finishes the  induction argument.

Next  we justify the inequality $diam (P_{l'}(F_k))\leq  c_0m^{-r-1}$ for $l'\geq 2l$ and large $k$ :  
We have 
\footnotesize{ $$(P_{l}(F_k))_{p}\subseteq \sum_{i=1}^{{l}_{p}} T_{k,p}^{-i} (d_{{\mathbf{j}_i}})_{p}  + T_{k,p}^{-{l}_{p}} (F_k)_{p}, \quad \quad (P_{l'}(F_k))_{p}\subseteq \sum_{i=1}^{{l}'_{p}} T_{k,p}^{-i} (d_{{\mathbf{j}_i}})_{p}  + T_{k,p}^{-{l}'_{p}} (F_k)_{p} \ \ \ \ \mathrm{(when} \ l \ \mathrm{ is \ replaced \ by} \ l')$$}
\normalsize by Definition \ref{piece}-(ii). 
Hence, for $l'\geq 2l $,  (\ref{iota1}) yields 
\begin{eqnarray*}
 diam \left( (P_{l'}(F_k))_{p}\right) \leq diam \left( T_{k,p}^{-{l}'_{p}} (F_k)_{p} \right) \leq ||T_{k,p}^{-l_p} || \ diam  \left( T_{k,p}^{-l_p} (F_k)_{p}\right) 
\end{eqnarray*} so that 
\begin{equation*}\label{Gen3} \footnotesize
diam (P_{l'}(F_k)) \leq  \sum_{p=1}^{s} diam \left( (P_{l'}(F_k))_{p}\right) \leq  \sum_{p=1}^{s}  ||T_{k,p}^{-l_p} || \ diam  \left( T_{k,p}^{-l_p} (F_k)_{p} \right) 
 \leq  ||T_{k,p}^{-l_s} || \ diam  \left( F_k \right) \sum_{p=1}^{s}  ||T_{k,p}^{-l_p} ||  
\end{equation*} 
since $l_s\leq l_p$.
As in (\ref{diam_estimate4}), for large $k$, we get 
\begin{eqnarray}\label{Gen4} diam (P_{l'}(F_k)) & \leq &  ||T_{k,p}^{-l_s} || \ diam  \left( F_k \right) \sum_{p=1}^{s}  ||T_{k,p}^{-l_p} || \leq   ||T_{k,p}^{-l_s} ||  \ \overset{s}{\underset{p=1}{\Sigma}} c''  l_{p}^{\alpha_0} (m_p)^{-l_{p}}  \nonumber \\ 
& \leq & ||T_{k,p}^{-l_s} ||  \ \sum_{p=1}^{s} c'' l^{\alpha_0} (m_1)^{-l} \leq ||T_{k,p}^{-l_s} ||  \ c'' l^{\alpha_0} (|\lambda_1|^{k})^{-l}
\end{eqnarray}  
since $m_p\geq m_1=\lceil r_1 \rceil \geq r_1=|\lambda_1|^{k}.$

Assume that  $2\leq m:=\lfloor |\lambda_1|\rfloor$.  
As  in the proof of Lemma \ref{lemma_norm_estimate2},  choose a constant $\eta$ such that $1<\eta<m \leq |\lambda_1|\leq |\lambda_p|$. Then there is an integer  $k_0$ such that
$|| T_{k,p}^{-l_s}||\leq \frac{c'}{\eta^{kl_s}}<\frac{1}{\eta l_s}$  for $k>k_0$. But there is an integer  $l_0$  such that $\frac{1}{\eta l_s}\frac{c''}{c'} l^{2\alpha_0} < \frac{1}{m^2}$ 
for $l>l_0$. Because, $ l_s\approx l \frac{\log m_1}{\log m_s}$ \ by \ (\ref{diam_estimate3})  (with  $|\lambda_1|, |\lambda_s|$  replaced  by  $m_1,  m_s$) and 
$\underset{k\rightarrow \infty}{\lim} \frac{\log m_1}{\log m_s} = \frac{\log |\lambda_1|}{\log |\lambda_s|}$ (thus, $\textsf{c}_1 l\leq   l_s\leq \textsf{c}_2l$  for some constants $\textsf{c}_1, \textsf{c}_2$, or roughly $l_s\approx l$ ).
Thus 
\footnotesize  \begin{eqnarray*}
 diam \left( P_{l'}(F_k)\right) & \leq & \frac{1}{\eta l_s} c'' l^{\alpha_0} (|\lambda_1|^{k})^{-l}
  \leq \frac{1}{\eta l_s}\frac{c''}{c'} l^{2\alpha_0} diam(P_l(F_k))
 \quad \ \ \mathrm{by} \  (\ref{Gen4}), \  (\ref{piece_cylinder}) \nonumber \\
    &\leq &   \frac{1}{m^2}diam(P_l(F_k))  \leq  \frac{1}{m^2} c_0m^{-r+1}= c_0m^{-r-1} \  \quad 
    \mathrm{by \  the \  commensurability \ of} \ diam(P_l(F_k))  \  
\nonumber \\
\end{eqnarray*} \normalsize 
 for $k>k_0$ and $l>l_0$. That is,  $diam (P_{l'}(F_k)), \ diam (P_l (F_k))$ are not commensurable with the same cube. 
  These inequalities also show that  $diam (P_{l'}(F_k))< diam (P_l (F_k))$ at the same time.   
\end{pf}

\bigskip

\bigskip 

One can
always maintain
the property in the above lemma by replacing the initial generating data $J_1,\tilde{A}_1$ for $\tilde{F}$  by $J_k,\tilde{A}_k$ for some large $k.$
We now remember a couple of numbers obtained as a result of box (or piece) counting process, in which $\tilde{F}$ may be replaced by $F_k$ as usual.

\begin{itemize}
  \item \emph{$N^r(\tilde{F})$ denotes the number  of mesh cubes $\textsf{C}_{\textbf{i}}$ of  $\Gamma$ (with side length $c_0m^{-r}$) that intersect $\tilde{F}$.}
  \item \emph{In view of Lemma \ref{large_pieces},  consider the levels $l\leq \ l'\leq 2l$. Let $N_l(\tilde{F})$ be the number of all pieces 
  $P_{l', k}(\tilde{F})$ of $\tilde{F}$, which intersect and are commensurable with an $n$-cube $\textsf{C}_{\textbf{i}}$ in the previous item. }
\end{itemize}
Lemma \ref{large_pieces} takes care of the question whether the number of the pieces that intersect and are commensurable with an $n$-cube $\textsf{C}_{\textbf{i}}$ cripples the proofs. In the proofs, considering the number of all pieces $P_{l', k}(\tilde{F})$ in the last item  rather than the number of all pieces 
  $P_{l, k}(\tilde{F})$ will not affect our proofs negatively. That is why, we may ignore $l'$ and just mention the level-$l$ pieces $P_{l,k}(\tilde{F})$ sometimes.  
  
   Another question that arises is whether there is a piece $P_{l, k}(\tilde{F})$ or $P_l (F_k)$  which intersect and is commensurable with a given $n$-cube $\textsf{C}_{\textbf{i}}$ as in the first item. To answer the question, we will cover $F_{k}$ by mesh $n$-cubes  $\textsf{C}_{\textbf{i}}$  with side length $c_0(m^{r_k})^{-r}$ (rather than $c_0m^{-r}$) for some integer $r_k$. Then we use  the fact that a larger piece intersects every such small cube and that  piece is a finite union of smaller pieces. 
We also need to study the diameters  
more closely.    
%
%
\bigskip

\begin{lemma}\label{existenc_of_piece}  There is a positive integer $r_k$  such that given $n$-cube $\textsf{C}_{\textbf{i}}$ with side length $c_0m^{-r_kr}$, there is a piece  which intersect and is commensurable with $\textsf{C}_{\textbf{i}}$. 
\end{lemma}
\begin{pf}
In Definition \ref{piece},
we have mentioned that every piece is between two (cylinder) sets in $F_k$. That is, if we set $\mathbf{I}_1=\mathbf{j}_1...\mathbf{j}_l$ (with $l=l_1=\tilde{l}_1$) and $\mathbf{I}_s=\mathbf{j}_1...\mathbf{j}_{l_s}$, then
  \begin{eqnarray}\label{commensurable_piece_1}
  (F_k)_{\mathbf{I}_1}:=\sum_{i=1}^{l_{1}} T_{k}^{-i} d_{{\mathbf{j}_i}}  + T_{k}^{-l_{1}} F_k \subseteq P_l(F_k) \subseteq \sum_{i=1}^{l_{s}} T_{k}^{-i} d_{{\mathbf{j}_i}}  + T_{k}^{-l_{s}} F_k= (F_k)_{\mathbf{I}_s},
    \end{eqnarray} 
    \begin{eqnarray*}\label{commensurable_piece_2}
  (\tilde{F})_{\mathbf{I}_1}:=\sum_{i=1}^{\tilde{l}_{1}} J_{k}^{-i} d_{{\mathbf{j}_i}}  + J_{k}^{-\tilde{l}_{1}} \tilde{F} \subseteq P_l(\tilde{F}) \subseteq \sum_{i=1}^{\tilde{l}_{s}} J_{k}^{-i} d_{{\mathbf{j}_i}}  + J_{k}^{-\tilde{l}_{s}} \tilde{F}= (\tilde{F})_{\mathbf{I}_s},
    \end{eqnarray*}   
and $P_l(F_k)$  \ (or $P_l(\tilde{F})$) \ is a finite union of the cylinder sets $(F_k)_{\mathbf{I}_1}$
for certain $d_{{\mathbf{j}_i}}\in D_k$ ($1\leq i \leq l_1=l$), but $d_{{\mathbf{j}_1}},...,d_{{\mathbf{j}_{l_{s}}}}$ are fixed. It is adequate to give the proof for $F_k$. 
Then \begin{eqnarray}\label{commensurable_piece_3}
diam(T_{k}^{-l_{1}} F_k)=diam((F_k)_{\mathbf{I}_1}) \leq  diam(P_l(F_k) ) \leq diam((F_k)_{\mathbf{I}_s})=diam(T_{k}^{-l_{s}} F_k)
\end{eqnarray} by  (\ref{commensurable_piece_1}).

 \begin{itemize}
  \item[(a)] $diam(T_{k}^{-l_{1}} F_k) \approx diam(T_{k}^{-l_{s}} F_k) \approx  diam(P_l(F_k))$ : 

 We estimate the diameters $$diam((F_k)_{\mathbf{I}_1})=diam(T_{k}^{-l_{1}} F_k) , \quad \quad  \quad  diam((F_k)_{\mathbf{I}_s})=diam(T_{k}^{-l_{s}} F_k).$$ 
For that, we estimate the norms. Recall $T_{k,p}$  
in (\ref{Power_Jordan_Block1}) and (\ref{Power_Jordan_Block2}) : \quad
$$
T_{k,p}=\lceil C_p^k \rceil   \odot  \left( I+N_2 \right) \quad \quad  \mathrm{so \ that} \quad \quad
T_{k,p}^{-l_p}=\lceil C_p^k \rceil^{-l_p}   \odot  \left( I+N_2 \right)^{-l_p}.$$ 
Now   
$l_p\approx l \frac{\log |\lambda_1|}{\log |\lambda_p|}$. 
  Because $ l_p\approx l \frac{\log m_1}{\log m_p}$  by  (\ref{diam_estimate3})  (with  $|\lambda_1|, |\lambda_p|$  replaced  by  $m_1, m_p$) and 
$\underset{k\rightarrow \infty}{\lim} \frac{\log m_1}{\log m_p} = \frac{\log |\lambda_1|}{\log |\lambda_p|}$. 
It follows that 
 $||T_{k}^{-l_{1}} x||  \approx ||T_{k}^{-l_{s}} x||$, 
 where $x\in R^{n}$. 
 Hence  for every positive integer $l$, 
  \begin{eqnarray*}
  & &  diam(T_{k}^{-l_{1}} F_k)  \approx diam(T_{k}^{-l_{s}} F_k) \Longrightarrow    diam(T_{k}^{-l_{1}} F_k) \approx diam(T_{k}^{-l_{s}} F_k) \approx  diam(P_l(F_k))
  \end{eqnarray*} 
by (\ref{commensurable_piece_3}). 

  \item[(b)]  $ \frac{\textsf{c}'}{||T_{k} ||} \   diam(P_l(F_k)) \ \leq  diam(P_{l+1}(F_k))  \leq \textsf{c}''||T_{k}^{-1} || \ diam(P_l(F_k))$  :

  By the previous item,  $  diam(P_{l+1}(F_k)) \approx diam(T_{k}^{-(l+1)} F_k)$ so that 
    \footnotesize $$ \frac{\textsf{c}'}{||T_{k} ||} \ diam(T_{k}^{-l}F_k) \leq  \textsf{c}' diam(T_{k}^{-(l+1)} F_k) \leq diam(P_{l+1}(F_k)) \leq \textsf{c}'' diam(T_{k}^{-(l+1)} F_k) \leq \textsf{c}''||T_{k}^{-1} || \ diam(T_{k}^{-l}F_k).$$ \normalsize
for some positive constants $\textsf{c}', \textsf{c}''$. Again by (a), $diam(T_{k}^{-l} F_k) \approx   diam(P_l(F_k)).$ This gives 
  \begin{eqnarray*} 
\frac{\textsf{c}'}{||T_{k} ||} \   diam(P_l(F_k)) \leq 
 diam(P_{l+1}(F_k)) \leq  
  \textsf{c}''||T_{k}^{-1} || \ diam(P_l(F_k)). 
  \end{eqnarray*}  by  modification of $\textsf{c}', \textsf{c}''$.
 
  \item[(c)]    Choose a positive integer $r_k$ such that $ 1 \leq \frac{\textsf{c}'}{||T_{k} ||} m^{r_k}$, where $c'$ is as in (b). We cover $F_{k}$ by mesh $n$-cubes  $\textsf{C}_{\textbf{i}}$  with side length $c_0m^{-r_kr}$. Now given a small  mesh $n$-cube  $\textsf{C}_{\textbf{i}}$  that intersect $F_{k}$, 
there exist a level $l$ and pieces $P_l(F_k), \ P_{l+1}(F_k)$  such that they intersect $\textsf{C}_{\textbf{i}}$ and 
$$diam(P_{l+1}(F_k)) \leq c_0(m^{r_k})^{-r+1} < diam(P_l(F_k))$$ because  $P_l(F_k)$ is a finite union of  smaller pieces $ P_{l+1}(F_k)$. 
Notice that $c_0(m^{r_k})^{-r+1} < diam(P_l(F_k))$ implies that $P_l(F_k)$ is not commensurable with $\textsf{C}_{\textbf{i}}$. 
Then $$  c_0(m^{r_k})^{-r} \leq \frac{\textsf{c}'}{||T_{k} ||} c_0(m^{r_k})^{-r+1} < \frac{\textsf{c}'}{||T_{k} ||} \   diam(P_l(F_k))\leq  diam(P_{l+1}(F_k))  
$$   by (b). Consequently, $ c_0 (m^{r_k})^{-r}  <  diam(P_{l+1}(F_k))  \leq c_0(m^{r_k})^{-r+1}$, i.e. $ P_{l+1}(F_k)$
is commensurable with  $\textsf{C}_{\textbf{i}}$. 
 \end{itemize}   
\end{pf}

\bigskip

By definition, the \textit{lower box dimension} and the \textit{upper box dimension} of $\tilde{F}$, denoted by $\underline{dim}_B \tilde{F}$ and $\overline{dim}_B \tilde{F}$, are given by 
\begin{eqnarray*}\label{box_counting1}
 \underline{dim}_B \tilde{F}=\underset{r\rightarrow\infty}{\liminf}  \frac{ \log N^r(\tilde{F})}{ \log m^{r}}=\underset{r\rightarrow\infty}{\underline{\lim}}\frac{ \log N^r(\tilde{F})}{ \log m^{r}}, \\
\overline{dim}_B \tilde{F}=\underset{r\rightarrow\infty}{\limsup}  \frac{ \log N^r(\tilde{F})}{ \log m^{r}}=\underset{r\rightarrow\infty}{\overline{\lim}}\frac{ \log N^r(\tilde{F})}{ \log m^{r}}.
\end{eqnarray*} 
When $\underline{dim}_B \tilde{F}=\overline{dim}_B \tilde{F}$, this common value is denoted by ${dim}_B \tilde{F}$, called the \textit{box dimension} of $\tilde{F}$, and we say that the box dimension exists. 
We will see that $ N^r(\tilde{F})$ in these formulas can be replaced by $N_l(\tilde{F}).$ 

\begin{lemma}\label{box_piece2} 
 Let $F=F(T,A)\subset \mathbb{R}^n$ be an  integral self-affine set and  let $\tilde{F}=F(J_k,\tilde{A}_k)$. Then 
\begin{equation*}\label{ul_box0}
 \underline{dim}_B F=\underline{dim}_B \tilde{F}=\underset{l\rightarrow\infty}{\underline{\lim}}\frac{-\log N_l(\tilde{F})}{\log diam(P_{l,k}(\tilde{F}))}, \quad \ \ \ \overline{dim}_B F=\overline{dim}_B \tilde{F}= \underset{l\rightarrow\infty}{\overline{\lim}} \frac{-\log N_l(\tilde{F}) }{\log  diam(P_{l,k}(\tilde{F})) }. 
\end{equation*}
$\tilde{F}$ in these equalities may be replaced especially by $F_k$, and $P_{l,k}(\tilde{F})$ by $P_l(F_k)$.
\end{lemma}

\begin{pf} In view of Lemma \ref{existenc_of_piece}, we may assume that for every  $n$-cube intersecting $\tilde{F}$,  there is a piece $P_{l,k}(\tilde{F})$  commensurable with  and intersects that cube. That gives $N^r(\tilde{F})\leq N_l(\tilde{F})$. 
Given a piece $P_{l,k}(\tilde{F})$  commensurable with an $n$-cube, Lemma \ref{large_pieces} implies that the number of levels $l'$ such that a piece $P_{l',k}(\tilde{F})$ may be commensurable with an $n$-cube  is bounded by $l$.  For every level $l'$, the number of 
pieces that intersect and  commensurable with that cube is bounded by $U_{l'}=\alpha_1 (l')^{\alpha_2}$ by Lemma \ref{bound}. We write $l_i'=l+i-1$,  $i=1,2,...,l$, for such possible levels. 
Reset $U_r: =l \alpha_1 (2l)^{\alpha_2} \geq \sum_{i=1}^{l} U_{l'_i}=\sum_{i=1}^{l} \alpha_1 (l+i-1)^{\alpha_2}$. For an integral self-affine set
or more exactly, for $\tilde{F}=F(J_{k},\tilde{A}_{k})$,  we thus have
\begin{equation}\label{number_of_cubes1}
  \frac{N_l(\tilde{F})}{l \alpha_1 (2l )^{\alpha_2} }=\frac{N_l(\tilde{F})}{ U_r}\leq N^r(\tilde{F})\leq N_l(\tilde{F})
\end{equation}
for some positive integer constants $\alpha_1,\alpha_2$. 
Writing $|\lambda_1|=\lambda$ in Lemma  \ref{inequality_pieces}, \ we then obtain  the following estimate
\begin{equation}\label{piece_diameter1}
c' {l}^{-\alpha_0} (\lambda^{k})^{-l}\leq diam(P_{l,k}(\tilde{F})) \leq c''{l}^{\alpha_0} (\lambda^{k})^{-l}
\end{equation} 
for suitable constants $c', c'',\alpha_0.$
These inequalities,  together with the commensurability inequalities $c_1 m^{-r}< diam(P_{l,k}(\tilde{F})) \leq c_2 m^{-r}$, 
   imply
$$c_1 m^{-r}< diam(P_{l,k}(\tilde{F})) \leq c''{l}^{\alpha_0} (\lambda^{k})^{-l}, \quad  \quad  \quad c' {l}^{-\alpha_0} (\lambda^{k})^{-l}\leq diam(P_{l,k}(\tilde{F})) \leq  c_2 m^{-r}.$$
That gives
\begin{equation}\label{number_of_cubes2}
c''' {l}^{-\alpha_0} (\lambda^{k})^{l}< m^{r}\leq c'''' \ {l}^{\alpha_0} (\lambda^{k})^{l}
\end{equation}  ($c''', c''''$ are constants). Clearly, $$l\rightarrow\infty \  \ \Longleftrightarrow \ \ r\rightarrow\infty.$$
 So (\ref{number_of_cubes2})  leads to
\begin{eqnarray*}
  0=\underset{l\rightarrow\infty}{\lim}\frac{\log( l \alpha_1 (2l)^{\alpha_2})}{\log c'''' \ {2l}^{\alpha_0} (\lambda^{k})^{l}} &=& \underset{r\rightarrow\infty}{\lim}\frac{\log(l \alpha_1 (2l)^{\alpha_2})}{\log c'''' \ {l}^{\alpha_0} (\lambda^{k})^{l}} \leq \underset{r\rightarrow\infty}{\lim}\frac{\log(l \alpha_1 (2l)^{\alpha_2})}{\log m^{r}} \\
   &\leq & \underset{r\rightarrow\infty}{\lim}\frac{\log( l \alpha_1 (2l)^{\alpha_2})}{\log c''' {l}^{-\alpha_0} (\lambda^{k})^{l}}= \underset{l\rightarrow\infty}{\lim}\frac{\log( l \alpha_1 (2l)^{\alpha_2})}{\log c''' {l}^{-\alpha_0} (\lambda^{k})^{l}}=0.
\end{eqnarray*}  
by standard calculus techniques for limits (like L'H$\hat{\mathrm{o}}$pital's rule). 
This gives 
$$\underset{r\rightarrow\infty}{\lim}\frac{\log  U_r}{\log m^{r}}=\underset{r\rightarrow\infty}{\lim}\frac{\log(l \alpha_1 (2l)^{\alpha_2})}{\log m^{r}}=0.$$
Then by (\ref{number_of_cubes1}), we get
\begin{equation}\label{Box_Inequality}
   \frac{\log N_l(\tilde{F})}{\log m^{r}} -\frac{\log(l \alpha_1 (2l)^{\alpha_2})}{\log m^{r}}=  \frac{\log N_l(\tilde{F})}{\log m^{r}} -\frac{\log  U_r}{\log m^{r}} \leq \frac{ \log N^r(\tilde{F})}{ \log m^{r}}\leq \frac{\log N_l(\tilde{F})}{\log m^{r}}
\end{equation}
so that 
\begin{equation*}\label{ul_box}
\underline{dim}_B \tilde{F}=
\underset{r\rightarrow\infty}{\underline{\lim}} \frac{ \log N^r(\tilde{F})}{ \log m^{r}}=\underset{r\rightarrow\infty}{\underline{\lim}}\frac{\log N_l(\tilde{F})}{\log m^{r}}, \ \ \ \
\overline{dim}_B \tilde{F}=\underset{r\rightarrow\infty}{\overline{\lim}}
\frac{ \log N^r(\tilde{F})}{ \log m^{r}} =\underset{r\rightarrow\infty}{\overline{\lim}} \frac{\log N_l(\tilde{F})}{ \log m^{r}} 
\end{equation*}
by the properties of lower and upper limits. 
Since we are counting the pieces  $P_{l,k}(\tilde{F})$ 
  commensurable with an $n$-cube of side length $ c_0m^{-r}$, we have  $c_1 m^{-r}< diam(P_{l,k}(\tilde{F})) \leq c_2 m^{-r}$. As  observed above,  
   $ l\rightarrow\infty$ if and only if  $r\rightarrow\infty$ . 
Then,  we can further write  
 $$\underline{dim}_B \tilde{F}=\underset{l\rightarrow\infty}{\underline{\lim}}\frac{-\log N_l(\tilde{F})}{\log diam(P_{l,k}(\tilde{F}))}, \quad \ \ \ \overline{dim}_B \tilde{F}= \underset{l\rightarrow\infty}{\overline{\lim}} \frac{-\log N_l(\tilde{F}) }{\log  diam(P_{l,k}(\tilde{F})) }.$$ 
\end{pf}

\bigskip

\begin{prop}\label{monotonicity0}  Let $F=F(T,A)\subset \Bbb{R}^n$ be an integral self-affine set and let  $\tilde{F}=F(J_k,\tilde{A}_k)$. 
 Then
 $$  \overline{dim}_B F = \overline{dim}_B \tilde{F} \leq \underset{k\rightarrow\infty}{\underline{\lim}}   {dim}_B F_k.$$ 
\end{prop}

\begin{pf} 
From real analysis, we know that there is a subsequence $\{F_{k_\textsf{n}}\}$ of $F_k$ such that 
\begin{eqnarray}\label{ul_box0} \underset{\textsf{n}\rightarrow \infty}{\lim} dim_B F_{k_\textsf{n}} = \underset{k\rightarrow \infty}{\underline{\lim}} dim_B F_k.\end{eqnarray}
Since $dim_B F_{k_{\textsf{n}}}$ exists by Proposition \ref{Existence} in the Appendix, we can replace it by $\overline{dim}_B F_{k_{\textsf{n}}}$ in (\ref{ul_box0}).  By Lemma \ref{box_piece2}, we have
\begin{eqnarray}\label{ul_box1} \underset{l\rightarrow\infty}{\overline{\lim}} \frac{-\log N_l(F_{k_{\textsf{n}}}) }{\log  diam(P_l(F_{k_{\textsf{n}}})) }=\overline{dim}_B F_{k_{\textsf{n}}}
\end{eqnarray}
for every $\textsf{n}$. If we let $\textsf{n}\rightarrow\infty$ in (\ref{ul_box1}), then (\ref{ul_box0}) takes the form of 
\begin{eqnarray}\label{ul_box22}
\underset{\textsf{n}\rightarrow\infty}{\overline{\lim}}  \underset{l\rightarrow\infty}{\overline{\lim}} \  \frac{-\log N_l(F_{k_{\textsf{n}}}) }{\log  diam(P_l(F_{k_{\textsf{n}}})) }
= \underset{\textsf{n}\rightarrow\infty}{\overline{\lim}} \ \overline{dim}_B F_{k_{\textsf{n}}}= \underset{k\rightarrow\infty}{\underline{\lim}} dim_B F_k.
\end{eqnarray}
Associated with $P_l(F_{k_{\textsf{n}}})$, we define the set  $$S_{{l,k_\textsf{n}}}(\tilde{F}):= \{x\in \tilde{F}=F(J_{k_\textsf{n}},\tilde{A}_{k_\textsf{n}}) : \ x_{k_\textsf{n}}\in  P_l(F_{k_\textsf{n}})  \}\subseteq \tilde{F},$$ where $x$ is related to $x_{k_\textsf{n}}$ by (\ref{special_correspondence}). Notice that the definition of  $S_{{l,k_\textsf{n}}}(\tilde{F})$ uses $l_p$ coming from  the definition of $P_l(F_{k_\textsf{n}})$ while that of $P_{l,k_\textsf{n}}(\tilde{F})$ uses $\tilde{l}_p$. Since $l_p$ and $\tilde{l}_p$ in Definition \ref{piece} may be unequal for $p>1$ (even if $l_1=\tilde{l}_1=l$),  the sets $S_{{l,k_\textsf{n}}}(\tilde{F})$ and $P_{l,k_\textsf{n}}(\tilde{F})$ may be different. Although $S_{{l,k_\textsf{n}}}(\tilde{F})$ is not necessarily the level-$l$ piece $P_{l,k_\textsf{n}}(\tilde{F})$, we may say that it is approximately a level-$l$ piece by Lemma \ref{number_of_pieces}.   Let  $N_l^S(\tilde{F})$ be the number of the sets $S_{{l,k_\textsf{n}}}(\tilde{F})$  
 corresponding to  $ P_l(F_{k_\textsf{n}})$. By Proposition \ref{neighboring_pieces4}  or Corollary \ref{neighboring_pieces5}, (for  large $\textsf{n}$) $N_l^S(\tilde{F})$ is  equal to  the number $N_l (F_{k_\textsf{n}})$ of the pieces  $ P_l(F_{k_\textsf{n}})$ commensurable and intersecting a cube of side length $c_0m^{-r}$ in a cover of $F_{k_\textsf{n}}$ by cubes. For the same reason,  one may also assume that $S_{{l,k_\textsf{n}}}(\tilde{F})$ includes all points $x$ obtained from all different series expansions of $x_{k_\textsf{n}}$. 
 In this manner, for each large $\textsf{n}$, we make  
 \ \ $\{S_{{l,k_\textsf{n}}}(\tilde{F}) : l\in \mathbb{N} \}$  a cover of $\tilde{F}=F(J_{k_\textsf{n}},\tilde{A}_{k_\textsf{n}})$. 
 
 Let $u_{\textsf{n},l}=\frac{-\log N_l^S(\tilde{F}) }{\log  diam(S_{{l,k_\textsf{n}}}(\tilde{F})) }  $, \ \ 
$v_{\textsf{n},l}=\frac{\log  diam(S_{{l,k_\textsf{n}}}(\tilde{F})) }{\log   diam(P_l(F_{k_{\textsf{n}}}))}=\frac{-\log  diam(S_{{l,k_\textsf{n}}}(\tilde{F})) }{-\log   diam(P_l(F_{k_{\textsf{n}}}))}.$ Then $u_{\textsf{n},l},v_{\textsf{n},l}\geq 0$ for  all sufficiently large $\textsf{n}$ because $diam((S_{{l,k_\textsf{n}}}(\tilde{F})), \ \ diam(P_l(F_{k_{\textsf{n}}}))<1$ for large  $\textsf{n}$. That allows us to use the formula 
\begin{eqnarray*}
  \underset{l\rightarrow\infty}{\overline{\lim}} (u_{\textsf{n},l} v_{\textsf{n},l})=  \underset{l\rightarrow\infty}{\overline{\lim}}  \ u_{\textsf{n},l} \ \lim \ v_{\textsf{n},l}
\end{eqnarray*}
below  when $u_{\textsf{n},l},v_{\textsf{n},l}\geq 0$, \  $\lim  v_{\textsf{n},l}$ exists and $\lim v_{\textsf{n},l}>0$.
Then we can write (\ref{ul_box1}) as 
\begin{eqnarray}\label{ul_box11} \underset{l\rightarrow\infty}{\overline{\lim}} \ u_{\textsf{n},l} \ \lim \ v_{\textsf{n},l}=\underset{l\rightarrow\infty}{\overline{\lim}} (u_{\textsf{n},l}v_{\textsf{n},l}) 
=\underset{l\rightarrow\infty}{\overline{\lim}} \frac{-\log N_l^S(\tilde{F}) }{\log  diam(P_l(F_{k_{\textsf{n}}})) }=\overline{dim}_B F_{k_{\textsf{n}}}
\end{eqnarray}
for large  $\textsf{n}$ because $N_l^S(\tilde{F})=N_l (F_{k_\textsf{n}})$.
We now study the limiting behaviour of $v_{\textsf{n},l}$. In particular, we show that both $\underset{l\rightarrow\infty}{\lim}  v_{\textsf{n},l} $ and 
$\underset{n\rightarrow\infty}{\lim} \underset{l\rightarrow\infty}{\lim}  v_{\textsf{n},l}$ exist.
 Similar to (\ref{piece_diameter1}), from Lemma \ref{lemma_norm_estimate1}, Remark \ref{constant_remark} and Lemma \ref{inequality_pieces}, we infer that 
{\footnotesize\begin{eqnarray*} 
c'''k_{\textsf{n}}^{-c''} {l}^{-\alpha_0} (r_1)^{-l} \leq diam(S_{{l,k_\textsf{n}}}(\tilde{F}))  \leq c'k_{\textsf{n}}^{c''} {l}^{\alpha_0} (r_1)^{-l}, 
 \quad  \quad \quad c'''k_{\textsf{n}}^{-c''} {l}^{-\alpha_0} (m_1)^{-l} \leq diam(P_l(F_{k_{\textsf{n}}}))\leq c'k_{\textsf{n}}^{c''} {l}^{\alpha_0} (m_1)^{-l} 
\end{eqnarray*}}\normalsize
where $\lambda=|\lambda_1|$, \ $r_1= \lambda^{k_{\textsf{n}}} $ and $m_1=\lceil \lambda^{k_{\textsf{n}}} \rceil$  as before. 
 Then \begin{eqnarray*}\frac{\log   r_1}{\log   m_1} & = & \underset{l\rightarrow\infty}{\underline{\lim}}  \frac{\log   c'k_{\textsf{n}}^{-c''} +(-\alpha_0) \log   l +l \log   r_1 }{ \log   c'''k_{\textsf{n}}^{c''} +\alpha_0 \log   l +l \log   m_1}    \\
 & = &  \underset{l\rightarrow\infty}{\underline{\lim}}  \frac{\log  c'k_{\textsf{n}}^{-c''} {l}^{-\alpha_0} (r_1)^{l} }{\log   c'''k_{\textsf{n}}^{c''} {l}^{\alpha_0} (m_1)^{l}} \leq \underset{l\rightarrow\infty}{\underline{\lim}}\frac{-\log  diam(S_{{l,k_\textsf{n}}}(\tilde{F})) }{-\log   diam(P_l(F_{k_{\textsf{n}}}))} =\underset{l\rightarrow\infty}{\underline{\lim}} v_{\textsf{n},l}
 \end{eqnarray*}
 by L'H$\hat{\mathrm{o}}$pital's rule.  Letting $\textsf{n}$ approach $\infty$, we then get  $$1=\underset{\textsf{n}\rightarrow\infty}{\lim}\frac{\log   \lambda^{k_{\textsf{n}}}}{\log   \lceil \lambda^{k_{\textsf{n}}} \rceil}=\underset{\textsf{n}\rightarrow\infty}{\lim}\frac{\log   r_1}{\log   m_1} \leq  \underset{\textsf{n}\rightarrow\infty}{\underline{\lim}} \underset{l\rightarrow\infty}{\underline{\lim}}  v_{\textsf{n},l}.$$
 In a similar manner, $\underset{l\rightarrow\infty}{\overline{\lim}}    v_{\textsf{n},l} \leq \frac{\log   r_1}{\log   m_1}, \quad   \underset{\textsf{n}\rightarrow\infty}{\overline{\lim}} \underset{l\rightarrow\infty}{\overline{\lim}}   v_{\textsf{n},l}\leq 1$ so that 
 \begin{eqnarray}\label{ul_box33}
 \underset{l\rightarrow\infty}{\lim}    v_{\textsf{n},l}=\frac{\log   r_1}{\log   m_1} \quad \mathrm{and} \quad 
 \underset{\textsf{n}\rightarrow\infty}{\lim}\underset{l\rightarrow\infty}{\lim}    v_{\textsf{n},l}= \underset{n\rightarrow\infty}{\lim}\frac{\log   r_1}{\log   m_1}=1.
 \end{eqnarray}
By the remarks in 
\cite[p. 41]{F1}, 
a  box dimension formula can be given as  $\overline{dim}_B \tilde{F}=\underset{l\rightarrow\infty}{\overline{\lim}} \frac{-\log N_l^s(\tilde{F}) }{\log  diam((S_{{l,k_\textsf{n}}}(\tilde{F})) }$, where $N_l^s(\tilde{F})$ is the smallest number of sets of diameter at most $diam(S_{{l,k_\textsf{n}}}(\tilde{F}))$ that cover $\tilde{F}$. Then 
$$\overline{dim}_B \tilde{F}=\underset{l\rightarrow\infty}{\overline{\lim}} \frac{-\log N_l^s(\tilde{F}) }{\log  diam(S_{{l,k_\textsf{n}}}(\tilde{F})) } \leq 
\underset{l\rightarrow\infty}{\overline{\lim}}\frac{-\log N_l^S(\tilde{F}) }{\log  diam(S_{{l,k_\textsf{n}}}(\tilde{F})) }=\underset{l\rightarrow\infty}{\overline{\lim}}u_{\textsf{n},l}$$ for all large $\textsf{n}$.
This inequality, (\ref{ul_box22}), (\ref{ul_box11}) and  (\ref{ul_box33}) lead us to   
\begin{eqnarray*} 
\overline{dim}_B \tilde{F} & = & \overline{dim}_B \tilde{F} \ \underset{\textsf{n}\rightarrow\infty}{\overline{\lim}}\underset{l\rightarrow\infty}{\overline{\lim}} v_{\textsf{n},l}
 \leq  \underset{\textsf{n}\rightarrow\infty}{\overline{\lim}}\underset{l\rightarrow\infty}{\overline{\lim}} (u_{\textsf{n},l}v_{\textsf{n},l} )\\ 
& = & \underset{\textsf{n}\rightarrow\infty}{\overline{\lim}} \overline{dim}_B F_{k_{\textsf{n}}}= \underset{k\rightarrow\infty}{\underline{\lim}} dim_B F_k.
 \end{eqnarray*}
\end{pf}

\begin{remark} {\rm
\label{monotonicity} If we prove that $\underset{k\rightarrow\infty}{\overline{\lim}}  dim_B F_k\leq \underline{dim}_B \tilde{F},$ then Proposition \ref{monotonicity0} yields 
$\underset{k\rightarrow\infty}{\lim}  dim_B F_k=dim_B \tilde{F},$ i.e. the box dimension of $\tilde{F}$ exists. That will be the case. Due to the organization of the paper, 
we defer the proof to Section \ref{Lower_Bounds_Dimension}.} $\hfill \Box$
\end{remark}

\section{The Hausdorff Dimension}\label{theorems}

For completeness, we first recall some dimension estimates from \cite{KK2} (see Theorem 4.1 therein), which are extensions of Falconer's bounds
\cite{F3} to certain graph directed sets. Besides
handling the difficulties arising from the overlapping pieces of $F$,
another advantage of using those graph directed sets is that they can also be used to represent the boundary of $F$ when $\overset{\circ} { F }\neq \emptyset $ \cite{HLR}.

For an invertible matrix $T\in M_n({\Bbb R})$,
the singular values
$\alpha_i \ (1\leq i \leq n)$ of $T$ are
the positive square roots of the eigenvalues of $T^*T$, where $T^*$ is the
adjoint of $T$. Assume that $0<\alpha_n \leq \cdots \leq \alpha_2 \leq  \alpha_1$. The \emph{singular value function} $\phi^s(T)$ is defined for $0< s \leq n$ as
\begin{equation}\label{eqn0}
\phi^s(T)=\alpha_1 \alpha_2  ...  \alpha_{m-1}\alpha_{m}^{s-m+1},
\end{equation}
where $m$ is the integer such that $m-1 < s \leq m$. For $s>n$,
it is defined by $|\det(T)|^{\frac{s}{n}}$. Note that every integral self-affine set $F=F(T,A)\subset\mathbb{R}^{n}$ can be obtained from a graph $(V,E)$ as a graph-directed set \cite{HLR},
where $V$ is the set of vertices and $E$ is the edge set.
We define a unique nonnegative number
$u$ by
\begin{equation}\label{g1}
\lim_{\textsf{r}\rightarrow \infty} \left[\sum_{i,j=1}^\textsf{m} \sum_{\mathbf{e} \in E_{i,
j}^{\textsf{r}}}\phi^{u} (T^{-\textsf{r}}) \right]^{\frac{1}{\textsf{r}}}=1,
\end{equation}
and another number $v$ by
\begin{equation}\label{g2}
\lim_{\textsf{r} \rightarrow \infty}  \left[ \min \left\{
\sum_{j=1}^\textsf{m}\sum_{\mathbf{e} \in E_{i, j}^{\textsf{r}}}
(\phi^{v}(T^{\textsf{r}}))^{-1} \ \ \bigg| \ \ 1\leq i \leq m  \right\}  \right]^{\frac{1}{\textsf{r}}}=1,
\end{equation}
where $\min $ stands for minimum, $\# V=\textsf{m}$ and $E_{i, j}^{\textsf{r}}$ denotes the
set of directed $\textsf{r}$-paths from vertex $i$ to vertex $j$. Then we have the following.

\begin{prop}\label{thm} {\rm \cite{KK2}} For an integral self-affine set $F=F(T,A)\subset\mathbb{R}^{n}$,
 we always have
 $$v\leq {dim}_H F  \leq u\leq n.$$
\end{prop}

\bigskip

\begin{remark} {\rm  Note that there are examples such that $ {dim}_H F $ takes the value $u$ or $v$.} $\hfill \Box$
\end{remark}

Since $u$ and $v$ defined by (\ref{g1})-(\ref{g2}) can be approximated by monotonic sequences (see \cite{KK2}), Proposition
 \ref{thm}
  can be used for computational purposes and the determination of exceptional self-affine fractals.
We want to improve these results and to obtain ${dim}_H F$ as a limit.

\subsection{Tile Measures: A set-up needed for the Hausdorff dimension}\label{Tile-Measures}
Let $H^{\alpha}$ denote the $\alpha$-dimensional Hausdorff measure (see \cite{F2}). 
Then we know that 
$$dim_{H}(F)=\inf \{\alpha : H^{\alpha}(F)=0\}= \sup \{\alpha : H^{\alpha}(F)=\infty\}.$$
Now we express this fact in terms of the mesh cubes of the previous section. We note that the tile measures with respect to tiles have been arisen in our earlier study \cite{K1}. Here we need a very simple type of them, a cube.
Let $F$ be an integral self-affine set and $\Gamma=F(M,D)$ be the cube  in
the proof of Lemma \ref{graph-tile2}.

%
%
%
%
%
%

\begin{equation}\label{tilemeasure} \mathcal{T}_{\epsilon}^{\alpha}(F)=\inf \bigg \{ \sum (diam(\textsf{C}_{\textbf{i}}))^{\alpha} : \ \mathcal{C}=\{\textsf{C}_{\textbf{i}}\} \ \textrm{is a countable} \ \epsilon\textrm{-cover of } \ F  \textrm{ by n-cubes} \bigg \},
\end{equation}
where $\epsilon$-\textrm{cover} means $diam(\textsf{C}_{\textbf{i}})\leq \epsilon$. As in \cite{K1}, we denote the \textit{${\alpha}$-dimensional  tile (outer) measure of $F$ with respect to $\Gamma$} by $\mathcal{T}^{\alpha}(F)$  and  define it by 
\begin{equation}\label{tilemeasure2}
\mathcal{T}^{\alpha}(F)=\lim_{\epsilon\rightarrow 0}\mathcal{T}_{\epsilon}^{\alpha}(F)= \underset{\epsilon>0}{\sup} \ \mathcal{T}_{\epsilon}^{\alpha}(F).
\end{equation}
Here $\mathcal{T}$ refers to  tile measure. 
 Then we have
\begin{equation}\label{tiledimension}
dim_{H}(F)=\inf \{\alpha : \mathcal{T}^{\alpha}(F)=0\}= \sup \{\alpha : \mathcal{T}^{\alpha}(F)=\infty\}.
\end{equation}

For the justification of this in the special case $M=2$, we refer to $\cite[p. 33]{F1}$, where $\mathcal{T}^{\alpha}(F)$ are called net measures for this special matrix. The more general case, $M=2I\in M_n(\mathbb{Z})$ or $M=mI$ in the proof of Lemma  \ref{graph-tile2} can be handled similarly. 
Recalling Remark \ref{auxiliary}, one can replace the cube $\Gamma=F(M,D)$ by an auxiliary tile $\Gamma$ such that $F$ and $\Gamma$ use the same generating matrix, i.e., 
$F=F(T,A)$ and $F\subseteq \Gamma=F(T,D)$ (see \cite{HLR}). Thus, if $T$ has some special form like $T_k$, then the pieces of $\Gamma$ serve as the sets of a  Markov partition for $F$ as used in the proofs of Lemma 4.6, Theorem 1.1 in \cite{KP1}. That is, $dim_{H}(F)$ can be expressed by the pieces or cylinder sets of $\Gamma$. For that, one should take into account of the fact that a piece is contained in a cylinder set and their diameters are commensurable in the sense of Remark \ref{Actual_Case1}-(ii). That can also be  done with the pieces of $F$, which is explained in the next section.

\subsection{Proof of Theorem \ref{Hausdorff}: First Part}

Clearly, only  the convergence of some sequence $\{S_k\}$ of compact sets to a compact set $S$ in the Hausdorff metric does
not necessarily imply $dim_H S_k \rightarrow dim_H S$. However, that will be the case for our sequence $\{F_k\}$ of perturbed fractals. The next thing that we want to show is that $dim_H(F)$ is the limit of the sequence
 $\{\delta_k\}$
 of perturbation dimensions.
For that, we need a few 
lemmas.

Assume that   $F=F(T,A)$ is an integral self-affine set. 
Let $\Gamma=F(M,D)$ be the cube  in the proof of Lemma \ref{graph-tile2} with $M=mI$. 
We can take $m$ to be  $m=\lfloor \lambda \rfloor$, where $\lambda$ is the minimum of the absolute values of the eigenvalues of $T$. As in the previous section, cover $\tilde{F}=F(J_{k},\tilde{A}_{k})$ by n-cubes $\textsf{C}_{\mathbf{i}}$ of side length $c_0m^{-r}$ (but $r\in \mathbb{N}$ in a cover may vary in this section). Let $\mathcal{C}$ denote such an $\epsilon$-cover of  $\tilde{F}$.
Then we associate a piece cover with each $\mathcal{C}$. 

Let $P_{l,k}(\tilde{F})$ denote  a level-$l$ piece of the self-affine set $\tilde{F}=F(J_{k},\tilde{A}_{k})$ as in Definition \ref{piece}. Since the side length $c_0m^{-r}$  of $\textsf{C}_{\mathbf{i}}$ in a cover $ \mathcal{C}$, hence $r$, is allowed to vary for the Hausdorff dimension, we should perhaps emphasize this for counting numbers. For each $r$, 
 
 $\bullet$ \emph{Let $N^r_l(\tilde{F})$  denote the total number of  pieces  intersecting   and   commensurable  with  $\textsf{C}_{\mathbf{i}}\in \mathcal{C}$ of side length $c_0m^{-r}$. So this $N^r_l(\tilde{F})$ has a slightly different meaning from $N^r(\tilde{F})$ in Section \ref{An application}, where only equal-size  cubes are considered in a cover. }
 
  $\bullet$  \emph{Similarly,  in the Hausdorff dimension context, let $N^r(\tilde{F})$ be the number of the cubes of side length $c_0m^{-r}$ in a cover $ \mathcal{C}$ by cubes (not necessarily of the same size). Thus, we typically have more than one counting number corresponding to a fixed cover. }
 
 Here and below $\tilde{F}$ may be replaced by any integral self-affine set like $F_k$. Using  $\epsilon$-covers $\mathcal{C}$ of  $\tilde{F}$ by cubes, we then define
\footnotesize{ \begin{equation*} \mathcal{P}_{\epsilon}^{\alpha}(\tilde{F})=\inf_{\mathcal{C}} \bigg \{ \sum_r N^r_l
(diam( P_{l,k}(\tilde{F})))^{\alpha} : \  P_{l,k}(\tilde{F}) \ \mathrm{intersects \ and \ is \ commensurable \ with \ a \ cube \ \textsf{C}_{\mathbf{i}}\in \mathcal{C}  }  \  \bigg \}
\end{equation*} }
\normalsize
and 
$ 
\mathcal{P}^{\alpha}(\tilde{F})=\underset{\epsilon>0}{\sup} \ \mathcal{P}_{\epsilon}^{\alpha}(\tilde{F}),$ 
which may be called an  ``$\alpha$-\textit{dimensional piece (outer) measure}'' on $\tilde{F}$. Recall that $\tilde{F}=F(J_{k},\tilde{A}_{k})$. Then it is clear that the 
$\mathcal{P}^{\alpha}(\tilde{F})$
depend on $k$, but are equal for all $k.$

 Let $\mathcal{T}^{\alpha}(\tilde{F})$ be the $\alpha$-dimensional tile measure given by (\ref{tilemeasure2}) with $F$ replaced by $\tilde{F}$, and $H^{\alpha}$ be the $\alpha$-dimensional Hausdorff measure. Then the following lemma may be interpreted 
 as: \ for an integral self-affine set $F=F(T,A)\subset\mathbb{R}^{n}$,
 we \underline{almost} have the two-sided implication:  $H^{\alpha}(\tilde{F})= \mathcal{T}^{\alpha}(\tilde{F})=0,\infty$ if and only if 
 $\mathcal{P}^{\alpha}(\tilde{F})=0,\infty.$ It also gives  $dim_{H}(F)=\inf \{\alpha : \mathcal{P}^{\alpha}(\tilde{F})=0\}= \sup \{\alpha : \mathcal{P}^{\alpha}(\tilde{F})=\infty\}.$
 
   From now on, we omit the details for the actual case  and use the sloppy notation (\ref{sloppy_notation}) for simplicity since the arguments  continue to hold in that case.
 
\begin{lemma}\label{piece_measure}   
For any $\alpha, \gamma >0$,  
we have 
\begin{itemize}
                                               \item[(i)] 
                                               $\mathcal{T}^{\alpha}(\tilde{F})=0 \  \Longrightarrow \ \ \mathcal{P}^{\alpha (1+\gamma)}(\tilde{F})=0 \ \Longrightarrow \ \ \mathcal{T}^{\alpha (1+\gamma)}(\tilde{F})=0.$  
                                               \item[(ii)] $\mathcal{T}^{\alpha (1+\gamma)}(\tilde{F})=\infty  \  \Longrightarrow \ \ \mathcal{P}^{\alpha (1+\gamma)}(\tilde{F})=\infty \  \Longrightarrow \ \ \mathcal{T}^{\alpha}(\tilde{F})=\infty.$
                                           \end{itemize}
\end{lemma}

\begin{pf} (i) Let $\mathcal{C}=\{\textsf{C}_{\textbf{i}}\} $ be an $\epsilon$-cover of $\tilde{F}$ by n-cubes.
%
%
%
%
Again we consider the covers $\mathcal{C}$ and $\{\textsf{P}_{l,k}(\tilde{F})\}$ of $\tilde{F}$ as above.  As before, let $N^r(\tilde{F})$ denote the number of the $n$-cubes in $\mathcal{C}$ of side length $c_0m^{-r}$.
By Lemma \ref{bound}, for each $r$, the number of the pieces $P_{l,k}(\tilde{F})$ associated with a cube $\textsf{C}_{\mathbf{i}}\in \mathcal{C}$
is bounded by $U_r=\alpha_1 l^{\alpha_2}$ for some positive integer constants $\alpha_1,\alpha_2$. 
Then we have
\begin{equation*}\label{number_of_cubes11}
 N^r(\tilde{F})\leq  N^r_l(\tilde{F})\leq U_r N^r(\tilde{F}).  
\end{equation*}
By the  commensurability (see Definition \ref{piece}-(iii)), we have
\begin{equation*} \label{commensurability_ineq2}
c_1 m^{-r} \leq   diam (P_{l,k}(\tilde{F})) \leq c_2 m^{-r},
\end{equation*}
where $c_1,c_2\geq 1$ are independent of $l,r$. 
That implies 
$$  N^r(\tilde{F}) (c_1m^{-r})^{\alpha(1+\gamma)} \leq  N^r_l(\tilde{F})  (diam( P_{l,k}(\tilde{F})) )^{\alpha (1+\gamma)} \leq U_r N^r(\tilde{F}) (c_2m^{-r})^{\alpha(1+\gamma)} $$ for any $\gamma >0$.
But one can show that $ \underset{r\rightarrow \infty}{\lim} \frac{U_r}{(m^{r})^{\alpha\gamma}}=0$ as we did in the proof of Lemma \ref{box_piece2} for 
$ \underset{r\rightarrow \infty}{\lim} \frac{\log U_r}{\log m^{r}}=0$. So there is a constant $c>0$ so that $\frac{U_r}{(m^{r})^{\alpha\gamma}} \leq c$ for all $r$. Then 
$$  \sum_ r N^r(\tilde{F})  (c_1m^{-r})^{\alpha(1+\gamma)} \leq  \sum_ r N^r_l(\tilde{F})  (diam( P_{l,k}(\tilde{F})) )^{\alpha (1+\gamma)} \leq  cc_2^{\alpha \gamma} \sum_ r N^r(\tilde{F}) ( c_2m^{-r})^{\alpha}.$$
This leads to $\mathcal{T}^{\alpha}(\tilde{F})=0 \  \Longrightarrow \ \ \mathcal{P}^{\alpha (1+\gamma)}(\tilde{F})=0 \Longrightarrow \ \ \mathcal{T}^{\alpha (1+\gamma)}(\tilde{F})=0$. 

\bigskip

(ii) As in the first part, we start with the inequalities
 \begin{equation*}\label{number_of_cubes22}
 N^r(\tilde{F}) \leq  N^r_l(\tilde{F})  \quad  \quad  \quad \quad  \mathrm{and}  \quad \quad \quad \quad      \frac{N^r_l(\tilde{F})}{U_r}\leq  N^r(\tilde{F}). 
\end{equation*}
Therefore
\begin{equation}\label{piece_cube1}
 \sum_ r N^r(\tilde{F}) (c_1m^{-r})^{\alpha(1+\gamma)} \leq \sum_ r N^r_l(\tilde{F})  (diam( P_{l,k}(\tilde{F})) )^{\alpha (1+\gamma)} 
\end{equation}
and 
$$
  \frac{N^r_l(\tilde{F})}{U_r} (diam( P_{l,k}(\tilde{F})) )^{\alpha (1+\gamma)} 
\leq   N^r(\tilde{F})  (c_2m^{-r})^{\alpha(1+\gamma)} 
$$ 
for any $\gamma >0$.
Since  $ \underset{r\rightarrow \infty}{\lim} \frac{(m^{r})^{\alpha\gamma}}{U_r}=\infty$, there exists $r_0\in \mathbb{N}$ such that $1<\frac{(m^{r})^{\alpha\gamma}}{U_r}$ for all 
$r\geq r_0.$ 
This gives 
$$
  \frac{N^r_l(\tilde{F})}{U_r} (diam( P_{l,k}(\tilde{F})) )^{\alpha (1+\gamma)} 
\leq   N^r(\tilde{F})  (c_2m^{-r})^{\alpha(1+\gamma)} \leq   c_2^{\alpha\gamma}\frac{N^r(\tilde{F})}{U_r}  (c_2m^{-r})^{\alpha} 
$$ 
so that 
\begin{equation}\label{piece_cube2}
  N^r_l(\tilde{F}) (diam( P_{l,k}(\tilde{F})) )^{\alpha (1+\gamma)} 
\leq      c_2^{\alpha\gamma}N^r(\tilde{F})  (c_2m^{-r})^{\alpha}. 
\end{equation}
Consequently, by (\ref{piece_cube1}) and (\ref{piece_cube2})
$$ 
\sum_ r N^r(\tilde{F}) (c_1m^{-r})^{\alpha(1+\gamma)} \leq  \sum_ r  N^r_l(\tilde{F}) (diam( P_{l,k}(\tilde{F})) )^{\alpha (1+\gamma)} 
\leq  c_2^{\alpha\gamma} \sum_ r    N^r(\tilde{F})  (c_2m^{-r})^{\alpha}
$$
for 
$r\geq r_0.$ 
Finally, we obtain $\mathcal{T}^{\alpha (1+\gamma)}(\tilde{F})=\infty  \  \Longrightarrow \ \ \mathcal{P}^{\alpha (1+\gamma)}(\tilde{F})=\infty \  \Longrightarrow \ \ \mathcal{T}^{\alpha}(\tilde{F})=\infty.$
\end{pf}

\bigskip

To be used in the next result,  
 we let
\begin{equation}\label{special_correspondence0}
x=\sum_{i=1}^\infty J_{k}^{-i} \tilde{a}_{\mathbf{j}_i}=(\tilde{a}_{\mathbf{j}_i})\in \tilde{F}, \ \ \ \ \ \ \ \ \ \ \ \ x_k=\sum_{i=1}^\infty T_{k}^{-i} d_{\mathbf{j}_i}=(d_{\mathbf{j}_i})\in  F_k,
\end{equation}
where $\tilde{a}_{\mathbf{j}_i}\in \tilde{A}_k$ and $d_{\mathbf{j}_i}\in D_{k}$. Thus $x$ and 
$x_k$ have the same multi-index sequence. We call $(\tilde{a}_{\mathbf{j}_i})$ the sequence notation for $x$ and there may be several sequences $(\tilde{a}_{\mathbf{j}_i})$ representing $x$.  Therefore, for each $x$, there may be several points $x_k$ corresponding to $x$. But, that does not occur for large $k$ by 
Corollary \ref{neighboring_pieces5}.

 To each countable $\epsilon$-cover $\{P_{l,k}(\tilde{F})\}$
of $\tilde{F}=F(J_k,\tilde{A}_k)$  
we can associate
a cover $ \{ S_l(F_{k}) \} $ of $F_{k}$ by certain sets $S_l(F_{k})\subset F_{k}$ through (\ref{special_correspondence0}).
More explicitly, for each set $P_{l,k}(\tilde{F})$
in this cover,
$ \{S_l(F_{k}) \} $ contains the corresponding set
\begin{equation}\label{sets_S}
S_l(F_{k}):=\{ x_k=(d_{\mathbf{j}_i})\in F_k \ : \ (\tilde{a}_{\mathbf{j}_i})\in P_{l,k}(\tilde{F}) \} \ \  
\end{equation}
 and it consists of such sets only.   When defining $S_l(F_{k})$, we take into account of all series expansions or sequences $(\tilde{a}_{\mathbf{j}_i})$ representing each point $x\in P_{l,k}(\tilde{F})\subset \tilde{F} $ so that $\{S_l(F_{k})\}$ covers $F_k$. We  employ this kind of sets often. 
 
For each positive integer $l$, 
\begin{itemize}
  \item \emph{let $N_l^S(F_k)$ denote the number of the sets $S_{l}(F_k)$  covering  $F_k$.  }
\end{itemize}
\noindent We  next study the relationships between   $S_{l}(F_k)$ and  $P_l(F_k)$, 
or  $ N_l^S(F_k)$ and $N_l(F_k)$, the number of level-$l$ pieces of $F_k$ here. 
As before, we let
\begin{equation*} \mathcal{S}_{\epsilon}^{\alpha}(F_{k})=\inf \bigg \{ \sum_l
 N_l^S(F_k)(diam(S_l(F_{k})))^{\alpha} : \  \{S_l(F_{k}) \}  \ {\footnotesize  \textrm{is a countable} \ \epsilon\textrm{-cover of } \ F_{k} \ \textrm{by \ the sets \ in \ (\ref{sets_S}) } }
  \bigg \}
\end{equation*}
and $ \mathcal{S}^{\alpha }(F_{k})= \underset{\epsilon >0}{\sup } \ \mathcal{S}_{\epsilon }^{\alpha }(F_{k})$. 
We may consider only finite covers $\{P_{l,k}(\tilde{F}) \}$, \ $ \{S_l(F_{k}) \} $ because $\tilde{F}, \ F_{k_j}$ are compact sets. 

\bigskip

 In the following lemma, we first establish a one-to-finite correspondence between the level-$l$ pieces of $F_k$ and the sets  $S_l(F_{k})$  respectively. That is, for large $k$, we show that for each level-$l$ piece of $F_k$, there are at most a finite number of the sets  $S_l(F_{k})$. 
 Secondly, we compare the diameters of $P_l(F_k)$  and the sets $S_{l}(F_k)\subseteq  F_k$, which correspond to $P_{l,k}(\tilde{F}).$

\bigskip

\begin{lemma}\label{diameter_dimension}  Let $F=F(T,A)\subset \mathbb{R}^n$ be an integral self-affine set. As usual, we assume that $A\neq \{0\}.$
Let $\tilde{F}=F(J_k,\tilde{A}_k)=P^{-1}F$ and $F_k=F(T_k,D_k)$ be a
 lower perturbation of $\tilde{F}$. Then there exists a positive integer $k_0$ and  a positive constant  $c<1$ (depending on $k$ only)   such that 
 $$N_l(F_k)\leq N_l^S(F_k) \leq q^{ks^2}N_l(F_k) \quad
\mathrm{ and } \quad 
c \cdot diam(P_l(F_k)) \leq diam(S_l(F_k)) \leq   diam(P_l(F_k))$$ for $ k\geq k_0$.
\end{lemma}

\begin{pf}
To analyse these sets, as in the proof of Lemma \ref{bound}, we consider the points of $P_{l,k}(\tilde{F})$ restricted to the  Jordan blocks  $J_{k,p}$: Choose an arbitrary  $p\in\{1,2,...,s\}$. Let $(\tilde{F})_{p}\in \mathbb{R}^{n_p} $ or $(\tilde{F})_{p}\in \mathbb{R}^{2n_p} $ 
denote the set of points of $\tilde{F}$ restricted to the corresponding block  $J_{k,p}.$ In this notation, $\tilde{F}$ may be replaced by other sets.  Then by the definition in (\ref{sets_S}), we have
\begin{equation}\label{piece_correspondence1}
 (P_{l,k}(\tilde{F}))_{p}=\sum_{i=1}^{\tilde{l}_{p}} J_{k,p}^{-i} (\tilde{a}_{{\mathbf{j}_i}})_{p}  + J_{k,p}^{-\tilde{l}_{p}} (\tilde{F})_{p}, \quad   \quad  \quad
 (S_{l}(F_k))_{p}=\sum_{i=1}^{\tilde{l}_{p}} T_{k,p}^{-i} (d_{{\mathbf{j}_i}})_{p}  + T_{k,p}^{-\tilde{l}_{p}} (F_k)_{p}.
\end{equation} 
We now count  pieces or sets with respect to the defining digits in $\sum_{i=1}^{\tilde{l}_{p}} J_{k,p}^{-i} (\tilde{a}_{{\mathbf{j}_i}})_{p}$ or $\sum_{i=1}^{\tilde{l}_{p}} T_{k,p}^{-i} (d_{{\mathbf{j}_i}})_{p} $ regardless of overlaps of pieces or sets.  
By Lemma \ref{number_of_pieces},  there is a positive integer $k_0$ (independent of $l$) 
such that  $l_p(k)=\tilde{l}_p(k)-\iota_p(k)\leq \tilde{l}_p(k)$  with $\iota_p\leq p-1\leq s-1$ ($\iota_p$ is a nonnegative integer)
for $k\geq k_0$.  That gives 
\begin{equation}\label{piece_correspondence2}  (S_{l}(F_k))_{p}=\sum_{i=1}^{l_p} T_{k,p}^{-i} (d_{{\mathbf{j}_i}})_{p}  +\sum_{i=l_p+1}^{\tilde{l}_{p}} T_{k,p}^{-i} (d_{{\mathbf{j}_i}})_{p}  +  T_{k,p}^{-\iota_p}\left( T_{k,p}^{-l_p} (F_k)_{p}\right).
\end{equation} 
If we  extract  some parts of this summation, we obtain  
\begin{equation}\label{piece_correspondence22}(P_l(F_k))_{p}=\sum_{i=1}^{l_p} T_{k,p}^{-i} (d_{{\mathbf{j}_i}})_{p} + T_{k,p}^{-l_p} (F_k)_{p} \quad \quad \mathrm{ and}  \quad  \quad 
diam \left( P_l(F_k))_{p}\right)=  diam  \left( T_{k,p}^{-l_p} (F_k)_{p}\right).
\end{equation}
Recall that $\tilde{A}_k=D_k=q^k$. Since $0\leq \tilde{l}_{p}-l_p=\iota_p\leq s-1,$ it also follows from (\ref{piece_correspondence2}) that the number of choices for the digits 
in the sum $\sum_{i=l_p+1}^{\tilde{l}_{p}} T_{k,p}^{-i} (d_{{\mathbf{j}_i}})_{p} $ is less than $(q^k)^s.$ Therefore, 
for the $p$-th block $(P_l(F_k))_{p}$ of $P_l(F_k)$, there correspond less than $(q^k)^s$  \ $p$-th blocks $(S_{l}(F_k))_{p}$ of $S_{l}(F_k)$. 
Since $1\leq p\leq s$,  it follows that the number of the sets $S_{l}(F_k)$ or pieces $P_{l,k}(\tilde{F})$ corresponding to each piece $P_l(F_k)$  (of the same level $l$) is not greater than $(q^k)^{s^2}$.
That yields  
\begin{eqnarray*}\label{counting_inequality1} N_l(F_k)\leq 
N_l^S(F_k) \leq q^{ks^2}N_l(F_k)
\end{eqnarray*} 
for $k\geq k_0$.

 Next we will observe that $diam(S_{l}(F_k))\leq diam(P_l(F_k))$ for large enough $k_0$.  But this is easy by the study of the  $p$-th blocks $(S_{l}(F_k))_{p}$,  $(P_l(F_k))_{p}$ :  \ Let $x,y\in F_k.$ Since $T_k$ is block-diagonal, we have
\begin{eqnarray}
\label{diam_squared1}(diam(S_{l}(F_k)))^2=\sup_{x,y\in F_k}\sum_{p=1}^{s} \left| T_{k,p}^{-\iota_p}\left( T_{k,p}^{-l_p} (x-y)_{p}\right) \right|^2
\end{eqnarray}
by (\ref{piece_correspondence1}).
If $\iota_p=0 $ in this summation, then $\left| T_{k,p}^{-\iota_p}\left( T_{k,p}^{-l_p} (x-y)_{p} \right) \right|^2=\left|  T_{k,p}^{-l_p} (x-y)_{p}\right|^2$; 
otherwise, by Lemma \ref{lemma_norm_estimate2}, there exist  constants $c'>0$ and $ \eta>1$ such that
$ \left|\left| T_{k,p}^{-\iota_p} \right|\right|\leq \frac{c'}{(\eta^k)^{\iota_p}}.$ Hence  $\left|\left| T_{k,p}^{-\iota_p} \right|\right|<1$ for sufficiently large $k\geq k_0$. 
Thus, if $\iota_p\neq 0 $ in the above summation, then $$\left| T_{k,p}^{-\iota_p}\left( T_{k,p}^{-l_p} (x-y)_{p} \right) \right|^2\leq \left|\left| T_{k,p}^{-\iota_p}\right|\right|^2 \left|  T_{k,p}^{-l_p} (x-y)_{p}\right|^2<\left| T_{k,p}^{-l_p} (x-y)_{p}\right|^2$$  for $k\geq k_0.$ 
Eventually, by (\ref{piece_correspondence22}), we get $$\sum_{p=1}^{s} \left|T_{k,p}^{-\iota_p}\left( T_{k,p}^{-l_p} (x-y)_{p}\right)  \right|^2\leq \sum_{p=1}^{s} \left|  T_{k,p}^{-l_p} (x-y)_{p}\right|^2\leq 
( diam(P_l(F_k)))^2$$
for any $x,y\in F_k.$  Then  $$(diam(S_{l}(F_k))^2=\sup_{x,y\in F_k}\sum_{p=1}^{s} \left| T_{k,p}^{-\iota_p}\left( T_{k,p}^{-l_p} (x-y)_{p} \right) \right|^2\leq ( diam(P_l(F_k)))^2$$ 
implying that $diam(S_{l}(F_k))\leq diam(P_l(F_k))$ for $k\geq k_0$.

 As for the lower bound for $diam(S_{l}(F_k))$,  
 we can make use of the norm inequality 
 \begin{eqnarray}\label{diam_squared2}
 \frac{ \left| T_{k,p}^{-l_p} (x-y)_{p} \right|}{\left|\left|T_{k,p}^{\iota_p} \right|\right|}\leq  \left| T_{k,p}^{-\iota_p}\left( T_{k,p}^{-l_p} (x-y)_{p}\right) \right|\end{eqnarray}
  in linear algebra \cite{L}.  Let $c=\underset{1\leq p \leq s}{\min}\frac{1}{||T_{k,p}^{\iota_p}||}$ with $c$ depends only on  $||T_{k,p}^{\iota_p}||$, or just $k$.
 Then 
$$c^2\sum_{p=1}^{s}\left| \left( T_{k,p}^{-l_p} (x-y)_{p}\right) \right|^2\leq \sum_{p=1}^{s} \left| T_{k,p}^{-\iota_p}\left( T_{k,p}^{-l_p} (x-y)_{p}\right) \right|^2\leq (diam(S_{l}(F_k)))^2$$ by (\ref{diam_squared1}) and (\ref{diam_squared2}). Hence 
$$c^2(diam ( P_l(F_k)))^2=c^2\sup_{x,y\in F_k}\sum_{p=1}^{s}\left| \left( T_{k,p}^{-l_p} (x-y)_{p}\right) \right|^2\leq (diam(S_{l}(F_k)))^2.$$ 
That gives the required inequality
\begin{eqnarray*}\label{piece_correspondence5} 
c \cdot diam ( P_l(F_k))\leq diam(S_{l}(F_k))
\end{eqnarray*}
for $k\geq k_0$.
\end{pf}

\bigskip

Another inequality about counting numbers of covering sets is the following.

\begin{lemma}\label{counting_number_inequality_3} 
$N_l^S(F_k) \leq N^r_l(\tilde{F}) $ when $k$ is large enough. 
 
\end{lemma}

\begin{pf}  Assume that $S^1_l(F_k)$ and $S^2_l(F_k)$ are distinct sets defined by $P^1_{l,k}(\tilde{F}), \  P^2_{l,k}(\tilde{F})$ through (\ref{sets_S}). 
Since $S^1_l(F_k)\neq S^2_l(F_k)$, one of these sets contains a point, say $y_1\in S^1_l(F_k)$, that is not in $S^2_l(F_k)$. Then 
$(d_{{\mathbf{i}_i}})=y_1\neq y_2=(d_{{\mathbf{j}_i}})$
for each $y_2\in S^2_l(F_k).$  
Let $x_1=(\tilde{a}_{\mathbf{i}_i})\in P^1_{l,k}(\tilde{F})$, 
$x_2=(\tilde{a}_{\mathbf{j}_i}) \in P^2_{l,k}(\tilde{F})$ 
be the points corresponding to $y_1, y_2$ respectively. We may use the symbolic codes $\mathbf{I}=\mathbf{i}_1\cdots \mathbf{i}_i\cdots $, \quad 
$\mathbf{J}=\mathbf{j}_1\cdots \mathbf{j}_i\cdots $ for the infinite index sequences corresponding to these points respectively. 
Let $G_k$ be the neighbor graph of $F_k$, and $G_{F(J_k,\tilde{A}_k)}$ be the neighbor graph of $\tilde{F}=F(J_k,\tilde{A}_k)$. Then  $[\mathbf{I},\mathbf{J}]$ is not an infinite path in $G_k$ because of  $y_1\neq y_2$. By Corollary \ref{neighboring_pieces5},
$[\mathbf{I},\mathbf{J}]$ cannot be an infinite path in 
$G_{F(J_k,\tilde{A}_k)}$ for $k\geq k_0$ either. Then we get $x_1\neq x_2$ for each $x_2\in P^2_{l,k}(\tilde{F}).$ That is,   
$S^1_l(F_k)\neq S^2_l(F_k)\Longrightarrow P^1_{l,k}(\tilde{F})\neq P^2_{l,k}(\tilde{F}).$ Therefore, $N_l^S(F_k) \leq N^r_l(\tilde{F})$ for $k\geq k_0$.
\end{pf}

\bigskip

Actually,  the argument of the preceding lemma can be used to show that the reversed inequality also holds for large $k$. Therefore,  the inequality in the lemma can be replaced by equality.   

As before, let $H^{\alpha}(\tilde{F})$ be the $\alpha$-dimensional  Hausdorff measure of $\tilde{F}$. The following lemma establishes some relations among the Hausdorff, the tile and the piece measures of $\tilde{F}$ and its perturbations $ F_{k}$. It may be interpreted as: \ for an integral self-affine set $F=F(T,A)\subset\mathbb{R}^{n}$,
 the  Hausdorff and the piece measures of $\tilde{F}$ are \underline{almost} the limits of  the tile measures of its perturbations.

\begin{lemma}\label{measure_convergence} 
Let $\alpha, \gamma>0$. Assume that $\{F_{k_\textsf{n}}\}$ is a subsequence of  $ \{F_{k}\}$. Then the following hold.
\begin{itemize}
  \item[(i)]   If $\mathcal{T}^{\alpha}(F_{k_\textsf{n}})=0$ for $\textsf{n}\geq N_1$, then $H^{\alpha(1+\gamma)}(\tilde{F})=\lim_{\textsf{n}\rightarrow \infty} \mathcal{T}^{\alpha(1+\gamma)}(F_{k_\textsf{n}}).$
  \item[(ii)]  If $\mathcal{T}^{\alpha(1+\gamma)}(F_{k_\textsf{n}})=\infty$ for $\textsf{n}\geq N_2$, then 
  $\mathcal{P}^{\alpha}(\tilde{F})=\lim_{\textsf{n}\rightarrow \infty} \mathcal{T}^{\alpha}(F_{k_\textsf{n}}).$
\end{itemize}

\end{lemma}

\begin{pf} 
For tile measures $\mathcal{T}^{\alpha}$, remember that we cover $\tilde{F}$ or $F_{k}$ by cubes of side length $c_0m^{-r}$ with $m=\lfloor \lambda \rfloor=\lfloor |\lambda_1| \rfloor$ see
  (\ref{smaller n-cubes}). 
  
(i) We assume that $\{F_{k_\textsf{n}}\}$ is a subsequence of  $ \{F_{k}\}$ such that
$\mathcal{T}^{\alpha}(F_{k_\textsf{n}})=0$ for $\textsf{n}\geq N_1$. 
In order to prove $H^{\alpha(1+\gamma)}(\tilde{F})=0$, it is enough to show that there is a sequence
$\mathcal{C}_{k_\textsf{n}}(\tilde{F})$ of covers of $\tilde{F}$  by some sets $S_{{l,k_\textsf{n}}}(\tilde{F})$ such that
$$\lim_{\textsf{n}\rightarrow \infty}   \sum_{S_{{l,k_\textsf{n}}}(\tilde{F})\in \mathcal{C}_{k_\textsf{n}}(\tilde{F})} (diam(S_{{l,k_\textsf{n}}}(\tilde{F})))^{\alpha(1+2\gamma)}  =0.$$

By  Lemma \ref{piece_measure}-(i), the hypothesis $\mathcal{T}^{\alpha}(F_{k_\textsf{n}})=0$  implies that $\mathcal{P}^{\alpha (1+\gamma)}(F_{k_\textsf{n}})=0$ for $\textsf{n}\geq N_1$. 
Then there exists a finite cover $ \mathcal{C}(F_{k_\textsf{n}}) $
of   $F_{k_\textsf{n}}$
by the pieces  $ P_l(F_{k_\textsf{n}})$ 
such that
\begin{equation}\label{limit1} \sum_r  N^r_l(F_{k_\textsf{n}})
 (diam( P_l(F_{k_\textsf{n}})))^{\alpha(1+\gamma)}  <  \frac{1}{\textsf{n}}
 \end{equation}
 for $\textsf{n}\geq N_1$.  
Corresponding to each piece $ P_l(F_{k_\textsf{n}})$ in the cover  $ \mathcal{C}(F_{k_\textsf{n}}) $,
we consider the set  $$S_{{l,k_\textsf{n}}}(\tilde{F}):= \{x\in \tilde{F}=F(J_{k_\textsf{n}},\tilde{A}_{k_\textsf{n}}) : \ x_{k_\textsf{n}}\in  P_l(F_{k_\textsf{n}})  \},$$ where $x$ is related to $x_{k_\textsf{n}}$ by (\ref{special_correspondence0}).   This $S_{{l,k_\textsf{n}}}(\tilde{F})$ is a subset of $\tilde{F}$ whereas $S_l(F_{k})$ in (\ref{sets_S}) 
is a subset of $F_{k}.$ Thus we should not be confused by that.  So $S_{{l,k_\textsf{n}}}(\tilde{F})$ depends on symbolic codings or series expansions of $x_{k_\textsf{n}}\in  P_l(F_{k_\textsf{n}})$, and therefore, there may correspond more than one such $S_{{l,k_\textsf{n}}}(\tilde{F})$ to each $ P_l(F_{k_\textsf{n}})$.   
But, by 
Corollary \ref{neighboring_pieces5},
 for all
 large 
 $\textsf{n}$ in (\ref{limit1}),   the number $N_l^S(\tilde{F})$ of the sets $S_{{l,k_\textsf{n}}}(\tilde{F})$  
 corresponding to  $ P_l(F_{k_\textsf{n}})$ is  equal to  the number $N^r_l (F_{k_\textsf{n}})$ of the pieces  $ P_l(F_{k_\textsf{n}})$ commensurable and intersecting a cube of side length $c_0m^{-r}$ in a cover of $F_{k_\textsf{n}}$ by cubes. That is why,  may also assume that $S_{{l,k_\textsf{n}}}(\tilde{F})$ includes all points $x$ obtained from all different series expansions of $x_{k_\textsf{n}}$, but they correspond to the same point for large $\textsf{n}$.
We express the relationship between $ P_l(F_{k_\textsf{n}})$ and $S_{{l,k_\textsf{n}}}(\tilde{F})$ by
$ P_l(F_{k_\textsf{n}}) \longmapsto S_{{l,k_\textsf{n}}}(\tilde{F})$. 
%
Let 
$$\mathcal{C}_{k_\textsf{n}}(\tilde{F})=\{ S_{{l,k_\textsf{n}}}(\tilde{F}) :  P_l(F_{k_\textsf{n}}) \longmapsto S_{{l,k_\textsf{n}}}(\tilde{F})  \quad \mathrm{and}  \quad  P_l(F_{k_\textsf{n}}) \in \mathcal{C}(F_{k_\textsf{n}})  \}.$$

 In this manner, for each large $\textsf{n}\geq N_1$, we make  
 \ \ $\mathcal{C}_{k_\textsf{n}}(\tilde{F})$  a cover of $\tilde{F}$.
 Then for suitable constants and all large $\textsf{n}$,  we have 
 \begin{equation}\label{limit3}
 \left|(diam( S_{{l,k_\textsf{n}}}(\tilde{F})))^{\alpha(1+\gamma)}-(diam(  P_l(F_{k_\textsf{n}}) ))^{\alpha(1+\gamma)} \right| < ( c' \ {l}^{\alpha_0}|\lambda_1|^{-lk_\textsf{n}})^{\alpha(1+\gamma)}
\end{equation}
by Lemma \ref{lemma_norm_estimate1}. 
 Also, by Lemma \ref{inequality_pieces}, $c' l^{-\alpha_0} (|\lambda_1|^{k_\textsf{n}})^{-l}\leq diam( P_l(F_{k_\textsf{n}}) ) $ for suitable constants $c',  \alpha_0.$ \quad 
The last inequality together with (\ref{limit3}) gives $$(diam( S_{{l,k_\textsf{n}}}(\tilde{F})))^{\alpha(1+\gamma)}\leq (c \ {l}^{2\alpha_0} diam(  P_l(F_{k_\textsf{n}}) ))^{\alpha(1+\gamma)}$$ 
for some constant $c.$
As mentioned above, for $\tilde{F}=F(J_{k_\textsf{n}},\tilde{A}_{k_\textsf{n}})$, we have $N_l^S(\tilde{F})= N^r_l (F_{k_\textsf{n}}) $ for large $\textsf{n}$.  
  Then, 
   we get
\noindent {\footnotesize \begin{eqnarray}\label{HIn}
   \sum_l N_l^S(\tilde{F})    (diam( S_{{l,k_\textsf{n}}}(\tilde{F})))^{\alpha(1+2\gamma)}  
   &\leq &     \sum_r  N^r_l (F_{k_\textsf{n}})   (diam( P_l(F_{k_\textsf{n}})))^{\alpha(1+\gamma)} (c\ {l}^{2\alpha_0} )^{\alpha(1+\gamma)}  (diam( S_{{l,k_\textsf{n}}}(\tilde{F})))^{\alpha\gamma}. \quad  \quad  
\end{eqnarray} }  

As in our box dimension consideration,   we have the limit
$$
\lim_{l\rightarrow \infty}   
\ \  (c\ {l}^{2\alpha_0} )^{\alpha(1+\gamma)}  (diam( S_{{l,k_\textsf{n}}}(\tilde{F})))^{\alpha\gamma}   =0 $$ uniformly in $\textsf{n}$, see the argument after 
(\ref{number_of_cubes2}).
Then  $(c\ {l}^{2\alpha_0} )^{\alpha(1+\gamma)}  (diam( S_{{l,k_\textsf{n}}}(\tilde{F})))^{\alpha\gamma}$ 
is uniformly bounded in $\textsf{n},l$.  Hence (\ref{HIn}) and  (\ref{limit1}) yield 
$$ \sum_l N_l^S(\tilde{F})    (diam( S_{{l,k_\textsf{n}}}(\tilde{F})))^{\alpha(1+2\gamma)}  \leq    \frac{\textsf{c}}{\textsf{n}}      $$ for some constant  $\textsf{c}$ so that
$$\lim_{\textsf{n}\rightarrow \infty}   \sum_{S_{{l,k_\textsf{n}}}(\tilde{F})\in \mathcal{C}_{k_\textsf{n}}(\tilde{F})} (diam(S_{{l,k_\textsf{n}}}(\tilde{F})))^{\alpha(1+2\gamma)}= \lim_{\textsf{n}\rightarrow \infty}   \sum_l N^S_l    (diam( S_{{l,k_\textsf{n}}}(\tilde{F})))^{\alpha(1+2\gamma)}  =0.$$

(ii)  Assume   that $\{F_{k_\textsf{n}}\}$ is a subsequence of  $ \{F_{k}\}$ such that
$\mathcal{T}^{\alpha(1+\gamma)}(F_{k_\textsf{n}})=\infty$ for $\textsf{n}\geq N_2$. Equations (\ref{piece_correspondence1}), (\ref{piece_correspondence2}) and (\ref{piece_correspondence22}) show how to pass 
from a piece $P_{l,k}(\tilde{F})$ to the set $S_l(F_{k_\textsf{n}})$ and then to a piece $P_{l}(F_{k_\textsf{n}})$, and vice versa.
We first work with the sets $P_{l}(F_{k_\textsf{n}})$, 
$S_l(F_{k_\textsf{n}})$ and show that $\mathcal{T}^{\alpha(1+\gamma)}(F_{k_\textsf{n}})=\infty \ \ (\textsf{n}\geq N_2)$ implies 
$\lim_{\textsf{n}\rightarrow \infty} \mathcal{S}^{\alpha(1+\gamma)}(F_{k_\textsf{n}})=\infty.$ After that, we will prove that 
$\lim_{\textsf{n}\rightarrow \infty} \mathcal{S}^{\alpha(1+\gamma)}(F_{k_\textsf{n}})=\infty \Longrightarrow \mathcal{P}^{\alpha}(\tilde{F})=\infty$, which gives the claimed equality.

By Lemma \ref{piece_measure}, covering $F_{k_\textsf{n}}$ by cubes $\textsf{C}_{\textbf{i}}(F_{k_\textsf{n}})$ or the sets $P_{l}(F_{k_\textsf{n}})$
makes no difference, so we may replace cubes by pieces.  More explicitly,  $\mathcal{T}^{\alpha(1+\gamma)}(F_{k_\textsf{n}})=\infty$ gives rise to 
 $\mathcal{P}^{\alpha(1+\gamma)}(F_{k_\textsf{n}})=\infty$ by Lemma \ref{piece_measure}-(ii). Thus for  $\textsf{n}\geq N_2$, 
 \begin{eqnarray*}\label{crucial0}
\infty= \mathcal{T}^{\alpha(1+\gamma)}(F_{k_\textsf{n}}) &=& \mathcal{P}^{\alpha(1+\gamma)}(F_{k_\textsf{n}}) \\
&=& 
{\footnotesize \underset{\epsilon>0}{\sup} \inf \bigg \{ \sum
(diam( P_l(F_{k_\textsf{n}})))^{\alpha(1+\gamma)} : \  \{ P_l(F_{k_\textsf{n}})\}  \ \textrm{is a finite} \ \epsilon\textrm{-cover of } \ F_{k_\textsf{n}}  } \bigg \}
 \end{eqnarray*}

By Lemma \ref{diameter_dimension}, $N_l(F_{k_\textsf{n}})\leq N_l^S(F_{k_\textsf{n}}) \leq q^{ks^2}N_l(F_{k_\textsf{n}})$; moreover 
there exists an index $\textsf{n}_0>N_2$ and a positive constant  $c< 1$ such that $c  \cdot diam( P_l(F_{k_\textsf{n}})) \leq  diam(S_l(F_{k_\textsf{n}})) \leq diam( P_l(F_{k_\textsf{n}}))$ and each $ S_l(F_{k_\textsf{n}})
$ is
in a finite $\epsilon$-cover
 for every $\textsf{n}\geq \textsf{n}_0.$  
Also, we note that  $$\mathcal{P}^{\alpha(1+\gamma)}(F_{k_\textsf{n}}) = {\footnotesize \underset{\epsilon>0}{\sup} \inf \bigg \{ \sum
(c \cdot diam( P_l(F_{k_\textsf{n}})))^{\alpha(1+\gamma)} : \  \{P_l(F_{k_\textsf{n}})\}  \ \textrm{is a finite} \ \epsilon\textrm{-cover of } \ F_{k_\textsf{n}} } \bigg \} $$ for each $\textsf{n}\geq \textsf{n}_0.$ 
This implies
{\footnotesize \begin{eqnarray*}\label{crucial}
 \infty=\lim_{\textsf{n}\rightarrow \infty} \mathcal{T}^{\alpha(1+\gamma)}(F_{k_\textsf{n}})  & = &  \lim_{\textsf{n}\rightarrow \infty}  \underset{\epsilon>0}{\sup} \inf \bigg \{ \sum
(c \cdot diam( P_l(F_{k_\textsf{n}})))^{\alpha(1+\gamma)} : \  \{P_l(F_{k_\textsf{n}})\}  \ \textrm{is a finite} \ \epsilon\textrm{-cover of } \ F_{k_\textsf{n}}  \bigg \} \nonumber \\
%
 & \leq & \lim_{\textsf{n}\rightarrow \infty}  \underset{\epsilon>0}{\sup} \ \inf \bigg \{ \sum
 ( diam(S_l(F_{k_\textsf{n}})))^{\alpha(1+\gamma)} :  \{S_l(F_{k_\textsf{n}}) \}
 \ \textrm{is a finite} \ \epsilon\textrm{-cover of }  \ F_{k_\textsf{n}}   \bigg \}  \nonumber \\
 & = & \lim_{\textsf{n}\rightarrow \infty} \mathcal{S}^{\alpha(1+\gamma)}(F_{k_\textsf{n}}).
 \end{eqnarray*} }
Then $\lim_{\textsf{n}\rightarrow \infty} \mathcal{S}^{\alpha(1+\gamma)}(F_{k_\textsf{n}})=\infty.$ 

\medskip

Now, let  $m_1=\lceil|\lambda_1|^{k_\textsf{n}}\rceil$ \  as \ in \ Definition \ \ref{piece}. Below we  borrow an inequality from Section \ref{Lower_Bounds_Dimension}. 

\bigskip
(1)  $diam(S_l(F_{k_\textsf{n}}))\leq diam(P_l(F_{k_\textsf{n}})) $ \quad \quad \quad   by \ Lemma \ \ref{diameter_dimension},

(2) $diam(P_l(F_{k_\textsf{n}})) \leq c'' l^{\alpha_0} (m_1)^{-l}  $ \quad \quad  \quad \quad  by \ the \ sentence \ before \ Lemma \ \ref{inequality_pieces}, 

(3) $diam(S_l(F_{k_\textsf{n}})) \leq  c_l diam(P_{l,k}(\tilde{F}))  \quad \quad  \ \  (c_l=c' \ l^{2\alpha_0})$ \  by  \ Lemma \ \ref{diameter1}.

\bigskip
\noindent Then
\footnotesize{ \begin{eqnarray*}
       ( diam(S_l(F_{k_\textsf{n}})))^{\alpha(1+\gamma)}  &=&   ( diam(S_l(F_{k_\textsf{n}})))^{\alpha\gamma} \ ( diam(S_l(F_{k_\textsf{n}})))^{\alpha}  \\
       & \leq & ( diam(P_l(F_{k_\textsf{n}})))^{\alpha\gamma} \ ( diam(S_l(F_{k_\textsf{n}})))^{\alpha} \quad \quad \quad \mathrm{ by \ (1)}  \nonumber \\
        & \leq &  (c'' l^{\alpha_0} (m_1)^{-l})^{\alpha\gamma} \ ( diam(S_l(F_{k_\textsf{n}})))^{\alpha} \quad \quad \quad \quad \mathrm{ by \ (2)}  \nonumber \\ 
         & \leq &  (c'' l^{\alpha_0} (m_1)^{-l})^{\alpha\gamma} \ ( c_l \ diam(P_{l,k}(\tilde{F})))^{\alpha}  \quad \quad \quad \mathrm{  by \ (3) } \nonumber \\
        & \leq &  (c'' l^{\alpha_0} (m_1)^{-l})^{\alpha\gamma} \ c_l^{\alpha} \ (  diam(P_{l,k}(\tilde{F})))^{\alpha}  \quad  \nonumber \\
         & \leq &  c \ (  diam(P_{l,k}(\tilde{F})))^{\alpha} \quad \quad \quad  \mathrm{because} \ (c'' l^{\alpha_0} (m_1)^{-l})^{\alpha\gamma} c_l^{\alpha}\leq  c  \ \mathrm{for \ some \ constant \ c.} \nonumber \\
    \end{eqnarray*} }
\normalsize     
Also  $N_l^S(F_k) \leq N^r_l(\tilde{F}) $ by Lemma \ref{counting_number_inequality_3}. Then  
$$ \sum N_l^S(F_k) ( diam(S_l(F_{k_\textsf{n}})))^{\alpha(1+\gamma)} \leq \sum N^r_l(\tilde{F}) c  (  diam(P_{l,k}(\tilde{F})))^{\alpha} \Longrightarrow \mathcal{S}^{\alpha(1+\gamma)}(F_{k_\textsf{n}})\leq \mathcal{P}^{\alpha}(\tilde{F}).$$   
That gives  that  $\infty=\mathcal{P}^{\alpha}(\tilde{F})=\lim_{\textsf{n}\rightarrow \infty} \mathcal{T}^{\alpha}(F_{k_\textsf{n}})$ because 
$\lim_{\textsf{n}\rightarrow \infty} \mathcal{S}^{\alpha(1+\gamma)}(F_{k_\textsf{n}})=\infty$,   and  
$$\mathcal{T}^{\alpha(1+\gamma)}(F_{k_\textsf{n}})=\infty \Longrightarrow \mathcal{T}^{\alpha}(F_{k_\textsf{n}})=\infty$$  by Lemma \ref{piece_measure}-(ii) or by (\ref{tiledimension}).
\end{pf}

\bigskip

\noindent  \textbf{ Proof of Theorem  \ref{Hausdorff}:  First Part}

\medskip

\textit{Proof for   Integer Diagonal Blocks}:

\medskip

If all diagonal blocks of  $J$ are integer matrices, then
$dim_H F=\delta_k
$ for each $k$, see Remark \ref{trivial_cases}. In this case, the stated convergence clearly holds.

\bigskip

\textit{Proof for Non-Integer Diagonal Blocks}:

\medskip

  Next we assume  $J$ has some non-integer diagonal blocks. 
  We just need to prove that $dim_{H}(F)=\lim_{k\rightarrow \infty} \delta_k
  $. Let $\delta $ be any accumulation point of the sequence $\{\delta_{k} \}_{k=1}^{\infty}$
 Thus it is enough to show that
 $
 \delta =dim_{H} F $.  
  For the proof, we assume that $\delta \not = dim_{H} F$
 and get a contradiction. Note that $dim_{H}(F)=dim_{H}(\tilde{F})$. Thus, we  have two cases: 

\textit{Case 1.} Assume that $\delta < dim_{H}(\tilde{F}).$ Since $\delta$ is an accumulation point of $\{\delta_{k} \}_{k=1}^{\infty}$, there exists a
 subsequence $\{\delta_{k_\textsf{n}} \}_{\textsf{n}=1}^{\infty}$ of $\{\delta_{k} \}_{k=1}^{\infty}$ converging to $\delta$. 
Fix a number $\alpha$  in the interval $(\delta, dim_{H}(\tilde{F})).$ Choose  another number $\gamma\in(0,1)$ so that $\alpha(1+\gamma)\in (\delta, dim_{H}(\tilde{F}))$ again.
%
Then there exists an index $N_1$ such that
 $$\delta_{k_\textsf{n}} < \alpha<\alpha(1+\gamma)< dim_{H}(\tilde{F})$$ for $\textsf{n}\geq N_1$.  Let $\mathcal{T}^{\alpha}(\tilde{F})$ be the $\alpha$-dimensional tile measure in  (\ref{tilemeasure2}).
Thus $ \mathcal{T}^{\alpha}(F_{k_\textsf{n}})=0$,
$ H^{\alpha(1+\gamma)}(\tilde{F})=\infty$   and $ \mathcal{T}^{\alpha(1+\gamma)}(F_{k_\textsf{n}})=0$ ($\textsf{n}\geq N_1$).
Then  Lemma \ref{measure_convergence}-(i) leads to
the contradiction $$H^{\alpha(1+\gamma)}(\tilde{F})=\lim_{\textsf{n}\rightarrow \infty} \mathcal{T}^{\alpha(1+\gamma)}(F_{k_\textsf{n}})=0  \ \ !!!$$  Thus we must have
$dim_{H}(\tilde{F})\leq \delta .$ 

\textit{Case 2}. 
Assume now that $dim_{H} \tilde{F}< \delta.$ Let $\{\delta_{k_\textsf{n}} \}_{\textsf{n}=1}^{\infty}$  be a subsequence of $\{\delta_k \}_{k=1}^{\infty}$ converging to $\delta$.
This time, choose  $\alpha$ between $dim_{H}(\tilde{F})$ and $\delta $. Fix  another number $\gamma\in(0,1)$ so that 
$\alpha(1-\gamma), \ \alpha(1+\gamma)$ are in the  interval $(dim_{H}(\tilde{F}), \delta)$. Then there exists an index $N_2$ such that 
$$dim_{H} (\tilde{F}) <\alpha(1-\gamma) < \alpha  <\alpha(1+\gamma)< \delta_{k_\textsf{n}} $$ for all $\textsf{n}\geq N_2$.  Then
$ \mathcal{T}^{\alpha(1+\gamma)}(F_{k_\textsf{n}})=\infty$ for $\textsf{n}\geq N_2$. 
Note that 
$$\alpha=\alpha(1-\gamma)+\alpha(1-\gamma)\frac{\gamma}{1-\gamma}$$ and we already have $ \mathcal{T}^{\alpha(1-\gamma)}(\tilde{F})=0$. Therefore, replacing $\alpha$ by $\alpha(1-\gamma)$ and 
$\gamma$ by $\frac{\gamma}{1-\gamma}$ in Lemma \ref{piece_measure}-(i), we get
$ \mathcal{P}^{\alpha}(\tilde{F})=0$.
Then Lemma \ref{measure_convergence}-(ii) leads to
the contradiction $\mathcal{P}^{\alpha}(\tilde{F})=\lim_{\textsf{n}\rightarrow \infty} \mathcal{T}^{\alpha}(F_{k_\textsf{n}})=\infty$ since $ \mathcal{T}^{\alpha(1+\gamma)}(F_{k_\textsf{n}})=\infty$ for  $\textsf{n}\geq N_2$.
Thus we must have
$\delta \leq  dim_{H}(\tilde{F}).$ We mention that there is another proof of the last inequality that follows from  Proposition \ref{monotonicity2}.
$\hfill\blacksquare$

\section{Lower Bounds for Dimension}\label{Lower_Bounds_Dimension}

In the foregoing sections, we have not used any  lower bounds for the dimension of  $F$ in the proofs. In this section, we will give lower bounds to be used in the next section.
We  use the sequence notation $x=(\tilde{a}_{\mathbf{j}_i})\in F(J_k,\tilde{A}_k)=\tilde{F}$, \ $x_k=(d_{{\mathbf{j}_i}})\in F(T_k,D_{k})=F_k$ below.
 Corresponding to the pieces $P_{l,k}(\tilde{F})$ of level $l$, let $S_l(F_k)$ be as in (\ref{sets_S}).

\bigskip

\begin{lemma}\label{diameter1}  Let $F=F(T,A)\subset \mathbb{R}^n$ be an integral self-affine set. As usual, we assume that $A\neq \{0\}.$ 
Let $\tilde{F}=F(J_k,\tilde{A}_k)=P^{-1}F$ and $F_k=F(T_k,D_k)$ be a
 lower perturbation of $\tilde{F}$. 
Then there exists a positive integer 
$k_0$ and a constant \ $c_l=c_l(k)$  such that $c_l$ is in the form of $c_l=c \ l^{2\alpha_0}$ and 
$$diam(S_l(F_k)) \leq  c_l\cdot diam(P_{l,k}(\tilde{F})),$$
for $ k\geq k_0$.

\end{lemma}

\begin{pf} By  Remark \ref{constant_remark}, \ Lemma \ \ref{inequality_pieces} and the  sentence  before  it, 
$$c' l^{-\alpha_0} (r_1)^{-l} \leq diam(P_{l,k}(\tilde{F})),  
 \quad \quad 
diam(S_l(F_k)) \leq c'' l^{\alpha_0} (m_1)^{-l} $$ 
for some constants $c', c''>0
$ depending on $k\geq k_0$. 
Taking into account of the inequality $m_1 
\geq r_1$ by definition, see (\ref{pert_matrix}), we  let
$$c_l :=\frac{c'' l^{2\alpha_0} }{c' }\geq 
\frac{c'' l^{\alpha_0} (m_1)^{-l}}{c' l^{-\alpha_0} (r_1)^{-l}}.$$ 
That gives $diam(S_l(F_k)) \leq  c_l\cdot diam(P_{l,k}(\tilde{F})).$ 
\end{pf}

\bigskip

\bigskip

Because of the following results, we call $F_k$ a lower perturbation.

\begin{prop}\label{monotonicity1} Let $F=F(T,A)\subset \Bbb{R}^n$ be an integral self-affine set and $\tilde{F}=F(J_k,\tilde{A}_k)=P^{-1}F$. Then there exists a positive integer $k_0$ such that  $$ dim_B F_k \leq \underline{dim}_B \tilde{F} \leq  \overline{dim}_B \tilde{F}$$  for $k\geq k_0$.
\end{prop}

\begin{pf} Clearly, it suffices to prove the inequality $ dim_B F_k \leq \underline{dim}_B \tilde{F} $.
We 
consider the following counting numbers. For each positive integer $r$,
\begin{itemize}
 \item $N^s_r(\tilde{F})$ is the smallest number  of sets (not necessarily mesh cubes) of diameter at most $cc_0m^{-r}$ covering $\tilde{F}$ ($c>0$ is any fixed constant).
  \item   $N_l^S(F_k)$ denotes the number of the sets $S_{l}(F_k)$ in (\ref{sets_S}) covering  $F_k$.
\end{itemize}
 $F_k$ and $\tilde{F}$ may be interchanged. As before, let $N^r_l(\tilde{F})$ is the number all pieces 
 $P_{l,k}(\tilde{F})$ of $\tilde{F}$, which intersect and are commensurable with an n-cube $\textsf{C}_{\mathbf{i}}$ of side length $c_0m^{-r}$ (see Section \ref{An application}).

By Lemma \ref{diameter1},  we also know that $$ diam(S_l(F_k)) \leq c_l \cdot  diam(P_{l,k}(\tilde{F}))\approx \frac{c_l c_0}{m^{r}}$$ for  $k\geq k_0$,  
where $c_l=c \ l^{2\alpha_0}$ for some constant $c$. Since $c_0$ doesn't affect our computation, we will ignore it.
Then $$\underset{r\rightarrow\infty}{\lim}\frac{ c_l}{ m^{r}}=0$$ like the limit computation in the proof of Lemma \ref{box_piece2}. That is, $\underset{r\rightarrow\infty}{\lim}   diam(S_l(F_k))=0,$ which is needed in the definition of box dimension.  Thus, for $F_k$, we can also say $N^s_r(F_k)$ is the smallest number  of sets (not necessarily mesh cubes) of diameter at most 
$c_lm^{-r}$
covering $F_k$. 


Now take any number $\gamma\in (0,1)$. By (\ref{number_of_cubes2}), we have $c''' l^{-\alpha_0}(\lambda^{k})^{l}\leq m^{r}\leq c'''' \ l^{\alpha_0}(\lambda^{k})^{l}$. So 
$ l\rightarrow\infty \ \ \Longleftrightarrow\ \ r\rightarrow\infty$.
Because of the limit $\underset{l\rightarrow\infty}{\lim}\frac{(c''' \left(\lambda^{k})^{l}\right)^{\gamma}}{c \ l^{\alpha_0 \left(2+\gamma \right)}}=\infty,$ it follows that $$1<\frac{(c''' \left(\lambda^{k})^{l}\right)^{\gamma}}{c \ l^{\alpha_0(2+\gamma)}}=\frac{\left(c''' l^{-\alpha_0}(\lambda^{k})^{l}\right)^{\gamma}}{c \ l^{2\alpha_0}}=\frac{\left(c''' l^{-\alpha_0}(\lambda^{k})^{l}\right)^{\gamma}}{c_l}\leq \frac{(m^{r})^{\gamma}}{c_l}$$ 
for all large $l$ or $r$. 
Then $(m^{r})^{1-\gamma}<\frac{m^{r}}{c_l}$ and hence,
\begin{eqnarray*}\label{box_counting}
\frac{1}{1-\gamma} \underline{dim}_B \tilde{F} &=& \underset{r\rightarrow\infty}{\underline{\lim}}  \frac{ \log N^r_l(\tilde{F})}{ \log (m^{r})^{1-\gamma}} \geq \underset{r\rightarrow\infty}{\underline{\lim}}  \frac{ \log N_l^S(F_k)}{ \log (m^{r})^{1-\gamma}} \quad \mathrm{by \ Lemma }\ \ref{counting_number_inequality_3} \\
&\geq &  \underset{r\rightarrow\infty}{\underline{\lim}}  \frac{ \log N_l^S(F_k)}{ \log \left(\frac{m^{r}}{c_l}\right)}  \geq \underset{r\rightarrow\infty}{\underline{\lim}}  \frac{ \log N_l^s(F_k)}{ \log \left(\frac{m^{r}}{c_l}\right)}=\underline{dim}_B  F_k, \\
\end{eqnarray*}
Thus $ \underline{dim}_B \tilde{F} \geq (1-\gamma)\underline{dim}_B  F_k$. Letting $\gamma \rightarrow 0$, we obtain $\underline{dim}_B \tilde{F} \geq \underline{dim}_B  F_k.$
For the perturbed fractals $F_k$, we know that the box dimension exists by Proposition \ref{Existence} in the Appendix.
Then $\underline{dim}_B F_k=\overline{dim}_B F_k=dim_B F_k$, and thus
$$\underline{dim}_B \tilde{F}\geq dim_B F_k$$  for $k\geq k_0$.
\end{pf}

\bigskip


\noindent \textbf{Proof of Theorem \ref{Main_Box}}: 
By Proposition \ref{monotonicity0}, we have 
$\underline{dim}_B \tilde{F}\leq   \overline{dim}_B \tilde{F} \leq \underset{k\rightarrow\infty}{\underline{\lim}}   {dim}_B F_k \leq \underset{k\rightarrow\infty}{\overline{\lim}}   {dim}_B F_k.$ In addition to this, by Proposition \ref{monotonicity1}, $ dim_B F_k \leq \underline{dim}_B \tilde{F} \leq  \overline{dim}_B \tilde{F}$  for $k\geq k_0$. That gives $\underset{k\rightarrow\infty}{\overline{\lim}}  dim_B F_k\leq \underline{dim}_B \tilde{F}.$  Combining these inequalities 
$$\underline{dim}_B \tilde{F}\leq   \overline{dim}_B \tilde{F} \leq \underset{k\rightarrow\infty}{\underline{\lim}}   {dim}_B F_k \leq\underset{k\rightarrow\infty}{\overline{\lim}}  dim_B F_k\leq \underline{dim}_B \tilde{F},$$ i.e. the box dimension of $\tilde{F}$ (or $F$) exists and  is equal to
$ \underset{k\rightarrow\infty}{\lim}  dim_B F_k$.  $\hfill \blacksquare$

\bigskip

We will also use the following trivial lemma.

\begin{lemma}\label{diameter3} Let $P_{l,k}(\tilde{F}),S_l(F_k)$ be the sets  in Lemma \ref{diameter1}.
Assume that $P_{l,k}(\tilde{F})$ is commensurable with
an n-cube $\textsf{C}$  of side length $c_0m^{-r}$,
then $S_l(F_k)$ can be covered by $3^{n+1}$-many  n-cubes with side length
$c_l c_0m^{-r}$ for large $k$. 
\end{lemma}

\begin{pf} 
An $n$-cube with side length
$c_lc_0m^{-r}$ will intersect $S_l(F_k)$. Then that cube and its neighbors with the same side length will cover $S_l(F_k)$ by Lemma \ref{diameter1}.
Therefore, $3^{n+1}$-many $n$-cubes with side length
$c_lc_0m^{-r}$
 will cover $S_l(F_k)$.
\end{pf}

\bigskip

\noindent \textbf{Proof of Theorem \ref{Hausdorff}: Second Part} 

\begin{prop}\label{monotonicity2} Let $F=F(T,A)\subset \Bbb{R}^n$ be an integral self-affine set. 
Let   $\delta_k$ be the
perturbation dimensions
given by (\ref{perturbation_dimensions}). 
Then there exists a positive integer
$
k_0$ such that $$\delta_k=dim_H F_k  \leq dim_H F =dim_H \tilde{F}
$$ for each $k\geq k_0$.

\end{prop}

\begin{pf} 
Assume that $F_k$ is a
perturbation for $\tilde{F}$ as in Lemma \ref{diameter1}.
We would like to show that for large $k$,  $ \mathcal{T}^{\alpha}(\tilde{F})=0$ implies $ \mathcal{T}^{\alpha(1+\gamma)}(F_k)=0$ for each $\gamma>0$.   
This implies $dim_H(F_k)  \leq \alpha(1+\gamma).$ 

 Let $\lambda ( \geq 2)$ be the minimum of the absolute values of the eigenvalues of $J$ and set $m=\lfloor \lambda \rfloor$. Take $m I$ ($I$ is the identity matrix) as the
generating matrix of the tile $\Gamma$ containing $\tilde{F},F_k$
in  the proof of
Lemma \ref{graph-tile2}.
%
%
As before,  $P_{l,k}(\tilde{F})$
denotes the pieces of  $\tilde{F}$
of level $l$, see Definition \ref{piece}. Let $S_l(F_k)$ be the set in (\ref{sets_S}). 
\begin{itemize}
  \item We will start with a cover of $\tilde{F}$ by cubes and use the sets $P_{l,k}(\tilde{F}), \ S_l(F_k)$ to find a cover of $F_k$ by certain cubes with the desired properties. 
\end{itemize} 
We usually cover $\tilde{F},F_k$ by the subsets  of $\Gamma$, which are $n$-cubes of side lengths $c_0m^{-r}$
as in Section \ref{An application}. But, for technical reasons, we will cover $\tilde{F}$ by  $n$-cubes $\textsf{C}_{\mathbf{i}}$ of side lengths $c_0m^{-2r}$ instead of   $c_0m^{-r}$. By Lemma \ref{bound},   each cube $\textsf{C}_{\mathbf{i}}$ contains at most   $U_r = \alpha_1 l^{\alpha_2}$  pieces $P_{l,k}(\tilde{F})$  commensurable with and intersecting $\textsf{C}_{\mathbf{i}}$. Here 
$\alpha_1 , \ \alpha_2$ are positive integers.

We have $\frac{U_r}{(m^{r})^{\alpha\gamma}} = \frac{\alpha_1 l^{\alpha_2}}{(m^{r})^{\alpha\gamma}}$. Since we use covers of $\tilde{F}$ by cubes $\textsf{C}_{\mathbf{i}}$ of side lengths $c_0m^{-2r}$ here, replacing $m^{r}$ by $m^{2r}$ 
in (\ref{number_of_cubes2}), we get 
$$c''' l^{-\alpha_0}(\lambda^{k})^{l}\leq m^{2r}\leq c'''' \ l^{\alpha_0}(\lambda^{k})^{l}\Longrightarrow  \frac{U_r}{(m^{r})^{\alpha\gamma}}\leq \frac{\alpha_1 l^{\alpha_2+\alpha_0\frac{\alpha\gamma}{2}}}{(c''')^{\frac{\alpha\gamma}{2}} (\lambda^{k})^{l\frac{\alpha\gamma}{2}}}$$ for large 
$l.$ Here $c''', c''''$ are constants. Note that $r\rightarrow \infty \Longrightarrow l\rightarrow \infty.$ Then 
$$\underset{r\rightarrow \infty}{\lim} \frac{U_r}{(m^{r})^{\alpha\gamma}} \leq \underset{l\rightarrow \infty}{\lim} \frac{\alpha_1 l^{\alpha_2+\alpha_0\frac{\alpha\gamma}{2}}}{(c''')^{\frac{\alpha\gamma}{2}} (\lambda^{k})^{l\frac{\alpha\gamma}{2}}}=0\Longrightarrow \underset{r\rightarrow \infty}{\lim} \frac{U_r}{(m^{r})^{\alpha\gamma}}=0.$$
Therefore,  $\frac{U_r}{(m^{r})^{\alpha\gamma}}$ is bounded by a positive integer $c({\gamma})$. Also,  by Lemma \ref{diameter1}, we have, for large $k\geq k_0$,
\begin{equation}\label{diam0}
 diam(S_l(F_k)) \leq c_l \cdot  diam(P_{l,k}(\tilde{F}))
\end{equation}
where $c_l$ can be taken in the form of $c_l=c \ l^{2\alpha_0}$. Similarly, 
$$ \underset{r\rightarrow \infty}{\lim} \frac{c_l}{(m^{r})^{\frac{\gamma}{1+\gamma}}}=0$$  
and hence, $\frac{c_l}{(m^{r})^{\frac{\gamma}{1+\gamma}}}\leq c$ for some constant $c$ (independent of $l$ or $r$). Thus we have obtained two inequalities: 
\begin{equation}\label{two_inequalities}
\frac{U_r}{(m^{r})^{\alpha\gamma}} \leq c({\gamma}), \quad \quad  \quad  \quad  \quad  \frac{c_l}{(m^{r})^{\frac{\gamma}{1+\gamma}}}\leq c
\end{equation}


 Now we use the constants $c({\gamma})$ and  $c$ obtained from the above limits. Suppose that $ \mathcal{T}^{\alpha}(\tilde{F})=0$. Then for every $\epsilon>0$, there is a cover $\mathcal{C} $ of $\tilde{F}$ by  $n$-cubes $\textsf{C}_{\mathbf{i}}$ of side lengths $c_0m^{-2r}$ ($r$ may vary) such that
\begin{equation}\label{epsilon_inequality}  \sum_{\textsf{C}_{\mathbf{i}}\in \mathcal{C}}(diam(\textsf{C}_{\mathbf{i}}))^{\alpha}<\frac{\epsilon}{3^{n+1} c({\gamma}) c^{\alpha(1+\gamma)}  (\sqrt{n}c_0)^{\alpha \gamma} }.
\end{equation}

 We have at most $U_r$ pieces $P_{l,k}(\tilde{F})$ intersecting and commensurable with $\textsf{C}_{\mathbf{i}}$. 
For such a piece $P_{l,k}(\tilde{F})$, we  consider the corresponding set $S_l(F_k)$. We may assume that $$diam(P_{l,k}(\tilde{F}))= diam(\textsf{C}_{\mathbf{i}})=\sqrt{n}c_0m^{-2r}.$$  
Then
by (\ref{diam0}) and Lemma \ref{diameter3} , for each cube $\textsf{C}_{\mathbf{i}}\in \mathcal{C}$, a total of at most
$3^{n+1}U_r$-many n-cubes $\textsf{C}_{\mathbf{i}}'$ of side lengths $c_lc_0m^{-2r}$ will cover all the sets $S_l(F_k)$ 
 (corresponding to the pieces $P_{l,k}(\tilde{F})$ that intersect and are commensurable with that particular  $\textsf{C}_{\mathbf{i}}$).  
  Let $\mathcal{C'}$ be the set of such $\textsf{C}_{\mathbf{i}}'.$ But the union of $\textsf{C}_{\mathbf{i}}\in\mathcal{C} $ contains $\tilde{F}$ and hence so does  the union of $P_{l,k}(\tilde{F})$. In view of (\ref{sets_S}),  the union of the sets $S_l(F_k)$ is $F_k$. Then the set $\mathcal{C'}$ consisting of  $\textsf{C}_{\mathbf{i}}'$ is  a cover  of $F_k$.
Thus for large $k$, we have 
\begin{eqnarray*}
 \sum_{\textsf{C}_{\mathbf{i}}'\in \mathcal{C'}} (diam(\textsf{C}_{\mathbf{i}}'))^{\alpha(1+\gamma)} & \leq  &  \sum_{\textsf{C}_{\mathbf{i}}'\in \mathcal{C'}} (\sqrt{n}c_lc_0m^{-2r})^{\alpha(1+\gamma)} \ \ \ \mathrm{by} \  (\ref{diam0})  \\
   & = & \sum_{\textsf{C}_{\mathbf{i}}'\in \mathcal{C'}} c_l^{\alpha(1+\gamma)}(m^{-2r})^{\alpha \gamma}  (\sqrt{n}c_0)^{\alpha \gamma}  (\sqrt{n}c_0m^{-2r})^{\alpha} \\
  & \leq &  \sum_{\textsf{C}\in \mathcal{C}} 3^{n+1}\frac{U_r}{(m^{r})^{\alpha\gamma}} \frac{c_l^{\alpha(1+\gamma)}}{(m^{r})^{\alpha\gamma}} (\sqrt{n}c_0)^{\alpha \gamma}  (\sqrt{n}c_0m^{-2r})^{\alpha}  \\
   &\leq & \sum_{\textsf{C}\in \mathcal{C}} 3^{n+1}c({\gamma}) c^{\alpha(1+\gamma)}  (\sqrt{n}c_0)^{\alpha \gamma}  (\sqrt{n}c_0m^{-2r})^{\alpha} \quad \quad \quad \ \ \ \mathrm{by} \  (\ref{two_inequalities}) \\
   & = & \sum_{\textsf{C}\in \mathcal{C}} 3^{n+1}c({\gamma}) c^{\alpha(1+\gamma)} (\sqrt{n}c_0)^{\alpha \gamma}  (diam(\textsf{C}_{\mathbf{i}}))^{\alpha}<\epsilon \quad \quad \ \ \ \mathrm{by} \  (\ref{epsilon_inequality}). \\
\end{eqnarray*}

Hence $\mathcal{T}^{\alpha(1+\gamma)}(F_k)=0$ since $\epsilon$ is arbitrary. Then $ \delta_k \leq \alpha(1+\gamma)$. Letting $\gamma\rightarrow 0$, we get 
$ \delta_k \leq \alpha$ when $\alpha > dim_H F$. Letting in turn $\alpha\rightarrow dim_H F$, we finally obtain $ \delta_k \leq dim_H F$ for  $k\geq k_0$.
\end{pf}

\begin{remark}
{\rm If the Jordan normal form $J$ has simple eigenvalues (i.e.  $J$ is diagonal), then $\alpha_2=0\Rightarrow  U_r = \alpha_1 $ by Lemma \ref{bound}.  Also $c_l=c$. Therefore, 
the proofs in the paper become simpler.
 $\hfill \Box$
 }
\end{remark}

At this point, we mention a mistake in the statement of the main theorem  in \cite{K1}. Namely, the inequality $\delta_{2k}\leq dim_H F \leq \upsilon_{2k}$ holds essentially, meaning for every large enough $k$.

\section{Another Application: \ Measures of Full Dimension}\label{Full_dimension}

Let $F=F(T,A)$ be our original integral self-affine set and $J$ be the Jordan normal form of $T$ as described in Section  \ref{INTRO}. 
For convenience, denote
$F(J_k,\tilde{A_k})\rm  \ {mod  \ 1}$ by $\tilde{F}_0$ and the sequence $\{ F(T_k,D_k) \rm  \ {mod \ 1} \}$  by $\{\tilde{F}_{k}\}$. 
Thus  
$$\tilde{F}_0=F(J,\tilde{A}) {\rm  \ \ {mod  \ 1}},   \quad \quad \quad  
\tilde{F}_k=F(T_k,D_k) {\rm  \ \ {mod \ 1}} \quad \quad \quad k=1,2,...$$ 
and hence $\tilde{F}_0, \ \tilde{F}_k\subseteq \mathbb{T}^n$ are different sets from $\tilde{F}, F_k$ in the previous sections. 
Since $\tilde{F}$ or $F_k$  can be written as a finite union $\bigcup (\mathbf{v}_\textsf{i} + S_\textsf{i})$, where $\mathbf{v}_\textsf{i}\in \mathbb{Z}^n$ and $S_\textsf{i}\subseteq  \mathbb{T}^n $ (the $n$-torus), we have
$\tilde{F}_k=\bigcup S_\textsf{i}$ and $dim_H \tilde{F}_k=dim_H F_k$ (by translation-invariance and finite stability of the Hausdorff dimension). Therefore, we may work with 
$\tilde{F}, F_k$ in place of $\tilde{F}_0, \ \tilde{F}_k$ in our context. 
More generally,  we may replace  $\tilde{F}_0, \ \tilde{F}_k$ by compact subsets 
$
K\subseteq \tilde{F}_0$ and $\bigcup  S_\textsf{i}=K_k\subseteq \tilde{F}_k$ such that the points in $\bigcup (\mathbf{v}_\textsf{i} + S_\textsf{i})\subseteq\tilde{F}, F_k$ corresponding to these sets  have  the same symbolic codings. 
In this way, 
any point of $\mathbf{v}_\textsf{i }+ S_\textsf{i}$ has index codings in $\tilde{F}$ or  $F_k$. That allows the dimension results of the previous sections to continue to hold for 
$K, K_k$ as well as $\tilde{F}, F_k$ (while using the same type of covering sets of $K
, \  K_{k_{\textsf{n}}}$ in the proofs, $n$-cubes and pieces of $\tilde{F}, F_k$). This observation will be used in the proof of Lemma \ref{full_dimension2}.

Consider the $n$-torus
$\mathbb{T}^n=\mathbb{R}^n/\mathbb{Z}^n$.
We  mention that  $\mathbb{T}^n$ is a compact complete metric space with respect to the natural torus metric $$d_{\mathbb{T}^n}(x,y)=\inf\{ |u-v|: x,y\in \mathbb{T}^n, \ \ u\in x+\mathbb{Z}^n, \ \ v\in y+\mathbb{Z}^n  \}$$ derived from the Euclidean metric and the natural torus metric gives the quotient topology on $\mathbb{T}^n$ \cite{wwww}, \cite{wwwww}, \cite{wwwwww}.  Clearly, $d_{\mathbb{T}^n}(x,y)\leq |x+\textsf{n}-(y+\textsf{m})|$ for $\textsf{n}, \textsf{m}\in \mathbb{Z}^n$. Then 
 $F_k\longrightarrow \tilde{F} \ \ \Longrightarrow \ \ \tilde{F}_k\longrightarrow \tilde{F}_0.$
 
Let $\textsf{P}(\mathbb{T}^n)$ be the set of all Borel probability measures on $\mathbb{T}^n$ and
$\textsf{B}(\mathbb{T}^n)$ be the $\sigma$- algebra of Borel sets in $\mathbb{T}^n$.
For our purposes, bring in the \textit{Prokhorov metric} on $\textsf{P}(\mathbb{T}^n)$ :
$$d_P(\mu, \nu):=\inf \left\{  \alpha >0 \ : \ \mu (E) \leq \nu (E_{\alpha})+\alpha \ \ {\rm and} \ \  \nu (E) \leq \mu (E_{\alpha}) + \alpha \ \ \  \forall E\in \textsf{B}(\mathbb{T}^n) \right\}$$ where  $E_{\alpha}:=\{ x: \ d_{\mathbb{T}^n}(x,y)< \alpha \ \ \ {\rm for \ some } \ y \in E \}$ is the open ${\alpha}$-neighborhood of $E$.
Being a compact metric space, $\mathbb{T}^n$ is a separable metric space. Then we have $$\mu_{k}\rightarrow \mu   \  {\rm (weakly) \ \ \ \ if \ and \ only \ if} \ \ \ \ d_P(\mu_{k},\mu)\rightarrow 0 $$ for
$\mu, \mu_{k} \in \textsf{P}(\mathbb{T}^n)$ (see \cite[(v), pp.72-73]{B}).

We also need the following
result of Kenyon and Peres  \cite[Theorem 3, p.117]{GP}, \cite[p.103]{PS}.

\begin{prop}\label{GP}
 Let $T$ be a (linear) expanding endomorphism of the $n$-dimensional torus, which has integer eigenvalues, or more generally eigenvalues with an integer power. Then any $T$-invariant subset of the torus supports a $T$-invariant ergodic measure of full Hausdorff dimension.
\end{prop}

An outline of the proof is in \cite[Theorem 3, p. 120]{GP}, but the proof there is not satisfactory. A detailed proof is in \cite{KP1}, where they don't assume
the angle-preserving property of conformal maps. Note that our perturbation matrix $T_k$ is a direct some of integer matrices as described in the proof of Theorem 1.1 in \cite{KP1} and it was mentioned there that the proof  holds as long as  $dim_H F_k$ or  $dim_H \tilde{F}_k$ can be calculated by using the cylinder sets of a Markov partition for $T_k$ (pieces or cylinder sets of an auxiliary tile in the context of our paper). There is a justification of this fact in \cite[Lemme 4.6]{KP1}.  There is another explanation at the end of Section \ref{Tile-Measures}, see also Lemma \ref{piece_measure}.  Since we will apply the above proposition to 
$\tilde{F_k}
$, the only moot point in their arguments might be the existence or the construction of a Markov partition of the torus for $T_k$ (see the proof of Lemma 4.6 in \cite{KP1}). But, that is
provided by the integral self-affine tiles generated by the diagonal blocks $T_{k,p}$ of $T_k$ and the digits sets $D_{k,p}$ given in Remark \ref{fixed modulus}.

\bigskip

On \ $F  {\rm  \  {mod  \ 1}}$, we would like to construct a $T$-invariant ergodic measure of full Hausdorff dimension  via $\tilde{F}_k$
according to the following steps.

\begin{enumerate}
  \item By  Proposition \ref{GP}, there is a sequence of $T_k$-invariant ergodic measures of full Hausdorff dimension on perturbed fractals $\tilde{F}_k$. A subsequence of this sequence weakly converges to a Borel probability measure $\mu$ on $\tilde{F}_0$.
  \item $\mu$ has  property (P): If $\mu(B)>0$ for a Borel set $B\subseteq \tilde{F}_0$, then $dim_H B=dim_H (\tilde{F}_0)=\delta$. 
This implies that $\mu$ has full dimension.
  \item  $\mu$ can be lifted to a measure $\bar{\mu}$ on $F \ {\rm mod \ 1} $ with the same property.
  \item Using $\bar{\mu}$, it can be shown that there is a $T$-invariant Borel probability measure $\mu^*$ on $F \ {\rm mod \ 1} $ retaining all the previous properties.
  \item Thus, the class $\mathcal{M}_{\delta}(F,T)$ of $T$-invariant Borel probability measures on $F  {\rm  \  {mod  \ 1}}$ with property (P) is non-empty. 
  \item $\mathcal{M}_{\delta}(F,T)$ is 
  compact and convex. Hence any extreme point of $\mathcal{M}_{\delta}(F,T)$ is the required ergodic measure.
\end{enumerate}

We divide the proof of Theorem  \ref{full_dimension} into several lemmas according to these steps. Let $spt(\mu)$ denote the support of a measure $\mu$.  
 Let $S^c$ denote the complement of a set $S$.  

\begin{lemma}\label{Complement_F0}  $\mathbb{T}^n\setminus \tilde{F}_0=\tilde{F}_0^c$  can be written as a countable union of closed (open) sets each of which is in $\tilde{F}_k^c$ for large $k$. 
The same is true  for $\mathbb{T}^n\setminus K=K^c$,  
 where $K$  is  a  compact subset of $\tilde{F}_0$.

\end{lemma}

\begin{pf}  It is enough to prove the first claim. Since $F^c$ is open,   it  can be written as a countable union of non-overlapping (i.e., interiors are disjoint) closed rectangular regions   in $\mathbb{R}^n$ \cite[Problem 9, p.35]{J}.
By the observations in the first paragraph of this section, $\tilde{F}_0^c$ can then be written as a countable union of non-overlapping  closed  sets (intersections of closed rectangles) since
$\tilde{F}^c$  is open. 

 Let $K_\textsf{m}$ be  such closed sets so that  $\tilde{F}_0^c=\underset{\textsf{m}\in \mathbb{N}}{\bigcup}  K_\textsf{m}$.  Also   $d_H(\tilde{F}_0,K_\textsf{m})>0$ because of  the   compactness of $\tilde{F}_0 $   and   $K_\textsf{m}\cap  \tilde{F}_0 = \emptyset $.   Then 
 there exits $\epsilon_\textsf{m}>0$ so that the open $\epsilon_\textsf{m}$-neighborhood $K_{\epsilon_\textsf{m}}$ of $K_\textsf{m}$ is still in $\tilde{F}_0^c$.  Thus $\tilde{F}_0^c=\bigcup K_\textsf{m}=\bigcup   K_{\epsilon_\textsf{m}}.$
 Being an  open ${\epsilon_\textsf{m}}$-neighborhood,  the set  $ K_{\epsilon_\textsf{m}}$ is open. 
 
 The fact that  
  $\tilde{F}_{k}\rightarrow \tilde{F}_0$ in the Hausdorff metric  implies that 
for each $\textsf{m}$, there exists    $\textsf{k}_\textsf{m}\in \Bbb{N}$   such that $d_H(\tilde{F}_0,\tilde{F}_k)<\frac{d_H(\tilde{F}_0,K_\textsf{m})}{2}$ for  $k\geq \textsf{k}_\textsf{m}$. Then $d_H(\tilde{F}_0,\tilde{F}_k)+d_H(\tilde{F}_k,K_\textsf{m}) \geq  d_H(\tilde{F}_0,K_\textsf{m})>0 \Longrightarrow d_H(\tilde{F}_k,K_\textsf{m}) \geq d_H(\tilde{F}_0,K_\textsf{m})-d_H(\tilde{F}_0,\tilde{F}_k) \geq \frac{d_H(\tilde{F}_0,K_\textsf{m})}{2}>0\Longrightarrow$ 
$ K_\textsf{m}\subseteq \tilde{F}^c_{k} $  for  $k\geq \textsf{k}_\textsf{m}$. Notice that for each $\textsf{m}$, the lower bound  $\frac{d_H(\tilde{F}_0,K_\textsf{m})}{2}$ in this inequality is independent of $k$. Then  the inequality $$ d_H(\tilde{F}_k,K_\textsf{m}) \geq \frac{d_H(\tilde{F}_0,K_\textsf{m})}{2}>0   \quad (k\geq \textsf{k}_\textsf{m})$$ 
also shows that for each $\textsf{m}$, we may choose  the value of $\epsilon_\textsf{m}$ so that $ K_\textsf{m}\subseteq K_{\epsilon_\textsf{m}} \subseteq\tilde{F}^c_{k} $  for  $k\geq \textsf{k}_\textsf{m}$. 
%
%
%
\end{pf}

\begin{lemma}\label{Borel_measure} There exists a sequence of Borel probability measures $\mu$, $\mu_{k_\textsf{n}}$ with the following properties:

(i) $\mu_{k_\textsf{n}}\rightarrow \mu$   (weakly) ~ and ~~   $spt(\mu_{k_\textsf{n}})=\tilde{F}_{k_\textsf{n}},$ \   $spt(\mu)\subseteq \tilde{F}_0.$ 

(ii) If $\mu(B)>0$ for a Borel set $B\subseteq \mathbb{T}^n$, then there is  a sequence compact sets $K\subset B, \ K_{k_{\textsf{n}}}\subset \tilde{F}_{k_{\textsf{n}}}$ such that $\mu(K), \ \mu_{k_{\textsf{n}}}(K_{k_{\textsf{n}}})>0$ for all large enough $\textsf{n}$. Further, 
$K_{k_\textsf{n}}
 \underset{\textsf{n} \rightarrow \infty}{\longrightarrow}  K \cap \tilde{F}_0  
$
in the Hausdorff metric induced by $d_{\mathbb{T}^n}$.
\end{lemma}

\begin{pf} (i) By Proposition \ref{GP},  each perturbed fractal ${\tilde{F}_k} \subset \mathbb{T}^n$  supports an ergodic $T_k$-invariant Borel probability measure $\mu_k$ of full dimension.
 Then there exists a subsequence $\{\mu_{k_\textsf{n}}\}$ of $\{\mu_k\}$  which converges weakly to a measure $\mu$  \cite[Proposition 1.9, p.14]{F2}, and by Lemma 1.8 in \cite{F2}, $$\mu(\mathbb{T}^n)\leq \liminf \mu_{k_\textsf{n}}(\mathbb{T}^n)=1=\limsup \mu_{k_\textsf{n}}(\mathbb{T}^n) \leq \mu(\mathbb{T}^n),$$
 which yields $\mu(\mathbb{T}^n)=1.$ Thus $\mu$ is a probability measure on $\mathbb{T}^n.$
%
We next show that $spt(\mu) \subseteq \tilde{F}_0$ :

In the proof of Lemma \ref{Complement_F0},  we have shown that $\tilde{F}_0^c$  can be written as a countable union of open sets  $K_{\epsilon_\textsf{m}}$ so that $\tilde{F}_0^c=\bigcup   K_{\epsilon_\textsf{m}}$ and 
$K_{\epsilon_\textsf{m}} \subseteq  \tilde{F}^c_{k_\textsf{n}}\Rightarrow \mu_{k_\textsf{n}}(K_{\epsilon_\textsf{m}})=0 $  with  $\textsf{n}\geq \textsf{n}_\textsf{m}$ for some 
$\textsf{n}_\textsf{m}\in \mathbb{Z}$. 
Since $\mu_{k_\textsf{n}}\rightarrow \mu   \  {\rm (weakly)}$ and  the set  $K_{\epsilon_\textsf{m}} $ is open, we have 
$$ \mu(K_{\epsilon_\textsf{m}} )\leq \underset{ n\rightarrow \infty }\liminf \mu_{k_\textsf{n}}(  K_{\epsilon_\textsf{m}} ) \leq \underset{ N\rightarrow \infty }\lim (\underset{\textsf{n}\geq N} \inf \mu_{k_\textsf{n}}(K_{\epsilon_\textsf{m}} ))=0.$$
Thus $ \mu(K_{\epsilon_\textsf{m}} )=0$
for each $\textsf{m}$, and hence $ \mu(\tilde{F}_0^c)=0$.
It follows that  
$spt(\mu) \subseteq \tilde{F}_0.$

\medskip

(ii) Since $\mu$ is regular, there exists a compact set $K\subset B$ with $\mu (K)>0$.  Now $\mu (K)  =\mu (K \cap \tilde{F}_0)$ by part (i). 

By the weak convergence in (i), $\{ \mu_{k_\textsf{n}} \}$ converges to $\mu$ with respect to the Prokhorov metric $d_P$ on $\textsf{P}(\mathbb{T}^n)$, i.e., $d_P(\mu_{k_\textsf{n}},\mu)\rightarrow 0 $.  From the definition of 
$d_P$, 
it is possible to find a decreasing sequence  $\alpha_{\textsf{n}}\geq d_P( \mu_{k_\textsf{n}},\mu) $ that converges to zero, and  the inequality 
 $\alpha_{\textsf{n}}+ \mu_{k_\textsf{n}}(E_{\alpha_{\textsf{n}}})\geq \mu_{}(E) 
 $
holds for  every  $\textsf{n}$ and every Borel set $E\subseteq \mathbb{T}^n$. Choose $E=K.$
Then there is an index $\textsf{n}_0$ such that $$\mu_{k_\textsf{n}}(K_{\alpha_{\textsf{n}}}\cap \tilde{F}_{k_\textsf{n}})=\mu_{k_\textsf{n}}(K_{\alpha_{\textsf{n}}}) \geq \mu_{}(K) -\alpha_{\textsf{n}_0}>0$$ for $\textsf{n}\geq \textsf{n}_0.$  In particular, we can choose an appropriate multiple $\beta_\textsf{n} \alpha_\textsf{n} $ of $\alpha_\textsf{n}$ (if necessary) with the following properties:

(1) $\beta_\textsf{n}  \geq 1$ and $\beta_\textsf{n} \alpha_\textsf{n} $  converges to $0$, 

(2) $K_{\beta_\textsf{n} \alpha_\textsf{n}}\cap \tilde{F}_{k_\textsf{n}}$ contains all points of $\tilde{F}_{k_\textsf{n}}$ converging to $K \cap \tilde{F}_0$ as $\textsf{n} \rightarrow \infty$ 
(in the sense of  Proposition \ref{deflection}).

As a result of (1), $\mu_{k_\textsf{n}}(K_{\beta_\textsf{n} \alpha_\textsf{n}}\cap \tilde{F}_{k_\textsf{n}})  \geq \mu_{k_\textsf{n}}(K_{\alpha_{\textsf{n}}}\cap \tilde{F}_{k_\textsf{n}})>0$ for $\textsf{n}\geq \textsf{n}_0.$ Hence we may assume that $K_{ \alpha_\textsf{n}}$ has  property (2) by considering $K_{\beta_\textsf{n} \alpha_\textsf{n}}$ in place of $K_{\alpha_\textsf{n}}$. Set $K_{k_\textsf{n}}$ equal to the closure of $K_{\alpha_{\textsf{n}}}\cap \tilde{F}_{k_\textsf{n}}$ 
(
in $\mathbb{T}^n$).  Note that the closed subset $K_{k_\textsf{n}}$ of the compact space $\mathbb{T}^n$ is compact and must be in $\tilde{F}_{k_\textsf{n}}$. 
By  Proposition \ref{deflection}, \ 
$K_{k_\textsf{n}}
 \underset{\textsf{n} \rightarrow \infty}{\longrightarrow}  K \cap \tilde{F}_0  
$
in the Hausdorff metric induced by $d_{\mathbb{T}^n}$.
\end{pf}

\bigskip

The following lemma is extracted from Theorem 1.1 in \cite{KP1} and follows from a variant of the Multidimensional Ledrappier-Young Formula  (see Lemma 3.1, Lemma 4.3 and the proof of Theorem 1.1 in \cite{KP1}).

\begin{lemma}\label{Ledrappier} Let $T_k$ be the toral endomorphisms  given by (\ref{pert_matrix}). Assume that $\mu_{k}$ is an ergodic $T_k$-invariant
Borel probability measure on the perturbed fractal $\tilde{F_k}\subseteq \mathbb{T}^n$. 
If $\mu_{k}(B)>0$ for a Borel set $B\subseteq \tilde{F_k}$,  then we have $$dim_H B =dim_H(\tilde{F_k})=\delta_{k}.$$
 \end{lemma}

\bigskip

Below, we  use the notation $$\delta:=dim_H F =dim_H(\tilde{F}_0).$$

\begin{lemma}\label{full_dimension2}   Let $\mu$ be as  in the proof of Lemma \ref{Borel_measure}. If $\mu(B)>0$ for a Borel set $B\subseteq \tilde{F}_0$, then $dim_H B =
dim_H(\tilde{F}_0)=\delta$.
Therefore, $\mu$ has full dimension.
\end{lemma}

\begin{pf}
Let $\mu_{k_\textsf{n}}$ be as in
Lemma \ref{Borel_measure}-(i). By the proof of Lemma \ref{Borel_measure}-(ii), 
there is a compact set $K\subset B$ and a sequence of compact sets 
$K_{k_{\textsf{n}}}\subseteq \tilde{F}_{k_{\textsf{n}}}
$ converging to $K\cap \tilde{F}_0=K$ such that $\mu(K), \  \mu_{k_{\textsf{n}}}(K_{k_{\textsf{n}}})>0$ for  $\textsf{n}\geq \textsf{n}_0$. By  Lemma \ref{Ledrappier}, $dim_H(K_{k_{\textsf{n}}})=\delta_{k_{\textsf{n}}}$ for  $\textsf{n}\geq \textsf{n}_0$. 
In view of the observation at the beginning of this section, Proposition \ref{deflection} and Corollary \ref{neighboring_pieces5},
we can replace $\tilde{F}, F_k$ by $K
, \  K_{k_{\textsf{n}}}$
in the statement of Theorem \ref{monotonicity2} (while using the same type of covering sets of $K
, \  K_{k_{\textsf{n}}}$ in the proofs, $n$-cubes and pieces of $\tilde{F}, F_k$).  We then obtain  $dim_H(K_{k_{\textsf{n}}})\leq dim_H(K\cap \tilde{F}_0) $ for $\textsf{n}\geq \textsf{n}_0$.   Therefore, 
$$\delta_{k_{\textsf{n}}}=dim_H(K_{k_{\textsf{n}}})    \leq 
dim_H K   \leq dim_H B $$   for $\textsf{n}\geq \textsf{n}_0$.  But by Theorem \ref{Hausdorff},
$\underset{\textsf{n}\rightarrow\infty}{\lim} \delta_{k_{\textsf{n}}}=\delta$.  This gives $dim_H B =\delta$. 
\end{pf}

\bigskip

We thus say that \textit{a measure $\mu$ on $\tilde{F}_0$ or $F$ has property} (P) provided that the following holds.

\bigskip

\begin{center}
  (P) \ \ If $B\subseteq \tilde{F}_0$ or $B\subseteq F$ is a Borel set with $\mu(B)>0$, then  $dim_H B =
\delta$.
\end{center}

\bigskip

%

\begin{lemma}\label{lift} Let $F=F(T,A)$. Then there exists a Borel probability measure  $\mu$ on $\tilde{F}_0\subseteq \mathbb{T}^n$ with property (P) if and only if there is
such a measure $\overline{\mu}$ on $F \ {\rm mod \ 1} $.
Therefore, by  Lemma \ref{full_dimension2}, there is a Borel probability measure $\overline{\mu}$ on $F \ {\rm mod \ 1} $
 with property (P).
\end{lemma}
 
\begin{pf} It suffices to prove one of the implications. So we prove the necessity. The proof consists of lifting the measure $\mu$ on $\tilde{F}_0$ to $F(J,\tilde{A})$, then transferring it to $F(T,A)$ and finally lifting it down to $F \ { \mathrm{mod}} \ 1 \ \subseteq \mathbb{T}^n$.

Assume that $\mu$ is a Borel probability measure on $\tilde{F}_0\subseteq \mathbb{T}^n$ with property (P). Consider the mapping: $\pi : F(J,\tilde{A})\rightarrow \tilde{F}_0$ given by $\pi(x)=x \ \rm mod \ 1$.
We define a subset $S$ of $F(J,\tilde{A})$ to be measurable if $S=\pi^{-1}(E)$ for some $\mu$-measurable set $E\subseteq \tilde{F}_0$. Then $\pi$ is a measurable function and
it is known (cf. \cite{ww}, \cite{www}) that  there exists a Borel probability measure $\mu_1$ on $F(J,\tilde{A})$ such that $\pi_{\mu_1}=\mu$ (i.e., $\mu_1(\pi^{-1}(E))=\mu(E)$), where $\pi_{\mu_1}$ is the push-forward of $\mu_1$.

Now transfer $\mu_1$ to 
$F(T,A)$ in the obvious way : Define $\mu_2 (PS):=\mu_1(S)$, where $PS\subseteq F$ and  $S\subseteq F(J,\tilde{A})$ is measurable.

We now consider the Borel algebra on $F \ \rm mod \ 1.$ Finally, we lift $\mu_2$ down to $F \ {\rm mod \ 1} $ by a push-forward measure again. Explicitly, $\overline{\mu}:=\pi_{\mu_2}$, where $\pi : F(T,A)\rightarrow F \ \rm mod \ 1$ given by $\pi(x)=x \ \rm mod \ 1$.
Clearly, $\overline{\mu}$ will have property (P)  by the observation at the beginning of this section again.
\end{pf}

\begin{lemma}\label{full_dimension3} There exists a $T$-invariant Borel probability measure  $\mu^*$ on $F(T,A) \ {\rm mod \ 1} \subseteq \mathbb{T}^n$ with full dimension. In fact,
$\mu^*$ has property (P).

\end{lemma}

\begin{pf} 
 \textit{Construction of $\mu^*$}: \quad For $i\leq n$, consider the closed (special) rectangles  $$R_i= [a_1,b_1]\times \cdots \times [a_i,b_i]=\{ x \in  \mathbb{R}^i \ : \ a_j \leq x_j \leq  b_j \ \ \ 1\leq j \leq i \}
 $$
and open (special) rectangles $$\overset{\circ}{R}_i= (a_1,b_1)\times \cdots \times (a_i,b_i)=\{ x \in  \mathbb{R}^i \ : \ a_j < x_j <  b_j \ \ \  1\leq j \leq i \}
.$$
$R_i,\overset{\circ}{R}_i$ can also be viewed as subsets of $\mathbb{R}^n$ in the obvious way: $R_i \equiv \{ (x,0,0,...,0) \in  \mathbb{R}^n \ :  x \in R_i  \}$, $\overset{\circ}{R}_i \equiv \{ (x,0,0,...,0) \in  \mathbb{R}^n \ :  x \in \overset{\circ}{R}_i  \}$, where $\equiv $ is used for the equivalence of sets. 
Note that $\overset{\circ}{R}_i / \mathbb{Z}^n$ is a finite union of non-overlapping rectangles in $[0,1]^n$ (not necessarily open or closed, e.g. half-open rectangles are possible). To handle the open sets in $\mathbb{T}^n$, we will consider $R_i,\overset{\circ}{R}_i$ as subsets of $\mathbb{R}^n$ and use the  rectangles in $\overset{\circ}{R}_i / \mathbb{Z}^n$.

 We will  construct a required type of measure on open sets. It is known that every non-empty open set $U$ in $\mathbb{R}^n$ can be written as a countable union of non-overlapping 
closed rectangles  in $\mathbb{R}^n$ \cite[Problem 9, p.35]{J}.
Then it follows that every non-empty open set in $\mathbb{R}^n$ can be written as a countable union of disjoint open rectangles  and a subset of the vertex set of such rectangles.
Since any open set in the subspace $\mathbb{T}^n$ is of the form $G=U /\mathbb{Z}^n$, the same is true for the open sets in $\mathbb{T}^n$.  
Let
$$\mathcal{R}=\{ \ J^{\circ} \  :  \ J^{\circ} \ \mathrm{is \ a \ rectangle \ in }\ \overset{\circ}{R}_i / \mathbb{Z}^n \ \mathrm{with \ vertices \ that \ have \ rational \ coordinates,}   \ 
1 \leq i \leq n \}.$$

Now we  construct a $T$-invariant measure on $F \ {\rm mod \ 1} \subseteq \mathbb{T}^n$. We start with rectangles $J^{\circ}$, then continue the construction with open sets and end up with Borel sets.
 Let $\overline{\mu}$ be a  measure on $F\ {\rm mod \ 1}$ as in Lemma \ref{lift}. Let $T_F$ denote the restriction of $T$ to $F\ {\rm mod \ 1}$. More explicitly, 
 let   $T_F: F\ {\rm mod \ 1}\rightarrow F\ {\rm mod \ 1} $ be given by $x\mapsto Tx \ {\rm mod \ 1},$ where $x\in F\ {\rm mod \ 1}.$
 For each $J^{\circ}\in \mathcal{R}$ and $k\in \mathbb{N}$, let
 \begin{equation}\label{mu_star1}
 S_k(J^{\circ})=\frac{1}{k}\sum_{l=0}^{k-1}\overline{\mu}(T_F^{-l}(J^{\circ}))  \quad  \quad \quad \mathrm{ and}  \quad \quad \quad S_k\left(T_F^{-1}(J^{\circ})\right)=\frac{1}{k}\sum_{l=0}^{k-1}\overline{\mu} \left(T_F^{-(l+1)}(J^{\circ})\right).
 \end{equation}
All  subsequences of these sequences are  bounded sequences (because $\overline{\mu}$ is a probability measure), and hence,  by the Bolzano–Weierstrass theorem, they have  convergent subsequences (depending on each $J^{\circ}$). But $\mathcal{R}$ is a countable set. Therefore, there is a
 subsequence $k_\textsf{s}$ of natural numbers such that  $\lim_{\textsf{s} \rightarrow \infty} S_{k_\textsf{s}}(J^{\circ})$ and $\lim_{\textsf{s} \rightarrow \infty} S_{k_\textsf{s}}\left(T_F^{-1}(J^{\circ})\right)$ exist for every $J^{\circ}$. Define
 \begin{equation}\label{mu_star2}
 \mu^*(J^{\circ}):=\lim_{\textsf{s} \rightarrow \infty} S_{k_\textsf{s}}(J^{\circ})  \quad  \quad \quad \mathrm{ and}  \quad \quad \quad  \mu^*\left(T_F^{-1}(J^{\circ}) \right):=\lim_{\textsf{s} \rightarrow \infty} S_{k_\textsf{s}}\left(T_F^{-1}(J^{\circ})\right).
 \end{equation}
 Next, by (\ref{mu_star1}) and (\ref{mu_star2}), we observe that 
 $$|\mu^*(T_F^{-1}(J^{\circ}))-\mu^*(J^{\circ})|=\lim_{\textsf{s} \rightarrow \infty} |S_{k_\textsf{s}}(T_F^{-1}(J^{\circ}))-S_{k_\textsf{s}}(J^{\circ})|=\lim_{\textsf{s} \rightarrow \infty}\frac{1}{k_\textsf{s}} |\overline{\mu}(T_F^{-k_\textsf{s}}(J^{\circ})) - \overline{\mu}(J^{\circ})|= 0.$$ 
That gives  \begin{equation}\label{mu_star3}\mu^*(T_F^{-1}(J^{\circ}))=\mu^*(J^{\circ}).\end{equation}

Now we can approximate an open set $G\subset \mathbb{T}^n$ by a finite number of non-overlapping  
rectangles  in $\mathcal{R}$ 
since it is a countable union of such rectangles.
Let
$V_0$
be the vertex set of 
the rectangles in $\mathcal{R}$.  
Note that we exclude the trivial case $dim_H F=0$ from the discussion.
Therefore,
$\delta = dim_H F > 0$.
 Because $V_0$ is countable, $dim_H (V_0)=0$. It then follows from Lemmas \ref{full_dimension2}-\ref{lift}  that $\overline{\mu}(V_0)=0$.

Let $P$ denote a finite union of non-overlapping  closed rectangles in 
$\mathcal{R}.$
Thus it can be written as a finite union of disjoint open rectangles in $\mathcal{R}$ and a subset of $V_0$. That subset may be
ignored because it has $\overline{\mu}$-measure 0.
Thus we can write
\begin{equation}\label{special_rectangle} P=S_\textsf{m}\bigcup \left(\overset{\textsf{m}}{\underset{\textsf{i}=1}{\cup}} J_\textsf{i}^{\circ}\right),
\end{equation}
 where $S_\textsf{m}\subseteq V_0$ and the $J_\textsf{i}^{\circ}\in \mathcal{R}$ are disjoint and open.
Using the measure in (\ref{mu_star2}), we define $$\mu^*(P):=\sum _{\textsf{i}=1}^\textsf{m} \mu^*(J_\textsf{i}^{\circ}), \quad \quad \quad \mu^*(T_F^{-1}(P)):=\overset{\textsf{m}}{\underset{\textsf{i}=1}{\sum}} \mu^*(T_F^{-1}(J_\textsf{i}^{\circ})).$$ 
Then $\mu^*(T_F^{-1}(P))=\mu^*(P)$
by (\ref{mu_star3}). Also 
(\ref{mu_star1}), (\ref{mu_star2}) yield 
\begin{equation}\label{polygon}
 \mu^*(P)= \sum _{\textsf{i}=1}^\textsf{m} \lim_{\textsf{s} \rightarrow \infty} \frac{1}{k_\textsf{s}}\sum_{l=0}^{k_\textsf{s}-1}\overline{\mu}(T_F^{-l}(J_\textsf{i}^{\circ}))=\lim_{\textsf{s} \rightarrow \infty} \frac{1}{k_\textsf{s}}\sum_{l=0}^{k_\textsf{s}-1}\sum _{\textsf{i}=1}^\textsf{m}\overline{\mu}(T_F^{-l}(J_\textsf{i}^{\circ}))=\lim_{\textsf{s} \rightarrow \infty} \frac{1}{k_\textsf{s}}\sum_{l=0}^{k_\textsf{s}-1}\overline{\mu}(T_F^{-l}P).
\end{equation}
Note that $P$ can be expressed in different ways
as a union of a vertex set and disjoint open rectangles as in (\ref{special_rectangle}).
But the equality $\sum _{\textsf{i}=1}^\textsf{m}\overline{\mu}(T_F^{-l}(J_\textsf{i}^{\circ}))=\overline{\mu}(T_F^{-l}(\bigcup_{\textsf{i}=1}^\textsf{m} J_\textsf{i}^{\circ})) = \overline{\mu}(T_F^{-l}P)$ is independent of the way
 $P$ is written as a disjoint union   because $\overline{\mu}$ is countably additive. This implies that  $\mu^*(P)$ is well-defined.

Let $G\subset \mathbb{T}^n$ be an open set. We define $\mu^*(G)$ by the classical recipe:
$$\mu^*(G):=\sup \{ \mu^*(P) \ : \ P\subseteq G \}.$$
For an arbitrary set $\textsf{A}$, we finally let $\mu^*(\textsf{A}):=\inf \{ \mu^*(G) \ : \ \textsf{A}\subseteq G, \ \ G \textrm{ \ is open}  \} $.
Strictly speaking, $\mu^*$ is an outer measure. If we take Carath$\acute{e}$odory's definition for measurability, $\mu^*$ restricted to Borel sets is a regular measure \cite[p.50]{J}.

\textit{T-invariance of $\mu^*$}: As above, let $G\subset \mathbb{T}^n$ be an open set. Then $G=S\bigcup \left(\overset{\infty}{\underset{\textsf{i}=1}{\cup}}J_\textsf{i}^{\circ}\right)$,
   where $S$ is a subset of $V_0$ and the $J_\textsf{i}^{\circ}\in \mathcal{R}$ are disjoint. Renaming the $P$  in (\ref{special_rectangle}) as $P_\textsf{m}$, we may 
   write the union $G=S\bigcup \left(\overset{\infty}{\underset{\textsf{i}=1}{\cup}} J_\textsf{i}^{\circ}\right)$ as  
   $G=\overset{\infty}{\underset{\textsf{m}=1}{\bigcup}}  P_\textsf{m}$, where $P_\textsf{m}\subseteq P_{_\textsf{m}+1}$, \ $S_\textsf{m}\subseteq S_{_\textsf{m}+1}$.   
This gives
\begin{eqnarray}\label{T_invariance11}
\mu^*(G) & = & \underset{\textsf{m}\rightarrow \infty} {\lim} \mu^*(P_\textsf{m})=\underset{\textsf{m}}{\sup} \ \mu^*(P_\textsf{m}) =  \underset{\textsf{m}}{\sup} \ \lim_{\textsf{s} \rightarrow \infty} \frac{1}{k_\textsf{s}}\sum_{l=0}^{k_\textsf{s}-1}\overline{\mu}(T_F^{-l}(P_\textsf{m}))  \quad \quad  \mathrm{by} \ \ (\ref{polygon})
\nonumber \\
& = & \lim_{\textsf{s} \rightarrow \infty} \frac{1}{k_\textsf{s}} \sum_{l=0}^{k_\textsf{s}-1}\underset{\textsf{m}}{\sup} \ \overline{\mu}(T_F^{-l}(P_\textsf{m}))=
\lim_{\textsf{s} \rightarrow \infty} \frac{1}{k_\textsf{s}} \sum_{l=0}^{k_\textsf{s}-1} \overline{\mu}(T_F^{-l}G).
\end{eqnarray}
For a closed set $K\subset \mathbb{T}^n$, we let $G_K=\mathbb{T}^n\setminus K$. Then, by (\ref{T_invariance11}),
\begin{eqnarray}\label{T_invariance2} 
\mu^*(K)& = & 1-\mu^*(G_K)=1-\lim_{\textsf{s} \rightarrow \infty}\frac{1}{k_\textsf{s}} \sum_{l=0}^{k_\textsf{s}-1} \overline{\mu}(T_F^{-l}(G_K)) \nonumber \\
 & = &\lim_{\textsf{s} \rightarrow \infty}\frac{1}{k_\textsf{s}} \sum_{l=0}^{k_\textsf{s}-1}(1-\overline{\mu}(T_F^{-l}(G_K)))  =
\lim_{\textsf{s} \rightarrow \infty}\frac{1}{k_\textsf{s}}\sum_{l=0}^{k_\textsf{s}-1}\overline{\mu}(T_F^{-l}(K)).
\end{eqnarray}
Similar to the proof of (\ref{mu_star3}), by (\ref{T_invariance11}), (\ref{T_invariance2}), we obtain 
\begin{eqnarray}\label{T_invariance3}
\mu^*(G)=\mu^*(T_F^{-1}(G)) \ \ \ \mathrm{and} \ \ \  \mu^*(K)=\mu^*(T_F^{-1}(K))
\end{eqnarray}  
leading to $T$-invariance of $\mu^*$: Since $\mu^*$ is regular, for a Borel set $B$, there exists a sequence $\{G_\textsf{i}\}$ of open sets and a sequence $\{K_\textsf{i}\}$ of compact sets such that $K_\textsf{i}\subseteq B\subseteq G_\textsf{i}$
and $\underset{\textsf{i}\rightarrow \infty}{\lim} \mu^*(G_\textsf{i}\setminus K_\textsf{i})=0.$ Then $\underset{\textsf{i}\rightarrow \infty}{\lim} \mu^*(G_\textsf{i}\setminus B)=0$ implies
\begin{eqnarray}\label{T_invariance4}
\underset{\textsf{i}\rightarrow \infty}{\lim}\mu^*(G_\textsf{i})=\mu^*(B).
\end{eqnarray}

On the other hand, (\ref{T_invariance3})
implies
$\underset{\textsf{i}\rightarrow \infty}{\lim}\mu^*(T_F^{-1}(G_\textsf{i}\setminus K_\textsf{i}))=0.$ Hence $\underset{\textsf{i}\rightarrow \infty}{\lim}\mu^*(T_F^{-1}(G_\textsf{i}\setminus B))=0$ and
\begin{eqnarray}\label{T_invariance5}
\underset{\textsf{i}\rightarrow \infty}{\lim}\mu^*(T_F^{-1}(G_\textsf{i}))=\mu^*(T_F^{-1}( B)).
\end{eqnarray}
Combining (\ref{T_invariance3}), (\ref{T_invariance4}), (\ref{T_invariance5}), we then get
$$\mu^*(B)=\underset{\textsf{i}\rightarrow \infty}{\lim}\mu^*(G_\textsf{i})=\underset{\textsf{i}\rightarrow \infty}{\lim}\mu^*(T_F^{-1}(G_\textsf{i}))=\mu^*(T_F^{-1}( B)),$$
which is what we mean by $T$-invariance of $\mu^*$.

\textit{$\mu^*$ has full dimension :} Now assume that $B$ is a Borel set with $B\subset F \ {\rm mod \ 1} $ and $\mu^*(B)>0$. By the regularity of $\mu^*$, there is a compact set $K\subseteq B$ with  $\mu^*(K)>0$. By (\ref{T_invariance2}),  $\mu^*(K)>0 \  \Rightarrow \ \overline{\mu}(T_F^{-l}K)>0$ for some $l$.
Then it follows from Lemma \ref{lift}
that $dim_H(T_F^{-l}K)=\delta$.
We then conclude that $dim_H B = dim_H K={\delta}.$
\end{pf}

\bigskip

Let $\mathcal{M}(X,T)$ denote the set of $T$-invariant Borel probability measures  on $X=F(T,A) \ {\rm mod \ 1} \subseteq \mathbb{T}^n$ and let
$$  \mathcal{M}_{\delta}(X,T)= \{ \mu \in \mathcal{M}(X,T) : \ ``\textrm{$B$ is a Borel set with} \ \mu(B)>0 "
\ \ \Rightarrow  \ dim_H B = \delta \}.$$


In other words, $\mathcal{M}_{\delta}(X,T)$ consists of the measures in $\mathcal{M}(X,T)$ which have property (P). Now we show that there exists an ergodic $T$-invariant Borel probability measure  $\mu^*$ on $X$
of full dimension.

\bigskip

\begin{theorem}\label{full_dimension4}
Assume that $F=F(T,A)$ is an integral self-affine set. 
Then, viewed as an invariant subset of the n-torus under the toral endomorphism $T$, there is an ergodic $T$-invariant Borel probability measure on $X$
of full dimension.
\end{theorem}


\begin{pf} (\textbf{Proof of Theorem 1.3})
 By Lemma \ref{full_dimension3},  $\mathcal{M}_{\delta}(X,T)$ is non-empty. Let $\{\mu_k\} \subset \mathcal{M}_{\delta}(X,T)$ be a sequence converging weakly to a measure $\mu$. 

 We first prove that $\mu$ has property (P).
Assume that $\mu(B)>0$  for a Borel set $B\subseteq X$. By regularity, there exists a compact set $K\subseteq B$ with $\mu(K)>0$. 
Let $\mathbf{1}_K$ denote the indicator function of $K$.   
It can be approximated by continuous  bounded functions because the set of continuous bounded functions are dense in the set of $\mu$-integrable functions \cite[Proposition 7.9, p.217]{F}.
Thus, for any positive integer $\textsf{n}$, there is a  continuous bounded function $f_{\textsf{n}}$ on $X$ such that $|\mathbf{1}_K-f_{\textsf{n}}|<\frac{1}{\textsf{n}}$. It follows that $|f_{\textsf{n}}|\leq 2$, i.e., $f_{\textsf{n}}$ is uniformly bounded. Then 
$$\mu(K)=\int_X \mathbf{1}_K \ d\mu =\int_K \left( \mathbf{1}_K-f_{\textsf{n}} \right) \ d\mu + \int_K f_{\textsf{n}} \ d\mu \leq \frac{1}{\textsf{n}} + \int_K f_{\textsf{n}} \ d\mu
\Longrightarrow \mu(K) \leq \underset{\textsf{n}\rightarrow \infty}{\limsup} \ \left( \frac{1}{\textsf{n}} + \int_K f_{\textsf{n}} \ d\mu \right).
$$
By \cite {B}, the weak convergence implies 
$$\mu(K)\leq \underset{\textsf{n}\rightarrow \infty}{\limsup} \  \int_K f_{\textsf{n}} \ d\mu =\underset{\textsf{n}\rightarrow \infty}{\limsup}  \left( \underset{k\rightarrow \infty}{\lim} \int_K f_{\textsf{n}}  \ d\mu_k \right)  \leq \underset{\textsf{n}\rightarrow \infty}{\limsup}  \left( \underset{k\rightarrow \infty}{\lim} \int_K 2 \ d\mu_k \right)  =\underset{k\rightarrow \infty}{\limsup} \ 2\mu_k(K).$$ As a result,  we must have 
$\mu_{k}(K)>0$ for some $k$ because of $\mu(K)>0$.  Since $\mu_k\in \mathcal{M}_{\delta}(X,T)$, it has property (P) by the definition of $ \mathcal{M}_{\delta}(X,T)$.
Hence $K\subseteq B \Longrightarrow dim_H B = dim_H K=\delta$. 

Secondly, we show that $\mu$ is $T$-invariant. Let   $T_F: F\ {\rm mod \ 1}\rightarrow F\ {\rm mod \ 1} $ be given by $x\mapsto Tx \ {\rm mod \ 1}$ as in the previous proof.
 It is clear that $T_F$ is bounded with respect to the natural torus metric. It is also continuous since  $T_F$ can be written as the composition of the quotient map  $q: \mathbb{R}^n \longrightarrow \mathbb{R}^n / \sim \ =\mathbb{R}^n$ mod 1 (for $\mathbb{R}^n$  mod 1, the equivalence relation is $x\sim y \Longleftrightarrow x- y\in \mathbb{Z}^n$)  and  $T$ on $F\ {\rm mod \ 1}$. 
  For continuous bounded functions $g$, we have $$\int_X g \ d\mu =  \underset{k\rightarrow \infty}{\lim} \int_X g  \ d\mu_k =
  \underset{k\rightarrow \infty}{\lim} \int_X g \circ T_F  \ d\mu_k =\int_X g \circ T_F   \ d\mu $$ because $\mu_{k}\rightarrow \mu   \  {\rm (weakly)}$ and  $\mu_k\in \mathcal{M}_{\delta}(X,T)$ is $T$-invariant. This implies that $\mu$ is $T$-invariant too:  
 Let $\textsf{B}$ be any $\mu$-measurable set. As above,   $\mathbf{1}_\textsf{B}$ is the limit of a sequence of  uniformly bounded continuous  functions $g_{\textsf{n}}$.  
 Therefore, we can use the Dominated Convergence Theorem to conclude that
  $$\mu(\textsf{B})=\int_X \mathbf{1}_\textsf{B} \ d\mu= \underset{\textsf{n}\rightarrow \infty}{\lim} \int_X  g_{\textsf{n}} \ d\mu=\underset{\textsf{n}\rightarrow \infty}{\lim} \int_X g_{\textsf{n}} \circ T_F\ d\mu =  \int_X \mathbf{1}_\textsf{B} \circ T_F \ d\mu=\mu(T_F^{-1}\textsf{B}).$$

Thus $  \mathcal{M}_{\delta}(X,T)$ is closed. Since $\mathcal{M}(X,T)$ is compact, $  \mathcal{M}_{\delta}(X,T)$
is also compact and non-empty. Then, by the  Krein-Milman theorem, 
the compact and convex set $\mathcal{M}_{\delta}(X,T)$ has an extremal point
 in $  \mathcal{M}_{\delta}(X,T)$, which is the required ergodic measure \cite[Proposition 4.6.2, p.86]{BS}, \cite[Thoerem 6.10, p.152]{W}.
\end{pf}
 
\bigskip

\section{Appendix : Dimension Formulas for Perturbed Fractals $F_k$}\label{appendix}


\subsection{Deconstruction of $F$}\label{Decon}

To find the dimension of an arbitrary integral self-affine set $F$,
we need  to use graph-directed sets to deal with the overlapping pieces of $F$. Here we mention that  we consider
IFSs involving affine maps that are
not necessarily similarities. We present the relevant
material now.
Let V be a set of vertices which we label $\{1, 2, ..., \textsf{m}\}$, and E be a set of
`directed edges' with each edge starting and ending at a vertex so that $(V,E)$
is a directed graph.
We write $E_{i, j}$
for the set of edges from vertex $i$ to vertex $j$, and $E_{i, j}^r$ for the
set of sequences of $\textsf{r}$  edges $(e_1,e_2,...,e_{\textsf{r}})$ which form a directed path
from vertex $i$ to vertex $j$.

\bigskip

We assume that $f_{e}(x)=T^{-1}(x+e)$ for $e \in E_{i, j}\subset \mathbb{Z}^{n}$,
where $T\in M_n({\Bbb{Z}})$ is expanding.
It is well known that there exists a
unique family of nonempty compact sets $\Gamma_1,...,\Gamma_\textsf{m} \subset \mathbb{R}^{n}$
such that
\begin{eqnarray}\label{eqn9}
\Gamma_i & = & \bigcup_{j =1}^\textsf{m} \bigcup _{e \in E_{i, j}} f_e(\Gamma_j)=\bigcup
_{j =1}^m T^{-1}(\Gamma_j + E_{i, j})
\end{eqnarray}

The set $\{ f_e: \ e \in E \}$ is called a
\emph{graph-directed iterated function system} and the sets $\Gamma_1,...,\Gamma_\textsf{m} $
form a family of \emph{graph-directed sets} \cite[p.48]{F2}.

We recall some facts from \cite{HLR}. Let $\partial S $ stand for the boundary of a set $S$ and $\# S$ the cardinality of $S$. For any integral
self-affine set $F=F(T,A)$ with at least two digits, there exist graph directed sets $\Gamma_0, \Gamma_1,...,\Gamma_\textsf{m}$
such that
$\Gamma_0=F$ or
$\Gamma_0=\partial F$. This graph directed system is obtained from an \emph{auxiliary self-affine tile}
$\Gamma:=F(T,D)$ such that $\Gamma=T^{-1}(\Gamma+D)$, \ \ 
$F\subset  { \Gamma }$, \ \ $\Gamma+A \subseteq T \Gamma$ and

\bigskip

(I) For each $i\in \{0,1,2,..,m\}$, \quad $ E_{i, j}\subset D$ and $ \sum_{j=0}^\textsf{m} \# E_{i, j}\leq \# D=|\det(T)|=N$.


(II) For each $i$, $E_{i, j}\cap  E_{i, j'}= \emptyset$ if $j \neq j'$.

(III) $D =a+c_0D'$, where $a \in \mathbb{Z}^n$, $c_0\in \mathbb{N}$ and $D'$ is a complete residue system for $T\in M_n({\Bbb{Z}})$.
In case of diagonal matrices like $T=diag(m_1,m_2,...,m_n)\in M_n({\Bbb{Z}})$, \ $D'$ will be of the form $$D'=\prod_{p=1}^n \{0,1,2...,m_p-1\}.$$
\bigskip

The vertices and the edges are as follows: We define the index sets
$$J_k=\{\mathbf{j}=(j_1,j_2,...,j_k): 1\leq j_i \leq q \},  \ \ \ \Sigma_k=\{\textbf{i}=(i_1,i_2,...,i_k): 1\leq  i_j \leq N \}.$$
For  ${\mathbf{j}}\in J_k$ and $\textbf{i}\in \Sigma_k$, we set $a_{\mathbf{j}}=\Sigma_{i=1}^{k}T^{(k-i)}a_{j_i}$ and $d_{\textbf{i}}=\Sigma_{j=1}^{k}T^{(k-j)}d_{i_j}$.
Let $$f_{\mathbf{j}}( \Gamma )=T^{-k}(  \Gamma  + a_{\mathbf{j}}), \ \ \ \ \psi_{\textbf{i}}( \Gamma ) =T^{-k}( \Gamma  +d_{\textbf{i}}).$$
For each $\textbf{i}\in \Sigma_k$, let
\begin{equation}\label{del}
 \Delta (\textbf{i} ):= \{ a_{\mathbf{j}}-d_{\textbf{i}} \ : \ {\mathbf{j}}\in J_k , \ f_{\mathbf{j}}( \Gamma )  \cap \psi_{\textbf{i}}( \Gamma )\neq \emptyset \}, \quad \quad 
 \quad \mathcal{S}_k=\{ \textbf{i}\in \Sigma_k \ : \ \Delta (\textbf{i}) \neq \emptyset \}
\end{equation} 
Finally, we define the vertex set
$V_F = \{ \Delta (\textbf{i} ) \ : \  \textbf{i}\in \overset{\infty}{\underset{{k=1}}{\cup}} \mathcal{S}_k \} $. On the other hand, $a_{\mathbf{j}}-d_{\textbf{i}} \in \Delta (\textbf{i} ) \Longrightarrow |a_{\mathbf{j}}-d_{\textbf{i}}|\leq diam(\Gamma)$ and $a_{\mathbf{j}}, d_{\textbf{i}} \in\mathbb{Z}^{n}$ so that $V_F$ is finite.  Since $V_F$ is a finite set, we write
\begin{equation}\label{Ver}
  V_F= \{v_0=\Delta (0):=\{0\}, v_1=\Delta (\textbf{i}_1),v_2=\Delta (\textbf{i}_2),...,v_m=\Delta (\textbf{i}_\textsf{m})  \}.
\end{equation}
We will also assume that first $l \ (\leq \textsf{m})$ vertices in $V_F$ form $\mathcal{S}_1 $, i.e., $\mathcal{S}_1=\{v_1,v_2,...,v_l\}$.
For $0 \leq i,j \leq m$, the edges are defined by
\begin{equation}\label{OrEd}
E_{i, j}=\{d_s\in D \ : \ \Delta (\textbf{i}_i s)= \Delta(\textbf{i}_j),  \  \ 1 \leq s \leq N \}.
\end{equation}
If $\bigcup_{j=0}^\textsf{m} E_{i, j} = \emptyset$ for some $i$, we will discard $v_i$ and all paths going to $v_i$.
Let
\begin{equation}\label{Ed}E_F:=\bigcup_{i, j=0}^\textsf{m} E_{i, j}.
\end{equation}

Notice that the graph $(V_F,E_F)$, depending on $\Gamma$ (or $D$), is not unique. For simple examples and more detail, we refer to \cite{HLR}.
In the more general setting of directed graphs, we have an analogue of the open set
condition. We say that $\{ f_e: \ e \in E \}$ satisfies the \emph{( graph) open set condition} if
there exist bounded non-empty open sets $O_0, O_1, . . . , O_\textsf{m}$ such that
$$  \bigcup_{j =0}^\textsf{m} \ \bigcup _{e \in E_{i, j}} f_e(O_j)\subseteq O_i,$$
where the unions are disjoint for each $i$.

Let $B=[b_{ij}]:=[\#E_{i, j}]$, $0\leq i,j\leq m$, be the adjacency
matrix of the directed graph $(V_F,E_F)$.
In order to be able to use the extension of McMullen's formula to certain graph directed sets below, we
must use a complete residue system like $D'$ in (III) rather than $D$. For that, take $d'_s\in D'$ instead of $d_s\in D$
in (\ref{OrEd}) and call the new edge set $E'_F:=\bigcup_{i, j=0}^\textsf{m} E'_{i, j}$, where  $E'_{i, j}:=\{d'_s\in D'\ : \ \Delta (\textbf{i}_i s)= \Delta(\textbf{i}_j),  \  \ 1 \leq s \leq N \}.$
When we do this, note that we keep the index $s$ same; i.e. it is obtained by using
$\Delta (\textbf{i} )$ in (\ref{del}) and $\Delta (\textbf{i} )$ uses $\Gamma (T,D)$.
The reason for this strategy is that $\Gamma= F(T,D)$ contains $F$ (as required for the deconstruction above)
although $\Gamma':=F (T,D')$ may not. Then write $\Gamma'_i$ and $E'_{i, j}$ in place of $\Gamma_i$ and $E_{i, j}$ in (\ref{eqn9}). This yields $B=[b_{ij}]=[\#E'_{i, j}]$, and
similar to Proposition 3.3 in \cite{HLR}, one can prove the following deconstruction result.

\begin{prop}\label{fractal2} Let $F=F(T,A)$ be an integral self-affine set,  $E'_F:=\bigcup_{i, j=0}^\textsf{m} E'_{i, j}$, and let $\Gamma'_i$ be the graph directed sets
corresponding to the graph $(V_F,E'_F)$
Then

(i) the graph directed IFS $\{ f_{e'}: \ e' \in E'_F \}$  satisfies the graph
OSC,

(ii) $\{T^{-k}d'_{\textbf{i}} : \ i_1 = i, \ {\textbf{i}}\in \mathcal{S}_k\}\rightarrow \Gamma'_i$ as $k\rightarrow \infty ,$

(iii)  $ {dim}_H(F')=
{dim}_H F, $ where $F':= \Gamma'_1\cup \cdots \cup \Gamma'_l\subseteq \Gamma'
$ and
 $l=\#\mathcal{S}_1$.
\end{prop}

\noindent  We also mention that $F$ is  a union of some affine copies of $\Gamma_1, ...,  \Gamma_l$ according to \cite{HLR}, where $\Gamma_0=F$. Further, 
$$\Gamma_i=(T-I)^{-1}a+c_0\Gamma'_i \quad 1\leq i\leq l$$  since  $D =a+c_0D'$ in (III) above. Therefore,  $ {dim}_H(F')={dim}_H F.$

\subsection{The Box Dimension of $F_k$
}\label{McMullen3}

Assume that $T$ is a diagonal matrix with
\begin{equation}\label{diagonal_general}
 T=diag(m_1,m_2,...,m_n)\in M_n({\Bbb{Z}})
 \ \ \ \textrm{and} \ \ \ {\tiny 2\leq m_1 <m_2 < \cdots < m_n}.
\end{equation} 
Let $B=[b_{ij}]:=[\#E_{i, j}]=[\#E'_{i, j}]$, $1\leq i,j\leq m$, be the \emph{adjacency}
matrix of the directed graph $G=(V_F,E'_F)$ in  Proposition \ref{fractal2}. 
Then  the edges in $E'_F$ are labeled by
$$D'=\prod_{p=1}^n \{0,1,2...,m_p-1\}$$ so that
$D=a+c_0D'$
 is our digit set $D$ \underline{for the auxiliary tile $\Gamma$} in the previous section.

For comparison of our results with \cite {KP}, we may also use the terminology of dynamical systems. Then by Proposition \ref{fractal2}, the self-affine set $F(T,A)$ can be viewed as a sofic affine invariant set
through the  graph $G=(V_F,E'_F)$. 
A graph
has a \textit{right-resolving} labelling if no two edges emanating from the same vertex have the same label. 
This is explicit in our setting. 
The properties (I)-(III) in the previous section already guarantee that the graph $G$  is right-resolving. That is, $G$ has already a right-resolving labelling 
$L=id: E'_F \rightarrow D'$.

Besides the adjacency matrix $B$, we now define $n-1$ more matrices. 
Denote the
points of $\mathbb{R}^n$ by $(x_1,x_2,...,x_{n})^t$.
 For $p=1,2,...,n-1$, let
$\pi_p(F)$ be the projection of $F$ onto $\mathbb{R}^p$; i.e, the set of points $(x_1,x_2,...,x_{p})^t$  with $(x_1,x_2,...,x_{n})^t\in F$. 
Let $B_p$ denote the adjacency matrix of the right-resolving graph  for $\pi_p(F)$.
Note that the logarithms of the spectral radii of $B_p $
correspond to the topological entropies of the maps $$(x_1,x_2,...,x_{p})^t \rightarrow ( m_1x_1, m_2x_2,..., m_px_{p})^t \ {\rm mod \ 1} $$ on the compact sets $\pi_p(F)$, $p=1,2,...,n-1$. Set $B_n=B$. Then the following are some extensions of dimension formulas of McMullen \cite{M}, and Kenyon $\&$ Peres \cite {KP}.

\begin{prop}\label{graph-dir_general3} Assume that $F(T,A)$ is an integral self-affine set and $T\in M_n({\Bbb{Z}})$ is as in (\ref{diagonal_general}). 
Let $A\subset {\Bbb{Z}}^n$ be arbitrary.
Suppose that the adjacency matrix $B$ and the matrices $B_p$  ($p=1,2,...,n$) are defined  as above. 
If $B$ is a primitive matrix, then 
\begin{equation}\label{KenyonPeresBox2}
dim_B F(T,A) =\sum_{p=1}^n \log_{m_p} \left( \frac{\lambda_p}{\lambda_{p-1}} \right),
\end{equation}
where $\lambda_0=1$ and the $\lambda_p$ are the spectral radii of the matrices $B_p$.
\end{prop}

\begin{pf}
We
give a proof
for completeness. It will essentially be a modification of the proof of
Proposition 3.5 in \cite{KP}.
Recall that $F$ can be obtained from the directed graph $G=(V_F,E'_F)$ and a
right-resolving labelling $L: E'_F \rightarrow D'=\prod_{p=1}^n \{0,1,2...,m_p-1\}$. By Proposition \ref{fractal2}-(ii), we may assume that every point of $F$ has a series expansion governed by the graph $G$ and the digits of the expansion are in 
$L(E'_F)\subseteq D'$. More explicitly, we may use $F'$ in Proposition \ref{fractal2}-(iii) instead of $F.$
As in Definition \ref{piece}-(ii),(iii), for a given positive integer $l=l_1$, inductively define the integers $$l_{p+1}=\lfloor l_p\log_{m_{p+1}} m_p \rfloor,$$ where $p=1,2,..,n-1$ and
let $P_l(F')$ denote a level-$l$ piece obtained by $G$.  More specifically, $p$-th block of  $P_l(F')$ will be in the form of
\begin{equation}\label{piece_form2}
 (P_l(F'))_{p}=\sum_{i=1}^{l_{p}} T_{p}^{-i} (d'_{j_i})_{p}  + T_{p}^{-l_{p}} (F')_{p},  \quad  \quad 
 \ p=1,2,...,n,
\end{equation}
as in  (\ref{piece_form}), but the digits $d'_{j_i}$ are in $L(E'_F)\subseteq D'$. 

According to Lemma \ref{box_piece2}, counting mesh cubes intersecting $F'$ or pieces of $F'$ makes no difference in the computation of the box dimension. The situation is the same for sofic integral affine-invariant   sets (i.e. $T$ and labels or digits are integer matrices in our context). Since $T$ is a diagonal matrix,  $P_l(F')$ is an approximate n-cube of diameter $\approx m_1^{-l}$. Therefore, we can work with the pieces $P_l(F')$ obtained from the graph $G$. 

Now we describe a piece counting process with respect to the graph $G$ and  give lower and upper bounds for the number of pieces.   What we will do is basically to count the digits $(d'_{j_i})_{p}$ (or their indices)  in the expansions of $(P_l(F'))_{p}$, which identify the piece $P_l(F')$.
For convenience, we  use the coordinate functions $ \mathrm{proj}_{p} :  E'_F \rightarrow \{0,1,2...,m_p-1\}$ ($p=1,2,..,n$ here) given by
$$(\mathrm{proj}_{1}(e),...,\mathrm{proj}_{n}(e))^t:=L(e).$$  \textrm{Let} \ \ \ \ \ \ \ \ \ \ \ \   $$L_p:=(\mathrm{proj}_{1},...,\mathrm{proj}_{p})^t $$
so that $L=L_n$. Thus $L_p=(\mathrm{proj}_{1},...,\mathrm{proj}_{p})^t $ is the projection of the range of the labelling function $L$ to the first $p$ coordinates. 
Remember that $l_1\geq l_2\geq \cdots \geq l_n$ by the definition of $l_p.$
Then, to every path $e_1, e_2,..., e_l$ in $G$, we attach the sequence $L(e_1),...,L(e_l)$, 
and let
$N(l)$ denote the number of such sequences.
Similarly, let $N_{p}(l)$ denote the number of distinct sequences of the form 
$$L_{p}(e_1),...,L_{p}(e_l)\in \mathbb{Z}^p$$
and $M(l)$ denote the number of distinct sequences of the form
\begin{align*}
     & L_n(e_1),L_n(e_2),...,L_n(e_{l_n}),\\
     & L_{n-1}(e_{_{l_n+1}}),...,L_{{n-1}}(e_{l_{n-1}}),\\
     & L_{n-2}(e_{_{l_{n-1}+1}}),...,L_{n-2}(e_{l_{n-2}}),\\
    \noalign{\centering$\vdots$}
     & L_{1}(e_{_{l_{2}+1}}),...,L_{1}(e_{l_{1}})=L_{1}(e_{l})
\end{align*}
respectively. Thus,
  $M(l)$ is the number of the approximate n-cubes $P_l(F')$ 
covering $F'$.
Let $h_{top}$ denote the topological entropy.
Then, as explained in \cite{KP},
\begin{equation}\label{top_entropy}
 h_{top}(\pi_p(F')) = \underset{l\rightarrow \infty}{\lim} \frac{\log N_{p}(l)}{l}
= \log \lambda_p,
\end{equation}
\begin{equation}\label{box_dimension}
dim_B F=dim_B F'= \underset{l\rightarrow \infty}{\lim} \frac{\log M(l)}{l \log m_1} \  
\end{equation}
(if the last limit exists).
Since $B$ is primitive, every power of $B^\textsf{r}$ is a positive integer matrix for some $\textsf{r} > 1$. For large $l=l_1$ (so $l_p$ gets large too), there exists a nonnegative integer $\textsf{r}_{p+1}\leq \textsf{r}$  such that
$l_{p+1}+\textsf{r}_{p+1}$ is a multiple of $\textsf{r}$  for each $p=1,2,...,n-1$ and $l_{p+1}+\textsf{r}_{p+1}< l_{p}$. Here we should perhaps note that the assumption $m_p< m_{p+1}$ (\ref{diagonal_general})  leads to the limit 
$$\lim_{l_p\rightarrow \infty} l_p\left(1-\frac{\log m_p}{\log m_{p+1}}\right)=\lim_{l\rightarrow \infty} l_p\left(1-\frac{\log m_p}{\log m_{p+1}}\right) \leq \lim_{l\rightarrow \infty} 
(l_p -l_{p+1}) \Longrightarrow \lim_{l\rightarrow \infty} 
(l_p -l_{p+1})=\infty.$$  Then one can see that
$$N(l_n- \textsf{r})\prod_{p=1}^{n-1} N_{p}(l_{p}-(l_{p+1}+\textsf{r}_{p+1}))\leq M(l) \leq N(l_n) \prod_{p=1}^{n-1} N_{p}(l_{p}-l_{p+1}).$$
This gives
\begin{equation}\label{inequalities}
\footnotesize{\frac{\log N(l_n- \textsf{r})}{\log m_1^l} \  + \sum_{p=1}^{n-1} \frac{\log N_{p}(l_{p}-l_{p+1}-\textsf{r}_{p+1})}{\log m_1^l} \  \leq \frac{\log M(l)}{\log m_1^l}  \leq \frac{\log N(l_n)}{\log m_1^l} \ +\sum_{p=1}^{n-1} \frac{\log N_{p}(l_{p}-l_{p+1})}{\log m_1^l}.}
\end{equation}
By  the  definition  of   $l_p,$ we have   $l_1=l,  \ m_p^{l_p} \approx m_{p+1}^{l_p+1}$  so  that  $m_1^l \approx m_p^{l_p}$ (see the proof of Lemma \ref{inequality_pieces}).  Now if we use the elementary limit
 $$\lim_{l\rightarrow \infty} \frac{l_p}{ \log m_1^l}=\frac{1}{\log m_p},
$$
 we get
\begin{eqnarray*}\label{telescope} \lim_{l\rightarrow \infty}  \frac{\log N_{p}(l_{p}-l_{p+1}-\textsf{r}_{p+1})}{\log m_1^l} & = & \lim_{l\rightarrow \infty}  \frac{\log N_{p}(l_{p}-l_{p+1})}{\log m_1^l} \\
& = & \lim_{l\rightarrow \infty}\frac{l_{p}-l_{p+1}}{\log m_1^l}\cdot \frac{\log (N_{p}(l_{p}-l_{p+1}))}{l_{p}-l_{p+1}} \\
& = & \left( \frac{1}{\log m_{p}}-\frac{1}{\log m_{p+1}} \right) \log(\lambda_p)
\end{eqnarray*}
by (\ref{top_entropy}).
Similarly, $$\lim_{l\rightarrow \infty}  \frac{\log N(l_n- \textsf{r})}{\log m_1^l}=\lim_{l\rightarrow \infty}  \frac{\log N(l_n)}{\log m_1^l}=\lim_{l\rightarrow \infty}\frac{l_{n}}{\log m_1^l}\cdot \frac{\log N(l_n)}{l_{n}}=\frac{1}{\log m_{n}}\log(\lambda_n).$$
Then \footnotesize{\begin{eqnarray*}
       \log_{m_n} (\lambda_n)+ \sum_{p=1}^{n-1} \left( \frac{1}{\log m_{p}}-\frac{1}{\log m_{p+1}} \right) \log(\lambda_p) 
        &=& \log_{m_n} (\lambda_n)+ \left( \frac{1}{\log m_{n-1}}-\frac{1}{\log m_{n}} \right) \log(\lambda_{n-1})+ \\
        & &  \left( \frac{1}{\log m_{n-2}}-\frac{1}{\log m_{n-1}} \right) \log(\lambda_{n-2}) +\cdots  \\
        & &  \left( \frac{1}{\log m_{2}}-\frac{1}{\log m_{3}} \right) \log(\lambda_{2}) + \left( \frac{1}{\log m_{1}}-\frac{1}{\log m_{2}} \right) \log(\lambda_{1}) \\
        &= & \frac{1}{\log m_{n}}(\log(\lambda_n)-\log(\lambda_{n-1}))+\frac{1}{\log m_{n-1}}(\log(\lambda_{n-1})-\log(\lambda_{n-2}))   \\
        &  &  +\cdots \frac{1}{\log m_{2}}(\log(\lambda_{2})-\log(\lambda_{1}))  + \frac{1}{\log m_{1}}(\log(\lambda_{1})-\log 1).     
     \end{eqnarray*}}   \normalsize
Therefore, by (\ref{box_dimension}) and (\ref{inequalities}),
$$\sum_{p=1}^n \log_{m_p} \left( \frac{\lambda_p}{\lambda_{p-1}} \right)=  \log_{m_n} (\lambda_n)+ \sum_{p=1}^{n-1} \left( \frac{1}{\log m_{p}}-\frac{1}{\log m_{p+1}} \right) \log(\lambda_p) \leq dim_B F \leq \sum_{p=1}^n \log_{m_p} \left( \frac{\lambda_p}{\lambda_{p-1}} \right).
$$
\end{pf}

\begin{remark}\label{fixed modulus} {\rm  More generally, we can replace $F(T,A)$ by a $T$-invariant sofic set (see \cite{KP} for the definition) in Proposition \ref{graph-dir_general3} and $T$   by an expanding block diagonal integer matrix with each diagonal block has eigenvalues of fixed modulus like $T_k$. 
For $p\in \{1,2,...,s\}$, let $T_{k,p}$ be a diagonal block of
$T_k$ similar to (\ref{Jordan1}) and the corresponding modulus $m_p$ of its eigenvalue (with algebraic multiplicity $n_p$) be an integer. Then we use the digit set $D_{k,p}=\{0,1,2...,m_p-1\}^{n_p}$ for $T_{k,p}$.  Otherwise, $T_{k,p}$ will be   similar to (\ref{Jordan2}) with
$C_p=\begin{bmatrix}
  a & -b\\
b & a
\end{bmatrix}$ and we may use the digit sets from  \cite[Proposition 3.2]{K1} or \cite[Theorem 1.1]{K} . More precisely, $D_{k,p}=(P \cap {\Bbb{Z}}^2)^{n_p}$ will correspond to $T_{k,p}$, where $P$ is the half-open square spanned by
$v_1= \tiny \begin{bmatrix}
  a \\
b
\end{bmatrix}$
and
$v_2= \tiny \begin{bmatrix}
  b \\
-a
\end{bmatrix}$; that is $P=\{t_1v_1+t_2v_2 : 0 \leq t_1,t_2 <1 \}$. So the Cartesian product $\prod_{p=1}^s D_{k,p}$ will replace the label set $\prod_{p=1}^n \{0,1,2...,m_p-1\}$ in the previous proof and the proof will essentially remain the same. 

In the setup of this paper, we also mention that we may always assume $1<r_1<r_2<...<r_s$  by considering the equality $F=F(T^2,A+TA)$ so that 
we will get $2\leq m_1<m_2<...<m_s$ for $T_k$ with large $k$. That is the case in  (\ref{diagonal_general}). In this way, the proofs in this section will apply to $T_k$.} $\hfill\Box$
\end{remark}

Let $s_n=q=\#A$, $s_0=1$ and $s_p$ be the number of distinct $(x_1,x_2,...,x_{p})^t\in \Bbb{Z}^p$ with $a=(x_1,x_2,...,x_{n})^t\in A.$
Obviously, the corresponding result for the Sierpi$\rm{\acute{n}}$ski sponges \cite{KP1} can be retrieved from Proposition \ref{graph-dir_general3} because in this
case, $B>0$ and $A$ is a subset of the complete residue system $\prod_{p=1}^n \{0,1,2...,m_p-1\}$ so that the primitivity and the right-resolving graph assumptions hold. That is, we readily have the following corollary.

\begin{coro}\label{graph-dir_general4} (cf. \cite[Proposition 1.3]{KP1} ) Assume that $F(T,A)$ is an integral self-affine set and $T\in M_n({\Bbb{Z}})$ is as in (\ref{diagonal_general}). Let $A$ be a subset of the complete residue system $\prod_{p=1}^n \{0,1,2...,m_p-1\}$ for $T$. Then
$$
dim_B F(T,A) =\sum_{p=1}^n \log_{m_p} \left( \frac{s_p}{s_{p-1}} \right).
$$
\end{coro}

\bigskip

Let
$F_k=F(T_k,D_k)$ be the perturbed  fractals as before. Let $B=[b_{ij}]:=[\#E_{i, j}]$, $1\leq i,j\leq m$, be the \emph{adjacency}
matrix of the directed graph $(V_{F_k},E_{F_k})$ described in Section \ref{Decon}.

\begin{prop}\label{Existence}
$dim_B (F_k)$ exists.
\end{prop}


\begin{pf} It is well-known \cite[pp.74-76]{G} that
there is a permutation matrix $P$ so that $P^{-1}BP$ is in the (Frobenius) normal form
$$
P^{-1}BP=\left[
\begin{array}{cccc}
  C_1 &\cdots    &  & \ast \\
   &  C_2  &  &\\
  &  &  \ddots  & \vdots \\
  {0} &    &  & C_l \\
\end{array}
\right],
$$
where each diagonal block $C_i $ ($i=1,...,l$) is either an irreducible matrix or the $1\times 1$ zero matrix.  Remember from the previous subsection that we discard
the vertices with no outgoing edges. Thus we don't have any zero matrices among the $C_i$.
So $C_1,..., C_l$ are, in fact, the irreducible components of $B$ in our context.
Furthermore, from the basic Perron-Frobenius theory of nonnegative matrices \cite[Corollary 2, p.82]{G}, \cite[Theorem 1.4, pp.21-22]{S},  we know that for each $C_i$, there is a permutation matrix $P_i$ such that $P_i^{-1}C_i^{p_i}P_i$ is in the block diagonal form
$$
P_i^{-1}C_i^{p_i}P_i=\left[
\begin{array}{cccc}
  M_1 &\cdots    &  & 0 \\
   &  M_2  &  &\\
  &  &  \ddots  & \vdots \\
  {0} &    &  & M_{l_i} \\
\end{array}
\right]=diag ( M_1,M_2,...,M_{l_i}),
$$
where $p_i$ is the period of $C_i$ and each diagonal block 
is a primitive matrix.
Now define the permutation matrix $Q=P_1\oplus P_2 \oplus \cdots \oplus P_l$. Let $p$ be the least common multiple of the periods $p_i$.
Then $(PQ)^{-1}B^p(PQ)$ will be in the form
$$
(PQ)^{-1}B^p(PQ)=\left[
\begin{array}{cccc}
  K_1 &\cdots    &  & \ast \\
   &  K_2  &  &\\
  &  &  \ddots  & \vdots \\
  {0} &    &  & K_t \\
\end{array}
\right],
$$
so that $K_1, K_2,...,K_t $ are primitive matrices.

Let
$\Gamma_{1},...,\Gamma_{t}$ be
 the graph-directed sets corresponding the subgraphs of $(V_{F_k},E_{F_k})$ specified by the primitive matrices $K_{1},...,K_{t}$.
Thus
$F_k$ is a
union of some affine copies of
$\Gamma_{1},...,\Gamma_{t}$ as in Proposition \ref{fractal2}-(iii).
Then it follows from the finite stability of $\overline{dim}_B$  that
$$\overline{dim}_B(F_k)=\max \{ \overline{dim}_B(\Gamma_{1}),...,\overline{dim}_B(\Gamma_{t}) \}.$$

Therefore, $\overline{dim}_B(F_k)=\overline{dim}_B(\Gamma_{t_0})$ for
some $t_0\in\{1,...,t\}$.
$K_{t_0}$ is primitive and its graph has a right-resolving labelling by (I)-(III) in Section \ref{Decon}. 
Then $\overline{dim}_B(\Gamma_{t_0})=\underline{dim}_B(\Gamma_{t_0})$ by  Remark \ref{fixed modulus}.
Hence $$\overline{dim}_B(\Gamma_{t_0})=\underline{dim}_B(\Gamma_{t_0})\leq \underline{dim}_B (F_k)\leq \overline{dim}_B(F_k)=\overline{dim}_B(\Gamma_{t_0})$$
implies that $\underline{dim}_B(F_k)=\overline{dim}_B(F_k)=\overline{dim}_B(\Gamma_{t_0}),$ which proves our assertion.
\end{pf}

\begin{remark}
 {\rm Proposition  \ref{Existence} is consistent with previous results in the literature. Namely, the box dimension exists if $T$ is diagonal  and $A$ has some special form (cf. \cite{KP}, \cite{KP1}), \cite{M}),   or
  if $T$  itself is a nontrivial Jordan block, i.e all eigenvalues are equal in modulus. The last assertion follows from Theorem 1.1 and Theorem 1.2 in \cite{HL2}.} $\hfill \Box$
\end{remark}

\subsection{The Hausdorff Dimension of $F_k$}
We point out that, by definition, the Sierpi$\rm{\acute{n}}$ski sponges have special digit sets as in Corollary \ref{graph-dir_general4} while the perturbed fractals
$F_k$ have arbitrary integer digit sets $D_k$. For calculation, one  also needs the dimension formulas from \cite{K21} mentioned \underline{for the more general  matrices $T_k$ and digit sets $D_k$ generating}
\noindent
\underline{the perturbed fractals $F_k$}.

Let $B=[b_{ij}]=[\#E'_{i, j}]$,  $1\leq i,j\leq m$,
be the \emph{adjacency} matrix of the directed graph $(V_F,E_F)$ and $D$, $[\#E'_{i, j}]$ be as in the Section \ref{Decon}.
Below we use the superscript $t$ for the transpose of a matrix.
In the following, we implicitly decompose $B$ into
 $m_1m_2\cdots m_{n-1}$ matrices whose sum is $B$.  Let $\mathbf{i}_r=(i_1,...,i_r)^t \in \prod_{i=1}^r \{0,1,2...,m_i-1\}$, $1\leq r \leq n $. Define the matrices  $B(\mathbf{i}_n)$ depending on $\mathbf{i}_n$ as follows:  The $ij$-th entry of $B(\mathbf{i}_n)$ is the number of the digits $d_s=(x_1,x_2,...,x_n)^t\in  E'_{i, j}\subseteq D=\prod_{i=1}^n \{0,1,2...,m_i-1\}$ with $x_s=i_s$,  \  \ $1\leq s\leq n$.
Notice that some of the $B(\mathbf{i}_n)$ defined in this way may be zero matrices.
For $1\leq r \leq n $, define the partition functions $Z^l_{r-1}$ depending on $\mathbf{i}_{r-1}$ as follows:

First let $r=n.$ Then
$$Z^l_{n-1}(\mathbf{i}_{n-1}):=\mathlarger{\sum}_{i_n=0}^{m_n-1} \left|\left| \prod_{k=1}^{l} B_{i_k}(\mathbf{i}_{n-1},i_n) \right|\right|, $$
where $B_{i_k}(\mathbf{i}_n)$ is a matrix in the form of $B(\mathbf{i}_n)$ described above and the norm $|| \cdot ||$  is the 
sum of the absolute values of the elements.
More generally, for $n-1\geq r\geq 2$
$$Z^l_{r-1}(\mathbf{i}_{r-1}):=\sum_{i_r=0}^{m_r-1} \left(Z^l_{r}(\mathbf{i}_{r-1},i_r) \right)^{ \frac{\log m_{r}}{\log m_{r+1}} }$$
and for
$r=1$, $Z^l_{0}:=\sum_{i_1=0}^{m_1-1} \left(Z^l_{1}(\mathbf{i}_{1}=i_1)\right)^{ \frac{\log m_{1}}{\log m_{2}}}$  only depends on $l$.
Define a sequence $ \{ u_l \}$ by
\begin{equation}\label{u_l}
u_l = \frac{1}{l}\log_{m_1} Z^l_{0}.
\end{equation}

\begin{prop}\label{graph-dir_general} \cite{K21} Assume that $F(T,A)$ is an integral self-affine set and $T\in M_n({\Bbb{Z}})$ is as in (\ref{diagonal_general}). Then

(i)
\begin{equation}\label{KenyonPeres} {dim}_H F(T,A) = \lim_{l\rightarrow \infty}u_l\end{equation}

(ii) $\left\{  u_{l 2^m} \right\}_{m=0}^{\infty}$ is non-increasing for any fixed $l$. That is,
 $  u_{l 2^m}\downarrow {dim}_H F(T,A) $ 
as $m\rightarrow\infty$.

\noindent More generally, $T$ can be replaced by $T_k$. 
\end{prop}

\bigskip
\begin{remark} \rm{
For $n=2$, the proof of (i) was given by Kenyon and Peres \cite[Theorem 3.2]{KP}. The main ingredient there is dimension formula for Sierpi$\rm{\acute{n}}$ski carpets \cite[Theorem, p.1]{M}. That proof works for the general case in Proposition \ref{graph-dir_general}-(i) if  the dimension formula for Sierpi$\rm{\acute{n}}$ski carpets is replaced by the formula for Sierpi$\rm{\acute{n}}$ski sponges \cite[Theorem 1.2]{KP1}.  For Jordan block matrices, Proposition \ref{graph-dir_general}-(i) is consistent with Theorem 1.4 (p. 1143) and Corollary 3.2 (p.1147) in
\cite{HLR}. As detailed in Remark \ref{fixed modulus},   $T$ can be replaced  by $T_k$
(see Lemma 4.6 and the proof of Theorem 1.1 in \cite{KP1}).
Moreover, (ii) follows from the sub-multiplicativity of $|| \prod_{k=1}^{l} B_{i_k}(\mathbf{i}_n) ||$
 in $l$ exactly as in the planar case, see \cite[Proposition 5.10-(i)]{K1} for the case $n=2$.} \hfill $\Box$
\end{remark}

\end{document}